\documentclass[11pt]{article}
\usepackage{amsmath}
\usepackage{amssymb}
\usepackage{graphicx}
\usepackage{amsthm}
\usepackage{verbatim}
\usepackage{color}
\usepackage{enumerate}
\usepackage{amsfonts}
\usepackage[all]{xy}
\usepackage{fancyhdr}
\usepackage{enumerate}
\theoremstyle{plain}

\newtheorem{theorem}{Theorem}[subsection]
\newtheorem{lemma}[theorem]{Lemma}
\newtheorem{proposition}[theorem]{Proposition}
\newtheorem{corollary}[theorem]{Corollary}

\theoremstyle{remark}
\newtheorem{remark}[theorem]{Remark}

\theoremstyle{definition}
\newtheorem{definition}[theorem]{Definition}
\newtheorem{conjecture}[theorem]{Conjecture}
\newtheorem{example}[theorem]{Example}

\newtheorem{convention}[theorem]{Convention}
\newtheorem{question}[theorem]{Question}

\usepackage{xcolor}
\definecolor{linkcol}{rgb}{0.10,0.20,0.55}
\definecolor{citecol}{rgb}{0.05,0.40,0.15}
\usepackage[colorlinks,linktoc=all,linkcolor=linkcol,citecolor=citecol,urlcolor=linkcol,bookmarksnumbered]{hyperref}
\begin{document}
\title{Geometric chains of orbifolds with corners and singular homology}

\author{Hao Yu}



\maketitle

\begin{abstract}
We study a geometric homology theory (GHT for short). Its chains are generated by smooth maps from compact oriented orbifolds with corners to a smooth space, and two kinds of chains are set to zero: those admitting an orientation reversing automorphism, and degenerate ones, which, together with their boundaries, factor through smooth maps to orbifolds of lower dimension. Bisecting the source of a geometric chain along a hypersurface transverse to its corners changes the chain only by the boundary of an explicit prism and by prisms over its faces; cutting a geometric cycle along a regular level set of a smooth function on the target does not change its homology class. Geometric chains also behave well under products and pullbacks along fibre bundles, which overcomes some of the combinatorial rigidity of singular chains. We prove that geometric homology is naturally isomorphic to singular homology with rational coefficients for every smooth space admitting a good cover, such as smooth manifolds, orbifolds and locally finite polyhedra, and also for orbifolds with corners. The proof uses cutting operators along level sets of functions on the target, realized by explicit prisms, instead of triangulations. For continuous maps into arbitrary spaces we conjecture the same statement, and prove it for Euclidean neighbourhood retracts in degrees at most one, with surjectivity in degree two. Without the relation killing degenerate chains the comparison fails for orbifold chains: the weighted projective plane $\mathbb{P}(1,1,p)$, $p\ge5$ prime, mapped to a point is a cycle that does not bound. As applications we give a chain map from the Morse complex of a Morse--Smale pair to geometric chains, depending only on the pair and on orientations of the descending manifolds, which is a quasi-isomorphism on closed oriented manifolds, and discuss the BV structure on geometric chains of moduli spaces of bordered Riemann surfaces, following Costello and Harrelson--Voronov--Z\'u\~niga.
\end{abstract}
\medskip
\tableofcontents

\section{Introduction}
In this paper, we introduce a new homology theory, which we call \textit{geometric homology theory}, or GHT for brevity. This theory is based on the notion of geometric chains, first introduced by Harrelson, Voronov and Z\'u\~niga (\cite{vhz}). GHT has several nice properties.
\begin{itemize}
\item For every smooth space admitting a good cover, the homology of smooth geometric chains is isomorphic to rational singular homology (Theorem \ref{main_thm}). This includes smooth manifolds, orbifolds and polyhedra, and also orbifolds with corners. We conjecture that the same holds for continuous geometric chains in arbitrary spaces (Conjecture \ref{conj_continuous}).
\item Chains in GHT behave well under natural operations such as Cartesian products and pullbacks along fibre bundles. This is in sharp contrast with singular chains, for which the product of two simplices is not a simplex and has to be subdivided, and pullbacks along fibre bundles have no canonical definition.
\item Every compact oriented manifold or orbifold with corners has a canonical fundamental chain in GHT, whereas no such chains exist in singular theory.
\end{itemize}
Because of these properties, we believe that in some situations it is useful to work with geometric chains in place of singular chains.

We first define geometric chains (see Sections \ref{sec_notation} and \ref{sec_orbifolds} for the relevant notions and Section \ref{sec_chains} for details). A \textit{smooth space} is a topological space $X$ together with a sheaf of continuous functions, called smooth. The sheaf contains the constants, is closed under composition with smooth functions of several variables, and admits smooth Urysohn functions (Definition \ref{def_smooth_space}). A continuous map $f$ from (the underlying space of) an orbifold with corners to $X$ is \textit{smooth} if $\varphi\circ f$ is smooth for every local smooth function $\varphi$ on $X$.
\begin{definition}\label{main_def}
(geometric chains).
Let $X$ be a smooth space.
\begin{itemize}
\item A \textit{geometric simplex} of dimension $n$ in $X$ is a pair $(Q,f)$, where $Q$ is a compact oriented smooth orbifold with corners of dimension $n$ (possibly disconnected) and $f:|Q|\to X$ is a smooth map.
\item Two geometric simplices $(Q,f)$ and $(Q',f')$ are \textit{isomorphic} if there is an orientation preserving diffeomorphism $h:Q'\to Q$ of orbifolds with corners with $f'=f\circ h$.
\item A geometric simplex $(Q,f)$ is \textit{thin} if $f=g\circ\pi$, where $\pi:Q\to N$ is a smooth map to an orbifold $N$ (possibly with corners, not necessarily compact) with $\dim N<n$, and $g:|N|\to X$ is smooth.
\item A linear combination of thin geometric simplices whose boundary is again a linear combination of thin geometric simplices is called \textit{degenerate}.
\end{itemize}
The space $GC^\infty_n(X;\mathbb{Q})$ of \textit{geometric chains} is the $\mathbb{Q}$-vector space spanned by the isomorphism classes of $n$-dimensional geometric simplices, modulo the relations
\begin{itemize}
\item $(-Q, f) = -(Q, f)$, where $-Q$ is $Q$ with the reversed orientation;
\item $(Q_1 \sqcup Q_2, f_1 \sqcup f_2) = (Q_1, f_1) + (Q_2, f_2)$;
\item degenerate chains are zero.
\end{itemize}
The boundary of $(Q,f)$ is $(\partial Q, f\circ i_Q)$, where $i_Q:\partial Q\to Q$ is the boundary of $Q$ (Convention \ref{conv_faces}) with the induced orientation. The homology of $GC^\infty_*(X;\mathbb{Q})$ is denoted by $GH^\infty_*(X;\mathbb{Q})$. See Definitions \ref{def_degenerate} and \ref{def_smooth_GC} for the precise construction.
\end{definition}
All orbifolds in this paper are effective, except for the orbifolds in the wide sense of Definition \ref{def_orb_chains}(b); non-effective sources are reduced to effective ones by a convention stated after Definition \ref{orbifold}.
\begin{remark}
We use coefficients $\mathbb{Q}$. In the proof of Theorem \ref{main_thm} they enter, as far as we can see, only through Lemma \ref{lem_Cg}; see Remark \ref{rem_stratified} for integral coefficients. They are needed for the weights $1/|K|$ of non-effective sources (after Definition \ref{orbifold}), and are used in Proposition \ref{prop_MVa_cont} and Remark \ref{rem_twisted}; see also the relevant paragraph in \cite{vhz}.
\end{remark}

Let us comment on the relations. By the first relation, a geometric simplex admitting an orientation reversing automorphism is zero (Lemma \ref{lem_Cg}).
\begin{itemize}
\item Without the third relation the comparison with singular homology fails, already for a point. Closed manifolds of positive dimension mapped to a point do bound without it, but closed orbifolds such as the weighted projective plane $\mathbb{P}(1,1,p)$, $p\ge5$ prime, define cycles that do not (Remark \ref{rem_role_of_D}). Such cycles are thin with empty boundary, hence degenerate, hence zero. Thin simplices play the role of degenerate simplices in simplicial theory.
\item In a previous version of this paper we imposed, instead, a \textit{gluing relation}: a geometric simplex was declared equal to the sum of the pieces of any decomposition of it. We no longer impose this relation. Cutting a geometric simplex along a hypersurface transverse to its corners changes the chain, but the change is realized by an explicit prism (Proposition \ref{prop_bisection}), just as barycentric subdivision is chain homotopic to the identity on singular chains.
\item The gluing relation, being an equality in a quotient, made the comparison with singular homology delicate: cancellations could take place through relations involving chains outside a given open set. With the present definition the relations are local (Lemma \ref{lem_D_local}).
\end{itemize}

A smooth singular simplex $s:\Delta^n\to X$ is a geometric simplex $(\Delta^n,s)$, and
\begin{equation*}
p: C^\infty_*(X;\mathbb{Q}) \rightarrow GC^\infty_*(X;\mathbb{Q}),\qquad s\mapsto(\Delta^n,s),
\end{equation*}
is a chain map (Section \ref{sec_comparison}). A \textit{good cover} of $X$ is an open cover such that every non-empty finite intersection of its members is a disjoint union of open, smoothly contractible subsets. If $X$ admits a good cover, smooth singular chains compute singular homology (Lemma \ref{lem_smooth_singular}), so $p$ induces a map $p_*:H_*(X;\mathbb{Q})\to GH^\infty_*(X;\mathbb{Q})$. We will prove the following theorem.
\begin{theorem}\label{main_thm}
(geometric homology is isomorphic to singular homology).
Let $X$ be a smooth space admitting a good cover. Then
\begin{equation}
p_*: H_*(X;\mathbb{Q}) \xrightarrow{\ \cong\ } GH^\infty_*(X;\mathbb{Q})
\end{equation}
is an isomorphism, natural with respect to smooth maps. This applies in particular to smooth manifolds, to locally finite polyhedra in $\mathbb{R}^N$, and to orbifolds (Section \ref{sec_examples}). The same conclusion holds for orbifolds with corners (Corollary \ref{cor_corners_general}), and also for chains given by liftable maps or by orbifold morphisms (Proposition \ref{prop_orb_comparison}).
\end{theorem}

Relative groups $GH^\infty_*(X,A;\mathbb{Q})$ are defined for an arbitrary subset $A\subseteq X$, using the chains with image in $A$. Excision holds without any assumption on $X$, and the comparison theorem extends to admissible pairs (Definition \ref{def_admissible_pair}), such as open subsets of manifolds and orbifolds, closed submanifolds of manifolds, and the boundary of a manifold with boundary (Section \ref{sec_relative}).

\textit{Continuous chains.} Definition \ref{main_def} makes sense for an arbitrary topological space $X$ if $f$, and the map $g$ in the definition of thinness, are only required to be continuous; $\pi$ is still required to be smooth. This gives the complex $GC_*(X;\mathbb{Q})$ of continuous geometric chains, with homology $GH_*(X;\mathbb{Q})$ (Section \ref{sec_chains}). Singular simplices define a chain map $p:C_*(X;\mathbb{Q})\to GC_*(X;\mathbb{Q})$. For a smooth space $X$ there is a natural chain map $GC^\infty_*(X;\mathbb{Q})\to GC_*(X;\mathbb{Q})$ compatible with $p$.
\begin{conjecture}\label{conj_continuous}
For every topological space $X$, or at least for every Euclidean neighbourhood retract, $p_*:H_*(X;\mathbb{Q})\to GH_*(X;\mathbb{Q})$ is an isomorphism.
\end{conjecture}
The proof of Theorem \ref{main_thm} uses smoothness in an essential way: chains are cut along level sets of functions pulled back from the target, so that two faces which cancel in the boundary of a cycle are cut in the same way. For continuous maps no such compatible smooth cutting need exist (Section \ref{sec_continuous}). We obtain the following partial results.
\begin{itemize}
\item The first half of the Mayer--Vietoris principle holds for continuous chains in metrizable spaces (Proposition \ref{prop_MVa_cont}).
\item For Euclidean neighbourhood retracts, the conjecture reduces to the acyclicity of a complex $\mathcal{K}_*$ measuring the locality of degenerate chains (Proposition \ref{prop_reduction}).
\item For Euclidean neighbourhood retracts, $p_*$ is an isomorphism in degrees $\le1$ and surjective in degree $2$ (Theorem \ref{thm_low_degrees}).
\item $p_*$ is an isomorphism in degree $2$ if a question in dimension theory has a positive answer for $m=1$ (Question \ref{q_dimension} and Corollary \ref{cor_degree_two}). We answer this question in several cases, for instance for piecewise monotone generators, which include the real analytic ones (Proposition \ref{prop_dim_laps}); the general case remains open.
\end{itemize}

If the sources $Q$ in Definition \ref{main_def} are required to be manifolds with corners, with the same notion of thinness ($N$ may still be an orbifold), we obtain a subcomplex $GC^{\infty,m}_*(X;\mathbb{Q})$ of $GC^\infty_*(X;\mathbb{Q})$.
\begin{corollary}\label{main_cor}
Under the hypotheses of Theorem \ref{main_thm}, $p$ also induces an isomorphism $H_*(X;\mathbb{Q})\cong GH^{\infty,m}_*(X;\mathbb{Q})$.
\end{corollary}
A related statement for manifolds with corners appears in \cite{cd}, where, however, no details are given. In \cite{vhz}, Harrelson, Voronov and Z\'u\~niga remark on this property: ``Our notion of chains leads
to a version of oriented bordism theory via passing to homology. If we impose extra
equivalence relations, such as some kind of a suspension isomorphism, as in \cite{ja},
or work with piecewise smooth geometric chains and treat them as currents, in the
spirit of \cite{fooo}, we may obtain a complex whose homology is isomorphic to the
ordinary real homology of X''. Although they did not state it as a formal conjecture, we take the liberty of regarding it as an implicit question. Remark \ref{rem_role_of_D} confirms that an extra relation is needed, although not because closed manifolds survive: closed manifolds of positive dimension mapped to a point bound even without it. From this point of view, Theorem \ref{main_thm} gives an affirmative answer, where the extra relation is the vanishing of degenerate chains.

\textit{Related work.} Several other geometric descriptions of singular homology are available.
\begin{itemize}
\item Jakob (\cite{ja98, ja}) describes the homology theory associated with a multiplicative cohomology theory $h^*$, in particular singular homology, by triples $(M,a,f)$. Here $M$ is a closed smooth manifold with an $h^*$-orientation, $a\in h^*(M)$, and $f\colon M\to X$ is a continuous map; the relations are bordism and vector bundle modification.
\item Kreck (\cite{kr}) defines stratifold homology, using regular stratifolds as cycles and bordism of $c$-stratifolds as the relation, and proves that it is isomorphic to integral singular homology for CW complexes.
\item Joyce defines M-homology (\cite{jo15}), isomorphic to ordinary homology with coefficients in any commutative ring. Its chains are built from smooth maps from manifolds with corners into manifolds or effective orbifolds, with additional data, and, together with M-cohomology, it has pushforwards, pullbacks, products and fundamental classes defined at the (co)chain level. His Kuranishi homology (\cite{jo07}) uses Kuranishi spaces with gauge-fixing data as sources, over $\mathbb{Q}$-algebras, and has no relation corresponding to degenerate chains.
\item Closest to the present paper is Lipyanskiy's chain-level theory (\cite{lip}). There, chains are smooth maps from compact oriented manifolds with corners into a manifold without boundary, with integer coefficients. Two kinds of chains are set to zero: those admitting an orientation reversing automorphism, and degenerate ones. A chain is degenerate if it has \emph{small image}, meaning that its image is contained in the image of a smooth map from a manifold of smaller dimension, and if its boundary has small image up to chains of the first kind.
\item Friedman, Medina-Mardones and Sinha (\cite{fmms}) develop Lipyanskiy's approach in detail for smooth manifolds without boundary, with integer coefficients. Their chains are smooth maps from compact oriented manifolds with corners, their cochains are proper co-oriented maps, and they set to zero trivial chains, which admit an orientation reversing automorphism, and degenerate ones, which have \emph{small rank} (the rank of the differential is everywhere smaller than the dimension) and whose boundary consists of trivial and small rank parts \cite[\S4.1, Definition 4.2]{fmms}. They prove that geometric homology and cohomology are naturally isomorphic to integral singular homology and cohomology on smooth manifolds without boundary, by verifying the axioms of Kreck and Singhof \cite[Theorem 5.42]{fmms}, and they construct explicit comparison maps from smooth singular simplices and cubes for homology \cite[\S5.3.1, Theorem 5.45]{fmms}, and an intersection map with cubical cochains for cohomology \cite[\S6]{fmms}. They treat Lipyanskiy's creasing construction in detail, with its orientations and boundary formula \cite[\S4.5.2, Lemma 4.54]{fmms}, and prove Mayer--Vietoris by creasing along level sets of separating functions pulled back from the target and regular on every corner stratum \cite[\S5.2]{fmms}. They also construct partially defined cup and cap products by transverse fibre products; the cup product agrees with the singular one \cite[Theorem 7.26]{fmms}, and the cap product is determined by the singular one \cite[Corollary 7.46]{fmms}. Poincar\'e duality and umkehr maps are treated in \cite[\S7.8]{fmms}.
\end{itemize}
Our notion of a degenerate chain, a thin chain with thin boundary, is modelled on Lipyanskiy's definition, and, as in \cite{lip} and \cite{fmms}, cutting along regular level sets of functions on the target is a basic tool. Our setting differs in the following respects.
\begin{itemize}
\item We allow orbifold sources and use rational coefficients, which are the natural choice with orbifold sources; $\mathbb{Q}$ enters the proof only through Lemma \ref{lem_Cg} (Remark \ref{rem_stratified}). For manifold sources and targets without boundary, the integral comparison for the relation of trivial and small rank chains is \cite[Theorems 5.42, 5.45]{fmms}.
\item Thinness requires a factorization $f=g\circ\pi$ of the map itself through a smooth map $\pi$ to a lower-dimensional orbifold. This is stronger than having small image; for manifold sources and targets every degenerate chain in our sense is degenerate in the sense of \cite[Definition 4.2]{fmms}, but not conversely.
\item Lipyanskiy verifies some of the Eilenberg--Steenrod axioms for manifolds and states the comparison with singular homology, referring to a Mayer--Vietoris argument; for smooth manifolds without boundary a complete proof, with integer coefficients, is given in \cite[Theorem 5.42 and \S5.3.1]{fmms}. Theorem \ref{main_thm} proves the analogous comparison, over $\mathbb{Q}$, for orbifold sources and for smooth spaces with a good cover, including orbifolds, polyhedra and orbifolds with corners. For manifold targets it is a rational analogue of \cite[Theorems 5.42, 5.45, 5.48]{fmms} for a different relation and different sources; neither result implies the other. Its proof is based on an explicit prism realizing the bisection of a geometric simplex (Section \ref{sec_bisection}), which plays the role of Lipyanskiy's crease construction as treated in \cite[\S4.5.2]{fmms}, and on cutting operators pulled back from the target (Section \ref{sec_cutting}), which parallel the creasing along separating functions in the proof of Mayer--Vietoris in \cite[\S5.2]{fmms}.
\end{itemize}

Similarly, one can define geometric cohomology groups, cup and cap products, and so on. Essentially, properties that hold in singular homology (cohomology) should still hold in GHT. This can be deduced from the method of proof of Theorem \ref{main_thm}, although, as a general principle, each basic property still has to be proved separately. For smooth manifolds without boundary, geometric cochains given by proper co-oriented maps, with cup and cap products defined by transverse fibre products and compatible with the singular products, are constructed in \cite{fmms}, with integer coefficients. We do not introduce geometric cohomology here; the companion paper \cite{yu_coh} treats geometric cochains with orbifold sources and the relation of degenerate chains, with rational coefficients, on manifolds and on orbifold targets.

As for applications, which were also a motivation for developing GHT, we mention that in \cite{cos04,cos05} Costello constructed an algebraic counterpart of mirror symmetry, in particular a purely algebraic construction of the Gromov--Witten potential. More specifically, he constructs a Batalin--Vilkovisky algebra structure on the singular chain complexes of the Deligne--Mumford compactifications of moduli spaces of Riemann surfaces with marked points, and uses fundamental chains as solutions of the quantum master equation; these encode the fundamental classes of the Deligne--Mumford spaces and hence the Gromov--Witten potentials. The BV operator involves pullbacks of chains along circle bundles, and Costello works with homotopy coinvariants of the circle actions on chains. This model is needed for other reasons as well, for instance to include $\psi$-class insertions and to act on Hochschild chains, which makes the theory explicit in terms of ribbon graphs. Geometric chains give a different chain-level model, in which pullbacks along circle bundles are defined directly and the moduli spaces have canonical fundamental chains. Using it, we extend part of these results to open-closed topological conformal field theories (see \cite{hy} and Section \ref{sec_applications}). This is also related to the motivation for geometric chains given in \cite{vhz}: ``we need the unit circle $S^1$ to have a canonical fundamental cycle and the moduli spaces which we consider to have canonical fundamental chains, irrespective of the choice of a triangulation. On the other hand, we believe it is generally better to work at the more basic level of chains rather than that of homology''. We do not claim that this model replaces the homotopy coinvariants of \cite{cos05}. Further applications are discussed in Section \ref{sec_further}.

The triangulability of compact orbifolds with corners, compatibly with their faces, which was stated in earlier versions of this paper, is treated in the separate note \cite{yu_tri}, where the triangulation is shown to be compatible with all corners; it is not used here. We briefly sketch the ideas of the proofs.

The comparison theorem does not use triangulations. Its proof follows the Eilenberg--Steenrod pattern. Homotopy invariance and the dimension axiom are elementary. The key point is a Mayer--Vietoris principle, which requires cutting a geometric chain into pieces that are small with respect to an open cover. For a smooth chain $(Q,f)$ we cut $Q$ along a regular level set of $\phi\circ f$, where $\phi$ is a function on the target. Since the cut is determined by $f$, it is well defined on isomorphism classes, and two faces which cancel in the boundary of a cycle are cut in the same way. The cut is realized by an explicit prism, whose top face is bent along the cut (Section \ref{sec_bisection}). This yields cutting operators that are chain homotopic to the identity, exactly as barycentric subdivision is in singular theory.

The rest of the paper is organized as follows.
\begin{itemize}
\item In Section \ref{sec_orbifolds}, we review orbifolds and orbifolds with corners, and linear charts at corners (Lemma \ref{lem_doubling}).
\item Section \ref{sec_main} contains the definition of the geometric chain complex, the bisection prism, the cutting operators, and the proof of the comparison theorem, followed by relative homology and excision (Section \ref{sec_relative}). It then treats examples of smooth spaces with good covers (Section \ref{sec_examples}), chains given by liftable maps and orbifold morphisms (Section \ref{sec_orb_chains}) and continuous maps (Section \ref{sec_continuous}), and ends with products and pullbacks of geometric chains (Section \ref{sec_operations}).
\item An application in categorical Gromov--Witten theory is briefly discussed in Section \ref{sec_applications}, and further applications (Morse theory, circle actions, string vertices) in Section \ref{sec_further}.
\item Appendix \ref{app_intervals} contains the proofs of the dimension-theoretic results of Section \ref{sec_continuous}, and Appendix \ref{app_fundamental} constructs the fundamental chains of orbifolds with corners used in Remark \ref{rem_thin_cycles}.
\end{itemize}

\section{Notation}\label{sec_notation}
Let $R^n$ denote the $n$-dimensional Euclidean space, and let $R^n_+:=\{(x_1,x_2,\dots,x_n) | x_i \geq 0, i=1,\dots,n\}$ be the subset of points all of whose coordinates are nonnegative. Let $R^n_k:=\{(x_1,x_2,\dots,x_n) | x_i\geq 0, i=1,\dots,k\}$ denote the subset of $R^n$ with the first $k$ coordinates nonnegative, and let $R^n_{+k}$ denote the subset of $R^n_+$ consisting of the points exactly $k$ of whose coordinates are $0$, i.e. the points of depth $k$.

Some care is needed with the definition of smooth maps between manifolds with corners. Several different definitions exist in the literature (\cite{ce, me, mon, jo}). It seems essential that a smooth map be at least weakly smooth; the definitions differ in the additional constraints they impose, in particular on the behaviour of smooth maps at the boundary and corners. We adopt the category of smooth manifolds with corners defined by Joyce \cite{jo}, and refer the reader to that paper for full details on the relevant notions, such as weakly smooth maps and submersions. We stress, however, that we \textit{never} need a general definition of smooth maps between manifolds with corners in this paper. Such maps occur only in the definitions of orbifolds with corners and of the groupoids in Definition \ref{def_orb_chains}(b), and only in very special cases, in which the definitions we use agree for \textit{any} of the notions of smooth maps mentioned above.


For a smooth manifold with corners $M$ of dimension $n$, we denote by $Bord_k({M})$ the subset of ${M}$ consisting of the points that have a neighbourhood diffeomorphic to an open set in $R^n_{k}$ by a diffeomorphism sending the point to $0$. It is called the \textit{$k$-corner} of $M$. When there is no ambiguity, we also call a connected component of it a $k$-corner. Each connected component of $Bord_k(M)$ is a smooth manifold of dimension $n-k$, called a \textit{corner stratum} of $M$. We call its closure in $M$ a \textit{$k$-face} of $M$; it is a subset of $M$ and need not be a manifold with corners (it is not the $k$-corner $C_k(M)$ of \cite[Definition~2.10]{jo}). The boundary $\partial M$ is defined in Convention \ref{conv_faces}. 
The manifold $M$ together with the subsets $Bord_k(M)$ is called the \textit{corner structure} of $M$.

\begin{convention}[boundary]\label{conv_faces}
We use the boundary of a manifold with corners in the sense of Joyce \cite[Definitions~2.6 and 2.10]{jo}. As a set, $\partial M$ consists of pairs $(x,\beta)$, where $x\in M$ and $\beta$ is a local boundary component of $M$ at $x$. It is a manifold with corners of dimension $\dim M-1$, and the map $i_M:\partial M\to M$, $(x,\beta)\mapsto x$, is smooth \cite[Theorem~3.4(iv)]{jo}. A point of depth $l$ (i.e. having a chart modelled on $R^n_l$ which sends the point to the origin) has exactly $l$ preimages. If $M$ is oriented, $\partial M$ carries the induced orientation, outward normal first \cite[Convention~7.2]{jo}. The second boundary $\partial^2M=\partial(\partial M)$ carries a free involution $\sigma$, which exchanges the two local boundary components at a point of depth two. It satisfies $i_M\circ i_{\partial M}\circ\sigma=i_M\circ i_{\partial M}$, and reverses the induced orientation if $M$ is oriented (\cite[\S7]{jo}: $S_k$ acts on the orientation of $\partial^kM$ by the sign).

For an orbifold with corners, locally $\tilde U/G$, we put $\partial(\tilde U/G):=(\partial\tilde U)/G$, where $g\in G$ acts by $(x,\beta)\mapsto(gx,g(\beta))$. Since $G$ may permute the connected components of $\partial\tilde U$, the local models are $(\tilde C,G_{\tilde C})$, where $\tilde C$ runs over these components and $G_{\tilde C}$ is the stabiliser of $\tilde C$. These local models glue to an orbifold with corners $\partial Q$ with a map $i_Q:\partial Q\to Q$. The above properties persist, except that a point of depth $l$ has at most $l$ preimages, and $\sigma$ need not be free (the local group may permute the local boundary components). The involution $\sigma$ is still orientation reversing, which is all that Lemma \ref{lem_Cg} uses. Moreover, $\partial Q$ is again effective, since a finite order linear map fixing a boundary hyperplane pointwise and preserving the half-space is the identity. The \textit{codimension one faces} of $Q$ are the connected components of $\partial Q$. We say that $Q$ has \textit{embedded faces} if $i_Q$ is injective on each of them. Simplices and their products have embedded faces. The teardrop manifold with corners \cite[Remark~2.11]{jo} (a closed disc with one corner point at which both local boundary components belong to the same face) does not, and neither do many moduli spaces of curves. We do not assume embedded faces.
\end{convention}
 When a Lie group $G$ acts continuously on a topological space $M$, its orbit space is denoted by $|M/G|$. A smooth action of a Lie group $G$ on a manifold with corners $M$ is a smooth map
\begin{equation}
f: G \times M \rightarrow M
\end{equation}
such that $f|_{G \times {Bord_k(M)}} \subseteq Bord_k{M}$.
If, moreover, $M=R^n_k$ and $G$ acts on each $Bord_l(M)$ linearly, we say that $R^n_k$ is a \textit{linear corner representation} of $G$.

Reflecting $R^n_k$ in the first $k$ coordinate hyperplanes turns it into $R^n$ and thus reduces the \textit{corner} case to the case without boundary or corners. We call this \textit{doubling}; it is used in Lemma \ref{lem_doubling}.


\section{Orbifolds and orbifolds with corners}\label{sec_orbifolds}
We review some preliminaries on orbifolds and orbifolds with corners. Standard references for orbifolds and their description by Lie groupoids are \cite{alr,mo,mp}; see also \cite{cm, mm} and the pioneering work \cite{s}.
We use the original definition by charts.
\begin{definition}\label{orbifold}
(orbifolds).
Let $X$ be a topological space, and fix $n \geq 0$.
\begin{itemize}
\item An $n$-dimensional orbifold chart on $X$ is given by a connected open subset
$\tilde{U}\subseteq R^n$, a finite group $G$ acting smoothly on $\tilde{U}$, 
and a map $\phi:\tilde{U}\rightarrow X$ such that $\phi$ is $G$-invariant and induces a homeomorphism of $\tilde{U}/G$  onto
an open subset $U \subseteq  X$.
\item An embedding $\lambda : (\tilde{U} , G, \varphi  ) \hookrightarrow (\tilde{V} ,H, \psi)$ between two such charts is a
smooth embedding $\lambda : \tilde{U} \hookrightarrow \tilde{V}$
 with $\psi \lambda = \varphi$.  An orbifold atlas on $X$ is a family $\mathcal{U} = {(\tilde{U} , G, \varphi)}$ of such charts, which
cover $X$ and are locally compatible: given any two charts $(\tilde{U} , G, \varphi )$ for
$U = \varphi(\tilde{U}) \subset X$ and $(\tilde{V},H,\psi  )$ for $V \subseteq X$, and a point $x \in U \cap V$, there
exists an open neighbourhood $W \subseteq U \cap V$ of $x$ and a chart $(\tilde{W},K,\mu  )$ for $W$
such that there are embeddings $(\tilde{W},K,\mu  ) \hookrightarrow (\tilde{U}, G, \varphi  )$ and $(\tilde{W},K,\mu  )\hookrightarrow (\tilde{V},H,\psi)$.
\item An atlas $\mathcal{U}$ is said to refine another atlas $\mathcal{V}$ if for every chart in $\mathcal{U}$ there
exists an embedding into some chart of $\mathcal{V}$. Two orbifold atlases are said to be
equivalent if they have a common refinement.
\end{itemize}
An orbifold of dimension $n$, or $n$-orbifold, is a paracompact Hausdorff space $X$ equipped with an equivalence class $\mathcal{U}$ of $n$-dimensional orbifold atlases.
\end{definition}
All orbifolds, with or without corners, are assumed to be \textit{effective}: each local group acts effectively on its chart. If $Q$ is a compact oriented orbifold with corners that is not effective, the kernel $K$ of the local actions is locally constant. For connected $Q$ we put $[Q,f]:=\frac{1}{|K|}[Q_{\mathrm{red}},f]$, where $Q_{\mathrm{red}}$ is the effective orbifold obtained by dividing the local groups by $K$. This is compatible with $\partial$, since $K$ also acts trivially on $\partial Q$ and $(\partial Q)_{\mathrm{red}}=\partial(Q_{\mathrm{red}})$. This convention applies to the moduli spaces of Section \ref{sec_applications}. Since geometric chains are built from orbifolds with corners, we extend Definition \ref{orbifold} to orbifolds with corners, in the same way as manifolds are generalized to manifolds with corners.
\begin{definition}\label{orbi_corners}
(orbifolds with corners).
Let $X$ be a topological space, and fix $n \geq 0$.
\begin{itemize}
\item A generalized $n$-dimensional orbifold chart on $X$ is given by a connected open subset
$\tilde{U} \subseteq R^n_+$, a finite group $G$ acting smoothly on $\tilde{U}$, 
and a map $\varphi :\tilde{U} \rightarrow X$ such that $\varphi$ is $G$-invariant, $G$ maps each $\tilde{U}_k := \tilde{U} \cap R^n_{+k}$ to itself (it may permute the connected components of $\tilde U_k$), and $\varphi$ induces a homeomorphism of $\tilde{U}/G$  onto
an open subset $U \subseteq X$.
\item An embedding $\lambda : (\tilde{U} , G, \varphi  ) \hookrightarrow (\tilde{V} ,H,\psi  )$ between two such charts is a
diffeomorphism $\lambda:\tilde{U}\to\lambda(\tilde U)$ onto an open subset of $\tilde{V}$, where $\lambda$ and $\lambda^{-1}$ are smooth up to the corners (locally restrictions of smooth maps between open subsets of $R^n$), with $\lambda(\tilde{U}_k)\subseteq\tilde{V}_k$ for all $k$, and $\psi \lambda = \varphi$.  An orbifold atlas on $X$ is a family $\mathcal{U} = {(\tilde{U} , G, \varphi  )}$ of such charts, which
cover $X$ and are locally compatible: given any two charts $(\tilde{U} , G, \varphi  )$ for
$U = \varphi(\tilde{U}) \subseteq X$ and $(\tilde{V} ,H, \psi )$ for $V \subseteq X$, and a point $x \in U \cap V$, there
exists an open neighbourhood $W \subseteq U \cap V$ of $x$ and a chart $(\tilde{W},K,\mu  )$ for $W$
such that there are embeddings $(\tilde{W},K,\mu  ) \hookrightarrow (\tilde{U} , G, \varphi  )$ and $(\tilde{W},K,\mu  ) \hookrightarrow (\tilde{V} ,H,\psi  )$.
\item An atlas $\mathcal{U}$ is said to refine another atlas $\mathcal{V}$ if for every chart in $\mathcal{U}$ there
exists an embedding into some chart of $\mathcal{V}$. Two generalized orbifold atlases are said to be
equivalent if they have a common refinement.
\end{itemize}
An orbifold with corners of dimension $n$, or $n$-orbifold with corners, is a paracompact Hausdorff space $X$ equipped with an equivalence class $\mathcal{U}$ of $n$-dimensional generalized orbifold atlases.
\end{definition}
\begin{remark}
In Definition \ref{orbi_corners} the local groups may permute the connected components of $\tilde U_k$, as in Convention \ref{conv_faces} and Lemma \ref{lem_doubling}; this is needed, for instance, when an automorphism of a Riemann surface permutes boundary components (Section \ref{sec_applications}). Requiring each component to be $G$-invariant gives a strictly smaller class: $[0,\infty)^3/\mathbb{Z}_3$, with $\mathbb{Z}_3$ permuting the coordinates cyclically, has a single ray of points of depth two emanating from the image of $0$, whereas every model in which the local group preserves the components has three. 
\end{remark}
We shall use the following linearization of corner charts. The doubled group in (2) is in general a semidirect product, since the local group may permute the boundary hyperplanes.
\begin{lemma}[Doubling lemma]\label{lem_doubling}
Let $G$ be a finite group acting smoothly on a $G$-invariant open neighbourhood $U$ of $0$ in $R^n_k$, preserving the corner structure and fixing $0$. Then:
\begin{enumerate}
\item After shrinking $U$ and applying a $G$-equivariant, corner preserving diffeomorphism onto a neighbourhood of $0$, there is a linear isomorphism $R^n\cong R^k\oplus R^{n-k}$, $x=(y,l)$, carrying $R^n_k$ to $\{y_1,\dots,y_k\ge 0\}$, such that $G$ acts linearly by $g(y,l)=(P_g\,y,\ A_g\,l)$ with $P_g$ a permutation matrix and $A_g$ orthogonal for a $G$-invariant inner product on $R^{n-k}$.
\item Let $\Gamma\cong Z_2^k$ be the group generated by the reflections $r_i: y_i\mapsto -y_i$. Then $G$ normalizes $\Gamma$, and $\hat G:=\Gamma\rtimes G$ acts linearly on $R^n$. The inclusion $R^n_k\subset R^n$ induces a homeomorphism $R^n_k/G\cong R^n/\hat G$.
\item For $x\in R^n_k$ of depth $d$ (exactly $d$ of the coordinates $y_i$ vanish) the subgroup $\hat G_x\cap\Gamma$ has order $2^d$. Consequently, if $x,x'\in R^n_k$ have isotropy groups $\hat G_x$, $\hat G_{x'}$ which are conjugate in $\hat G$, then $x$ and $x'$ have the same depth.
\end{enumerate}
\end{lemma}
\begin{proof}
(1) Since $g$ preserves the corner structure, $dg_0$ preserves the tangent cone $R^n_k$ and maps each boundary hyperplane $\{x_i=0\}$, $i\le k$, to another one. Put $\varphi(x)=\frac{1}{|G|}\sum_{g\in G}(dg_0)^{-1}g(x)$. Each summand maps $U\cap\{x_i=0\}$ into $\{x_i=0\}$ and $U\cap\{x_i>0\}$ into $\{x_i>0\}$, and these sets are convex, so $\varphi$ preserves the corner structure. Moreover $d\varphi_0=\mathrm{id}$ and $\varphi\circ g=dg_0\circ\varphi$. By the inverse function theorem (applied to a smooth extension of $\varphi$ to a neighbourhood of $0$ in $R^n$), $\varphi$ is a corner preserving diffeomorphism near $0$ which linearizes the action. Now choose a $G$-invariant inner product. The lineality space $L=\{x_1=\dots=x_k=0\}$ of the cone $R^n_k$ is $G$-invariant, hence so is $L^\perp$, and $R^n_k=C\times L$ with $C\subset L^\perp$ a simplicial cone. $G$ permutes the extremal rays of $C$; letting $w_i$ be the unit vector on the $i$-th ray, $g w_i=w_{\sigma_g(i)}$ exactly. The coordinates $y_i$ with respect to the basis $(w_i)$ of $L^\perp$, together with linear coordinates $l$ on $L$, have the required properties.

(2) $g r_i g^{-1}=r_{\sigma_g(i)}$, so $G$ normalizes $\Gamma$. The map $(y,l)\mapsto(|y_1|,\dots,|y_k|,l)$ identifies $R^n/\Gamma$ with $R^n_k$, $G$-equivariantly, hence $R^n/\hat G=(R^n/\Gamma)/G\cong R^n_k/G$.

(3) For $x=(y,l)$ with $y\ge 0$, an element $\gamma\in\Gamma$ fixes $x$ if and only if it only changes signs of vanishing coordinates, so $|\hat G_x\cap\Gamma|=2^d$. If $\hat G_x$ and $\hat G_{x'}$ are conjugate in $\hat G$, then, $\Gamma$ being normal, $\hat G_x\cap\Gamma$ and $\hat G_{x'}\cap\Gamma$ are conjugate, so $x$ and $x'$ have the same depth. This proves the last claim.
\end{proof}

\section{Geometric chains and the comparison theorem}\label{sec_main}
In this section we give the precise definition of the geometric chain complex and prove Theorem \ref{main_thm}.

\subsection{The geometric chain complex}\label{sec_chains}
Coefficients are $\mathbb{Q}$ throughout and are omitted from the notation. We first treat continuous chains; the smooth variants are obtained by adding smoothness assumptions (Definition \ref{def_smooth_GC}). Let $X$ be a topological space. A \textit{(continuous) geometric simplex} in $X$ is a pair $(Q,f)$ with $Q$ as in Definition \ref{main_def} and $f:|Q|\to X$ continuous; isomorphisms are as in Definition \ref{main_def}. Let $\mathcal{B}_n(X)$ be the set of isomorphism classes of \textit{connected} $n$-dimensional geometric simplices, and write $[Q,f]$ for the class of $(Q,f)$. A disconnected $Q$ is identified with the sum of its components. Put
\begin{equation*}
C^g_n(X):=\mathbb{Q}[\mathcal{B}_n(X)]\big/\big\langle[-Q,f]+[Q,f]\big\rangle,\qquad \partial[Q,f]:=\sum_F[F,f|_F],
\end{equation*}
where $F$ runs over the codimension one faces of $Q$, i.e. the connected components of $\partial Q$ (Convention \ref{conv_faces}), with the induced orientation (outward normal first). Here and below, for a face $F$ we write $f|_F:=f\circ i_Q|_F$, and similarly for other functions on $Q$.

\begin{lemma}\label{lem_Cg}
$(C^g_*(X),\partial)$ is a chain complex, functorial in $X$. If $(Q,f)$ admits an orientation reversing automorphism, i.e. an orientation reversing diffeomorphism $a$ of $Q$ with $f\circ a=f$, then $[Q,f]=0$. The classes not admitting such an automorphism, each with a chosen orientation, form a basis of $C^g_n(X)$.
\end{lemma}
\begin{proof}
An isomorphism induces an isomorphism of boundaries, preserving induced orientations, so $\partial$ is well defined. The components of $\partial^2Q$ are either exchanged in pairs by the involution $\sigma$ of Convention \ref{conv_faces}, which reverses orientations and is compatible with the maps to $X$, or preserved by $\sigma$. In the first case the two classes cancel; in the second case the class vanishes by the next statement. Hence $\partial^2=0$. If $a$ is as stated, then $[Q,f]=[-Q,f]=-[Q,f]$, so $[Q,f]=0$. The last statement is clear.
\end{proof}

\begin{definition}\label{def_degenerate}
An $n$-dimensional geometric simplex $(Q,f)$ is \textit{thin} if there exist a smooth orbifold $N$ (possibly with corners, not necessarily compact) with $\dim N<n$, a smooth map $\pi:Q\to N$, and a continuous map $g:|N|\to X$, such that $f=g\circ\pi$. Let $S_n(X)\subseteq C^g_n(X)$ be the span of the thin classes, and
\begin{equation*}
D_n(X):=\{\,c\in S_n(X)\ :\ \partial c\in S_{n-1}(X)\,\}.
\end{equation*}
Elements of $D_*(X)$ are called \textit{degenerate}. The geometric chain complex is $GC_*(X):=C^g_*(X)/D_*(X)$, and its homology is denoted by $GH_*(X)$.
\end{definition}
Here $\pi$ is smooth in the orbifold sense: every point of $Q$ has charts of $Q$ and of $N$ between which $\pi$ lifts to a smooth map, smooth up to the corners. Replacing $N$ by the union of its components meeting $\pi(Q)$, which are finitely many, we may assume that $N$ is second countable. For example, if $s:\Delta^{n-1}\to X$ and $\sigma_j:\Delta^n\to\Delta^{n-1}$ is a codegeneracy map, then $(\Delta^n,s\circ\sigma_j)$ is thin. Its $j$-th and $(j+1)$-st faces both equal $[\Delta^{n-1},s]$ and cancel, and the remaining faces factor through codegeneracies, hence are thin. So it is degenerate, as in simplicial theory.

\begin{lemma}\label{lem_D_local}
$D_*(X)$ is a subcomplex, and $\phi_*D_*(X)\subseteq D_*(Y)$ for every continuous map $\phi:X\to Y$. Moreover, for every open subset $U\subseteq X$,
\begin{equation*}
S_*(U)=S_*(X)\cap C^g_*(U),\qquad D_*(U)=D_*(X)\cap C^g_*(U).
\end{equation*}
In particular $GC_*(U)\to GC_*(X)$ is injective.
\end{lemma}
\begin{proof}
If $c\in D_n$, then $\partial c\in S_{n-1}$ and $\partial(\partial c)=0\in S_{n-2}$, so $\partial c\in D_{n-1}$. If $f=g\circ\pi$ then $\phi\circ f=(\phi\circ g)\circ\pi$, so $\phi_*$ preserves $S$, hence $D$. Since $S_*(X)$ is spanned by a subset of the basis, an element of $C^g_*(U)$ lies in $S_*(X)$ if and only if each of its components is thin in $X$. If $f=g\circ\pi$ and $f(|Q|)\subseteq U$, then $\pi(Q)$ lies in the open suborbifold $g^{-1}(U)$ of $N$, and $f=g|\circ\pi$ exhibits $(Q,f)$ as thin in $U$. The statement for $D$ follows.
\end{proof}

\begin{definition}\label{def_smooth_space}
A \textit{smooth space} is a topological space $X$ together with, for every open $U\subseteq X$, a subalgebra $C^\infty(U)$ of the algebra of continuous real functions on $U$, such that:
\begin{enumerate}[(a)]
\item $U\mapsto C^\infty(U)$ is a sheaf;
\item $C^\infty(U)$ contains the constants, and $h(\varphi_1,\dots,\varphi_k)\in C^\infty(U)$ whenever $\varphi_1,\dots,\varphi_k\in C^\infty(U)$ and $h\in C^\infty(\mathbb{R}^k)$;
\item (smooth Urysohn functions) for every open $U$ and disjoint relatively closed subsets $A,B\subseteq U$, there is $\varphi\in C^\infty(U)$ with values in $[0,1]$, $\varphi|_A=0$ and $\varphi|_B=1$.
\end{enumerate}
Open subsets inherit smooth structures. A continuous map $f$ from (the underlying space of) an orbifold with corners, or from a smooth space, to $X$ is \textit{smooth} if $\varphi\circ f$ is smooth on $f^{-1}(U)$ for every $\varphi\in C^\infty(U)$. A \textit{smooth homotopy} is a homotopy $H:X\times[0,1]\to Y$ such that $H\circ(f\times\mathrm{id})$ is smooth for every smooth map $f$ from an orbifold with corners to $X$. The space $X$ is \textit{smoothly contractible} if $\mathrm{id}_X$ is smoothly homotopic to a constant map. A \textit{good cover} of $X$ is an open cover $\mathcal{G}$ such that every non-empty finite intersection of members of $\mathcal{G}$ is a disjoint union of open, smoothly contractible subsets.
\end{definition}

\begin{definition}\label{def_smooth_GC}
Let $X$ be a smooth space. The \textit{smooth} variants $C^{g,\infty}_*(X)$, $S^\infty_*(X)$, $D^\infty_*(X)$, $GC^\infty_*(X)$ and $GH^\infty_*(X)$ are defined in the same way as above, requiring that $f$ be smooth and, in the definition of thinness, that $g$ be smooth. Smooth singular chains $C^\infty_*(X)$ are generated by smooth maps $\Delta^n\to X$. Lemmas \ref{lem_Cg} and \ref{lem_D_local} hold verbatim, functoriality being with respect to smooth maps: in the proof of Lemma \ref{lem_D_local} one restricts $g$ to the open suborbifold $g^{-1}(U)$. Restricting to $Q$ which are manifolds with corners gives subcomplexes $C^{g,m}_*\subseteq C^g_*$ and $C^{g,\infty,m}_*\subseteq C^{g,\infty}_*$. We put $D^{m}_*:=D_*\cap C^{g,m}_*$ (the factor $N$ in the definition of thinness may still be an orbifold), $GC^m_*:=C^{g,m}_*/D^m_*\subseteq GC_*$, and similarly $GC^{\infty,m}_*$.
\end{definition}

\subsection{Comparison map, dimension axiom and homotopy invariance}\label{sec_comparison}
The standard simplex $\Delta^n$ with its standard orientation is a manifold with embedded faces. Define $p(s):=[\Delta^n,s]$ on singular simplices $s:\Delta^n\to X$ (resp. smooth singular simplices into a smooth space $X$). The orientation induced on the $i$-th face of $\Delta^n$ is $(-1)^i$ times the one transported by the face map $d_i$. Hence $p$ is a natural chain map $C_*(X;\mathbb{Q})\to GC_*(X)$ (resp. $C^\infty_*(X;\mathbb{Q})\to GC^\infty_*(X)$). By the example after Definition \ref{def_degenerate}, it factors through the normalized chain complex.

\begin{proposition}[dimension axiom]\label{prop_point}
$GC_n(pt)=0$ for $n\ge1$, and $GC_0(pt)\cong\mathbb{Q}$.
\end{proposition}
\begin{proof}
For $n\ge1$ every $[Q,c]$ is thin, since the constant map $c$ factors through a point. For $n\ge2$ the faces are thin as well, so $C^g_n(pt)=D_n(pt)$. For $n=1$, a connected compact oriented effective $1$-orbifold with corners is an interval or a circle, since a finite group acting effectively by orientation preserving maps on an open subset of $\mathbb{R}$ or $\mathbb{R}_{\ge0}$ and fixing a point is trivial. Hence $\partial[Q,c]=[pt]-[pt]=0$ or $0$, so $[Q,c]\in D_1$. In degree $0$ no class is thin, and the signed count of points gives $GC_0(pt)\cong\mathbb{Q}$.
\end{proof}

\begin{proposition}[homotopy invariance]\label{prop_homotopy}
Homotopic maps $\psi_0\simeq\psi_1:X\to Y$ induce chain homotopic maps $GC_*(X)\to GC_*(Y)$. The same holds for $GC^\infty_*$ and smooth homotopies between smooth maps of smooth spaces.
\end{proposition}
\begin{proof}
Let $H:X\times[0,1]\to Y$ be a homotopy. With the product orientation and the outward-normal-first convention,
\begin{equation*}
\partial[Q\times[0,1]]=[\partial Q\times[0,1]]+(-1)^n\big([Q\times\{1\}]-[Q\times\{0\}]\big)
\end{equation*}
for $\dim Q=n$. Put $\mathcal{H}[Q,f]:=(-1)^n[Q\times[0,1],\,H\circ(f\times\mathrm{id})]$; then $\partial\mathcal{H}+\mathcal{H}\partial=\psi_{1*}-\psi_{0*}$ on $C^g_*(X)$. If $f=g\circ\pi$, then $H\circ(f\times\mathrm{id})=H\circ(g\times\mathrm{id})\circ(\pi\times\mathrm{id})$, and $\dim(N\times[0,1])<\dim(Q\times[0,1])$, so $\mathcal{H}(S)\subseteq S$. For $c\in D$, the formula $\partial\mathcal{H}c=\psi_{1*}c-\psi_{0*}c-\mathcal{H}\partial c$ shows $\partial\mathcal{H}c\in S$, so $\mathcal{H}(D)\subseteq D$.
\end{proof}

\begin{corollary}\label{cor_contractible}
If $X$ is contractible (resp. $X$ is a smoothly contractible smooth space), then $GH_*(X)$ (resp. $GH^\infty_*(X)$) is $\mathbb{Q}$, concentrated in degree $0$, and $p_*$ is an isomorphism.
\end{corollary}

\subsection{Bisection and the bisection prism}\label{sec_bisection}
\begin{definition}\label{def_bisection}
Let $Q$ be a compact orbifold with corners. A \textit{bisecting function} is a smooth function $\rho:Q\to\mathbb{R}$ (given in orbifold charts by invariant functions) such that $0$ is a regular value of the restriction of $\rho$ to every stratum of the corner stratification of $Q$, i.e. to every connected component of the set of points of a given depth. We put
\begin{equation*}
Q_-=\{\rho\le0\},\qquad Q_+=\{\rho\ge0\},\qquad I=\{\rho=0\}.
\end{equation*}
\end{definition}

\begin{lemma}\label{lem_bisection_pieces}
$Q_\pm$ are orbifolds with corners. Their boundaries consist of $I$ (if non-empty) and of the parts $(\partial Q)_\pm$ of $\partial Q$ where $\pm\rho\circ i_Q\ge0$. The two orientations induced on $I$ from $Q_-$ and $Q_+$ are opposite.
\end{lemma}
\begin{proof}
Near a point of depth $k$ choose a chart $\tilde U\subseteq R^n_k$ with boundary coordinates $x_1,\dots,x_k$. Regularity of $\rho$ on the stratum means $d\tilde\rho\notin\mathrm{span}(dx_1,\dots,dx_k)$ at points where $\tilde\rho=0$. By the implicit function theorem $(\tilde\rho,x_1,\dots,x_k)$ extend to a coordinate system in which $\{\tilde\rho\le0\}\cap R^n_k$ becomes $R^n_{k+1}$. Since $\tilde\rho$ is invariant, this is compatible with the local group action. The remaining statements follow from this local description.
\end{proof}
\begin{remark}
Not every decomposition of $Q$ along a common face is a bisection. The interface must be the zero set of a function that is regular on every corner stratum, so it is transverse to all faces and contains no point of depth $\dim Q$. Neither cutting $[0,1]^2$ along a diagonal nor the barycentric subdivision of $\Delta^n$ is a bisection. Proposition \ref{prop_bisection} gives invariance under bisections only, and only this is used below.
\end{remark}

\begin{definition}[bisection prism]\label{def_prism}
Let $\rho$ be a bisecting function on $Q$, and put $\tau:=\frac14\tanh\circ\rho$. Define
\begin{align*}
W(Q,\rho)&:=\{(x,t)\in Q\times\mathbb{R}\ :\ 0\le t\le 1-|\tau(x)|\}\\
&=\{t\ge0\}\cap\{t\le1-\tau\}\cap\{t\le1+\tau\},
\end{align*}
with the product orientation. The function $\tau$ has the same zeros and signs as $\rho$, $d\tau=\frac14\,d\rho$ along $\{\rho=0\}$, and $|\tau|<\frac14$. No choice is involved, so $W(F,\rho|_F)=(i_Q|_F\times\mathrm{id})^{-1}W(Q,\rho)$ for every face $F$.
\end{definition}

\begin{lemma}\label{lem_prism}
$W=W(Q,\rho)$ is a compact orbifold with corners, and $\partial W$ is the disjoint union of:
\begin{itemize}
\item the bottom $B=Q\times\{0\}$;
\item the two top faces $T_-=\{t=1+\tau,\ \rho\le0\}$ and $T_+=\{t=1-\tau,\ \rho\ge0\}$, which project diffeomorphically onto $Q_-$ and $Q_+$, and meet along $\{(x,1):x\in I\}$;
\item the side faces $W(F,\rho|_F)$, one for each face $F$ of $Q$.
\end{itemize}
With the induced orientations, $[B]=(-1)^{n+1}[Q]$, $[T_\pm]=(-1)^n[Q_\pm]$, and the side faces carry the product orientations of $W(F,\rho|_F)$.
\end{lemma}
\begin{proof}
At a point of $W$, the active constraints are some boundary coordinates $x_i$ of $Q$, possibly $t\ge0$, and possibly one or both of $1\mp\tau-t\ge0$. If $t=0$ the top constraints are inactive, since $1-|\tau|>\frac34$. Both top constraints are active exactly when $\rho=0$ and $t=1$. The differentials of the active constraints are linearly independent. Indeed, the $dx_i$ are independent; $d(1\mp\tau-t)$ has a $dt$-component; and if both top constraints are active, their difference $2\,d\tau=\frac12\,d\rho$ does not lie in $\mathrm{span}(dx_i,dt)$, by regularity of $\rho$ on the stratum. Hence $W$ is locally modelled on $R^{n+1}_d$, $d$ being the number of active constraints. The local groups act trivially on $t$ and preserve $\rho$, so this holds for orbifolds as well. The list of faces follows. The function $\rho|_F$ is bisecting on $F$, because in charts $i_Q$ lifts to the inclusions $\tilde U\cap\{y_i=0\}\hookrightarrow\tilde U$, which map the points of depth $j$ of $\tilde U\cap\{y_i=0\}$ onto an open subset of the points of depth $j+1$ of $\tilde U$. (On underlying spaces the induced map of strata need not be a local homeomorphism, since the local group may permute the boundary hyperplanes: for $(y_1,y_2,v)\mapsto(y_2,y_1,-v)$ on $R^3_2$ it is $v\mapsto|v|$.) For the orientations, the bottom is treated as for $Q\times[0,1]$. For the top faces we deform $\tau$ to $s\tau$, $s\in[0,1]$, and let $s\to0$, which turns $T_\pm$ into the flat top $Q\times\{1\}$ restricted to $Q_\pm$. The side faces are as for $\partial Q\times[0,1]$.
\end{proof}

\begin{proposition}[bisection formula]\label{prop_bisection}
Let $(Q,f)$ be an $n$-dimensional geometric simplex and $\rho$ a bisecting function. In $C^g_*(X)$,
\begin{multline*}
[Q,f]-[Q_-,f]-[Q_+,f]\\=(-1)^{n+1}\,\partial[W(Q,\rho),f\circ\mathrm{pr}]+(-1)^n\sum_F[W(F,\rho|_F),f\circ\mathrm{pr}],
\end{multline*}
and $[W(F,\rho|_F),f\circ\mathrm{pr}]=0$ whenever $\rho|_F$ has no zero. In particular, if $Q$ has no boundary, then $[Q,f]$ and $[Q_-,f]+[Q_+,f]$ are homologous.
\end{proposition}
\begin{proof}
The formula follows from Lemma \ref{lem_prism}. If $\rho|_F$ has no zero, then $|\rho|_F|$ is smooth on $F$, and $(x,t)\mapsto(x,t/(1-|\tau(x)|))$ is a diffeomorphism $W(F,\rho|_F)\cong F\times[0,1]$. Composing with $t\mapsto1-t$ gives an orientation reversing automorphism preserving $f\circ\mathrm{pr}$, so the class vanishes by Lemma \ref{lem_Cg}.
\end{proof}
Proposition \ref{prop_bisection} replaces the gluing relation of the previous version of this paper: bisection is realized by a chain homotopy rather than imposed as a relation.

\subsection{Cutting operators}\label{sec_cutting}
In this subsection and the next one we work in the smooth setting of Definition \ref{def_smooth_GC}. Let $X$ be a smooth space, $\phi\in C^\infty(X)$ and $t\in\mathbb{R}$. For a smooth geometric simplex $(Q,f)$ the function $\phi\circ f$ is smooth on $Q$. We say that $t$ is \textit{regular} for a finite set of classes $[Q,f]$, closed under taking faces, if $t$ is a regular value of $\phi\circ f$ on every stratum of every such $Q$. By Sard's theorem (applied in orbifold charts) regular values form a set of full measure. For such classes put $\rho:=\phi\circ f-t$, which is a bisecting function on $Q$, and define
\begin{gather*}
\mathrm{Cut}^\pm[Q,f]:=[Q_\pm,f],\qquad \mathrm{Cut}:=\mathrm{Cut}^-+\mathrm{Cut}^+,\\ P[Q,f]:=(-1)^{n+1}[W(Q,\rho),f\circ\mathrm{pr}].
\end{gather*}
Since $\rho$ is determined by $f$, these operators are well defined on isomorphism classes and are linear.

\begin{lemma}\label{lem_cut}
On the span of classes for which $t$ is regular:
\begin{enumerate}
\item $\partial\,\mathrm{Cut}=\mathrm{Cut}\,\partial$;
\item $\partial P+P\partial=\mathrm{id}-\mathrm{Cut}$;
\item $\mathrm{Cut}$ and $P$ preserve $S^\infty_*$ and $D^\infty_*$.
\end{enumerate}
\end{lemma}
\begin{proof}
(1) The face $I$ appears in $\partial[Q_-]$ and in $\partial[Q_+]$ with opposite orientations (Lemma \ref{lem_bisection_pieces}), and the remaining faces are those of $\mathrm{Cut}(\partial[Q,f])$. (2) This is Proposition \ref{prop_bisection}: since $\dim F=n-1$, we have $P[F,f|_F]=(-1)^n[W(F,\rho|_F)]$, so the side terms equal $P(\partial[Q,f])$. (3) If $f=g\circ\pi$, then the maps on $Q_\pm$ and on $W$ factor through $\pi$, so $\mathrm{Cut}$ and $P$ preserve $S^\infty$. For $c\in D^\infty$ we get $\partial\,\mathrm{Cut}(c)=\mathrm{Cut}(\partial c)\in S^\infty$ and $\partial P(c)=c-\mathrm{Cut}(c)-P(\partial c)\in S^\infty$.
\end{proof}

\begin{lemma}[thin interfaces]\label{lem_thin_interface}
Let $(Q,f)$ be thin, with a chosen factorization $f=g\circ\pi$ through $\pi:Q\to N$. Assume that either $\dim N\le n-2$, or $t$ is a regular value of $\phi\circ g$ on every corner stratum of $N$ (in orbifold charts). Then the interface $(I,f|_I)$ is thin. Consequently, if $t$ satisfies this condition for the chosen factorizations of all thin classes involved, then $\mathrm{Cut}^-$ alone maps $D^\infty$ into $D^\infty$.
\end{lemma}
\begin{proof}
If $\dim N\le n-2$, then $f|_I=g\circ\pi|_I$ already exhibits $(I,f|_I)$ as thin. Otherwise $I=\pi^{-1}(L)$ with $L=(\phi\circ g)^{-1}(t)$, which by the argument of Lemma \ref{lem_bisection_pieces} is a suborbifold with corners of $N$ of dimension $\dim N-1<n-1=\dim I$, and $f|_I$ factors through $\pi|_I:I\to L$. By Sard's theorem the admissible $t$ form a set of full measure. For $c\in D^\infty$, $\partial\,\mathrm{Cut}^-(c)$ is $\mathrm{Cut}^-(\partial c)$ plus interface terms; both lie in $S^\infty$.
\end{proof}

\subsection{Mayer--Vietoris}\label{sec_MV}
Let $O$ be a smooth space, $O=U\cup V$ with $U,V$ open. By Definition \ref{def_smooth_space}(c) choose $\phi\in C^\infty(O)$ with values in $[0,1]$, $\phi=0$ on $O\setminus V$ and $\phi=1$ on $O\setminus U$, and use it in the cutting operators with $t\in(0,1)$. Then:
\begin{itemize}
\item the pieces produced by $\mathrm{Cut}^-$ map into $\{\phi\le t\}\subseteq U$;
\item the pieces produced by $\mathrm{Cut}^+$ map into $\{\phi\ge t\}\subseteq V$;
\item interfaces map into $\{\phi=t\}\subseteq U\cap V$;
\item the prism $P[Q,f]$ maps into $f(|Q|)$.
\end{itemize}
By Lemma \ref{lem_D_local} we regard $GC^\infty_*(U)$, $GC^\infty_*(V)$ and $GC^\infty_*(U\cap V)$ as subcomplexes of $GC^\infty_*(O)$.

\begin{proposition}\label{prop_MVa}
The inclusion $GC^\infty_*(U)+GC^\infty_*(V)\hookrightarrow GC^\infty_*(O)$ is a quasi-isomorphism.
\end{proposition}
\begin{proof}
\textit{Surjectivity on homology.} Let $z$ be a cycle, i.e. $\partial z\in D^\infty$, and choose $t$ regular for the classes of $z$ and their faces. By Lemma \ref{lem_cut},
\begin{equation*}
z=\mathrm{Cut}(z)+\partial P(z)+P(\partial z),
\end{equation*}
where $P(\partial z)\in D^\infty$ and $\mathrm{Cut}(z)=\mathrm{Cut}^-(z)+\mathrm{Cut}^+(z)\in GC^\infty_*(U)+GC^\infty_*(V)$.

\textit{Injectivity on homology.} Let $c\in C^{g,\infty}_*(U)+C^{g,\infty}_*(V)$ with $\partial c\in D^\infty_*(O)$, and suppose that $c=\partial w+e$ with $w\in C^{g,\infty}_*(O)$ and $e\in D^\infty_*(O)$. Choose $t$ regular for all classes occurring in $c,w,e$ and their faces. Applying $\mathrm{Cut}$ to $c=\partial w+e$ and using Lemma \ref{lem_cut},
\begin{equation*}
c-\partial\big(\mathrm{Cut}(w)+P(c)\big)=\mathrm{Cut}(e)+P(\partial c)\in D^\infty_*(O).
\end{equation*}
Here $\mathrm{Cut}(w)\in C^{g,\infty}_*(U)+C^{g,\infty}_*(V)$, and, since prisms map into the images of their bases, $P(c)\in C^{g,\infty}_*(U)+C^{g,\infty}_*(V)$. Hence $c$ is a boundary in $GC^\infty_*(U)+GC^\infty_*(V)$.
\end{proof}

\begin{proposition}\label{prop_MVb}
Let $K:=GC^\infty_*(U)\cap GC^\infty_*(V)$, the intersection taken inside $GC^\infty_*(O)$. Then $K/GC^\infty_*(U\cap V)$ is acyclic. Hence $GC^\infty_*(U\cap V)\to K$ is a quasi-isomorphism.
\end{proposition}
\begin{proof}
An element of $K$ is represented by $a\in C^{g,\infty}_*(U)$ such that $a-b=e\in D^\infty_*(O)$ for some $b\in C^{g,\infty}_*(V)$. Suppose its image in $K/GC^\infty_*(U\cap V)$ is a cycle, i.e. $\partial a=y+d$ with $y\in C^{g,\infty}_*(U\cap V)$ and $d\in D^\infty_*(O)$. Since $d=\partial a-y\in C^{g,\infty}_*(U)$, Lemma \ref{lem_D_local} gives $d\in D^\infty_*(U)$. Choose $t$ regular for all classes occurring in $a,b,e,y,d$ and their faces, and satisfying the condition of Lemma \ref{lem_thin_interface} for chosen factorizations of the thin classes in $e$. By Lemma \ref{lem_cut}, in $GC^\infty_*(U)$,
\begin{equation*}
[a]=[\mathrm{Cut}^-a]+[\mathrm{Cut}^+a]+\partial[Pa]+[P(\partial a)].
\end{equation*}
We examine the terms one by one.
\begin{itemize}
\item $\mathrm{Cut}^+a\in C^{g,\infty}_*(U\cap V)$, because $a$ maps into $U$ and $\mathrm{Cut}^+$ produces pieces mapping into $V$.
\item $P(\partial a)=P(y)+P(d)$, with $P(y)\in C^{g,\infty}_*(U\cap V)$ and $P(d)\in D^\infty_*(U)$.
\item $[Pa]\in K$, because $Pa-Pb=Pe\in D^\infty_*(O)$ and $Pb\in C^{g,\infty}_*(V)$.
\item $\mathrm{Cut}^-a=\mathrm{Cut}^-b+\mathrm{Cut}^-e$. Here $\mathrm{Cut}^-b\in C^{g,\infty}_*(U\cap V)$, because $b$ maps into $V$ and $\mathrm{Cut}^-$ produces pieces mapping into $U$. Moreover $\mathrm{Cut}^-e\in D^\infty_*(O)$ by Lemma \ref{lem_thin_interface}, and it maps into $U$, so $\mathrm{Cut}^-e\in D^\infty_*(U)$ by Lemma \ref{lem_D_local}.
\end{itemize}
Hence $[a]\equiv\partial[Pa]$ modulo $GC^\infty_*(U\cap V)$, with $[Pa]\in K$.
\end{proof}

\subsection{Proof of the main theorem}
\begin{lemma}\label{lem_smooth_singular}
Let $X$ be a smooth space admitting a good cover. Then the inclusion $C^\infty_*(X)\hookrightarrow C_*(X)$ of smooth singular chains into singular chains is a quasi-isomorphism.
\end{lemma}
\begin{proof}
Both complexes are homotopy invariant under smooth homotopies; for smooth singular chains, use the standard prism operator built from the affine triangulation of $\Delta^n\times[0,1]$. Both are acyclic in positive degrees on smoothly contractible spaces, and hence on disjoint unions of open, smoothly contractible subsets. Both satisfy the small chains lemma, because iterated barycentric subdivision is affine and hence preserves smoothness. Therefore the argument of Steps 1--3 of the proof of Theorem \ref{main_thm}, which uses only the properties just listed, applies to the inclusion $C^\infty_*\to C_*$.
\end{proof}

\begin{proof}[Proof of Theorem \ref{main_thm}]
Let $\mathcal{G}$ be a good cover of $X$, and let $\mathcal{G}^\cap$ be the set of non-empty finite intersections of members of $\mathcal{G}$. By Lemma \ref{lem_smooth_singular} we may compare $GH^\infty_*$ with the homology of smooth singular chains.

\textit{Step 1: members of $\mathcal{G}^\cap$.} Such a member is a disjoint union $\bigsqcup_\alpha W_\alpha$ of open, smoothly contractible sets. A simplex with connected domain maps into a single $W_\alpha$, so $C^{g,\infty}_*(\bigsqcup_\alpha W_\alpha)=\bigoplus_\alpha C^{g,\infty}_*(W_\alpha)$, and $\partial$ preserves the summands. Since $S^\infty_*$ is spanned by a subset of the basis, Lemma \ref{lem_D_local} gives $D^\infty_*(\bigsqcup_\alpha W_\alpha)=\bigoplus_\alpha D^\infty_*(W_\alpha)$. Hence $GC^\infty_*(\bigsqcup_\alpha W_\alpha)=\bigoplus_\alpha GC^\infty_*(W_\alpha)$, and likewise for smooth singular chains, so the claim follows from Corollary \ref{cor_contractible}.

\textit{Step 2: finite unions.} Let $O=O_1\cup\dots\cup O_k$ with $O_i\in\mathcal{G}^\cap$, and argue by induction on $k$. Put $U=O_1\cup\dots\cup O_{k-1}$, $V=O_k$ and $\mathcal{U}=\{U,V\}$. Then $U\cap V=\bigcup_{i<k}(O_i\cap O_k)$ is a union of $k-1$ members of $\mathcal{G}^\cap$ (or empty). The sequence
\begin{equation*}
0\to K\to GC^\infty_*(U)\oplus GC^\infty_*(V)\to GC^\infty_*(U)+GC^\infty_*(V)\to0,
\end{equation*}
with maps $c\mapsto(c,c)$ and $(a,b)\mapsto a-b$, is exact. By Propositions \ref{prop_MVa} and \ref{prop_MVb} it yields a long exact Mayer--Vietoris sequence for $GH^\infty_*$. The map $p$ induces a compatible map from the corresponding sequence for $\mathcal{U}$-small smooth singular chains to this sequence. The five lemma and the induction hypothesis show that $p_*$ is an isomorphism for $O$.

\textit{Step 3: $X$ itself.} Consider the directed family of finite unions $O\subseteq X$ of members of $\mathcal{G}$. Every element of $C^{g,\infty}_*(X)$ and every smooth singular chain has compact support, hence lies in some $O$. By Lemma \ref{lem_D_local}, $GC^\infty_*(X)=\mathrm{colim}_{O}GC^\infty_*(O)$. Homology commutes with filtered colimits, so Step 2 implies the claim for $X$.
\end{proof}

\begin{remark}[the role of degenerate chains]\label{rem_role_of_D}
Write $C'_*:=C^g_*(pt)$ for geometric chains in a point without the third relation. If $D$ is replaced by $0$, the proofs of Propositions \ref{prop_homotopy}, \ref{prop_bisection}, \ref{prop_MVa} and \ref{prop_MVb} and of Steps 2 and 3 above go through; in Proposition \ref{prop_MVb}, $C^{g,\infty}_*(U)\cap C^{g,\infty}_*(V)=C^{g,\infty}_*(U\cap V)$, both sides being spanned by subsets of the basis of Lemma \ref{lem_Cg}. So the third relation is used essentially only in Proposition \ref{prop_point}, through Corollary \ref{cor_contractible}, and the comparison theorem without it would amount to $H_n(C'_*)=0$ for $n>0$. This fails.

Let $p\ge5$ be a prime, and call a point $x$ of a $4$-dimensional orbifold with corners $Q$ a \textit{$p$-point} if its local group acts on $T_xQ\cong\mathbb{R}^4$ as $\mu_p\subset\mathbb{C}^*$ acts on $\mathbb{C}^2$ by scalars. A real linear map normalizing this action is $\mathbb{C}$-linear or $\mathbb{C}$-antilinear, hence has positive determinant; so each $p$-point carries a well defined sign, comparing the orientation of $Q$ with the complex one. At a point of depth $\ge1$ the local group fixes the sum of the covectors of the local boundary hyperplanes, hence a non-zero vector, so $p$-points are interior; they are isolated. Their signed count $\phi_p(Q)$ is additive and changes sign with the orientation, so it is a linear form on $C'_4$. Moreover $\phi_p\circ\partial=0$ on $C'_5$. Indeed, in a $5$-dimensional $W$ the points whose local group is $\mathbb{Z}_p$, acting as $\mathbb{R}\oplus\mathbb{C}^2$, form a $1$-manifold. It has no points of depth $k\ge2$: the sums over orbits of the covectors of the local boundary hyperplanes are linearly independent fixed covectors, and the fixed space is one-dimensional, so the local group would permute the $k$ hyperplanes transitively; then $k=p$, which is impossible for $p\ge7$ and for $p=5$ gives the regular representation of $\mathbb{Z}_5$ on $\mathbb{R}^5$, which is not of the given type. So it meets $\partial W$ transversally at points of depth one, which correspond to the $p$-points of the faces. With the outward normal first, the two ends of an arc carry opposite signs. The remaining ends are points $y$ whose local group $G_y$ strictly contains a subgroup $H\cong\mathbb{Z}_p$ of this type. Applied to $H$, the argument above shows that $y$ has depth at most one. Near $y$ the $1$-manifold lies in the union of the fixed lines of such subgroups of $G_y$. If $G_y$ fixed a non-zero vector, as it does at depth one, this vector would span all these lines, and their points near $y$ would have local group $G_y$. Hence $G_y$ fixes no non-zero vector, so $y$ is interior, and it is isolated among these ends, since the local group of any other point near $y$ fixes a non-zero vector. There the contributions of the link $S^4/G_y$ cancel in pairs $\pm v$, since $gv=-v$ for some $g\in G_y$ would give an orientation reversing map normalizing the action. The weighted projective plane $\mathbb{P}(1,1,p)$ is a closed oriented effective orbifold with a single singular point, which is a $p$-point. Hence $[\mathbb{P}(1,1,p),c]$ is a cycle in $C'_4$ which is not a boundary, and the classes for different $p$ are linearly independent; $\mathbb{P}(1,\dots,1,p)$ gives non-zero classes in degrees $4j$ in the same way. With the third relation these cycles are thin with empty boundary, hence zero.

On the other hand, $C'_1=C'_2=0$, as every compact oriented effective orbifold with corners of dimension $1$ or $2$ has a corner preserving orientation reversing diffeomorphism, and closed manifolds of positive dimension bound in $C'_*$. A closed $3$-manifold is homologous, by Proposition \ref{prop_bisection} applied along a Heegaard surface, to the sum of two handlebodies, each of which admits an orientation reversing diffeomorphism. $\mathbb{CP}^n$ is, up to sign, the boundary of the prism $\{(x,t):0\le t\le1-\epsilon\max_i|z_i|^2/|z|^2\}$, whose other faces are polydiscs $(D^2)^n$, again with orientation reversing diffeomorphisms; since $\Omega_*\otimes\mathbb{Q}$ is generated by products of complex projective spaces, every closed oriented manifold of positive dimension bounds. We do not know whether the third relation is needed in Corollary \ref{main_cor}, i.e. whether $H_n(C^{g,m}_*(pt))=0$ for $n\ge3$. The forms $\phi_p$ give no information here, since manifolds have no $p$-points. For $p=3$ the argument above fails: the $1$-manifold can also end at points of depth $3$, whose local group permutes three boundary hyperplanes cyclically, and the analogous count $\phi_3$ does not vanish on $\partial C'_5$.
\end{remark}

\begin{remark}[injectivity via Stokes]\label{rem_stokes}
For a smooth manifold $X$, injectivity of $p_*:H_n(X;\mathbb{Q})\to GH^\infty_n(X)$ also follows directly from Stokes' theorem, without the Mayer--Vietoris argument for geometric chains. Represent a class by a smooth singular cycle $s$ (Lemma \ref{lem_smooth_singular}).
\begin{itemize}
\item For an $n$-form $\omega$ on $X$ and $[Q,f]\in C^{g,\infty}_n(X)$ put $\int_{[Q,f]}\omega:=\int_Qf^*\omega$, with the $1/|G|$ convention of orbifold integration. Here $f^*\omega$ is defined because $f$ lifts locally to smooth maps of charts, $X$ being a manifold.
\item This vanishes on thin simplices: $g$ lifts locally to smooth maps on charts of $N$, so $f^*\omega=\pi^*(g^*\omega)$ with $g^*\omega=0$, as $\dim N<n$.
\item By Stokes' theorem on orbifolds with corners, with the boundary of Convention \ref{conv_faces}, $\int_{\partial c}\omega=\int_cd\omega$ for $c\in C^{g,\infty}_{n+1}(X)$.
\end{itemize}
If $p(s)=\partial w+d$ with $d\in D^\infty_n(X)$ and $\omega$ is closed, then $\int_s\omega=\int_wd\omega+\int_d\omega=0$. By de Rham's theorem for smooth singular chains, $[s]=0$ in $H_n(X;\mathbb{R})$, hence in $H_n(X;\mathbb{Q})$.

For an orbifold $X$, possibly with corners, an $n$-form $\omega$ is a compatible family of invariant forms $\tilde\omega$ on charts. For maps which are only smooth in the sense of Definition \ref{def_smooth_space}, $f^*\omega$ is not defined a priori: such maps need not lift (Section \ref{sec_examples}), and an invariant form need not be a combination of wedge products of differentials of invariant functions; on $\mathbb{R}^2/\{\pm1\}$ all such combinations vanish at $0$, whereas $dx\wedge dy$ does not. For liftable $f$ (Definition \ref{def_orb_chains}), however, $f^*\omega$ is well defined. Let $\tilde f_1,\tilde f_2:\tilde U\to\tilde V$ be lifts of $f$ on a connected chart $(\tilde U,G)$ of $Q$, into a chart $(\tilde V,H)$ of $X$; for instance $\tilde f_2=\tilde f_1\circ\gamma$ with $\gamma\in G$. The closed sets $\{\tilde f_2=h\tilde f_1\}$, $h\in H$, cover $\tilde U$, so the union of their interiors is dense. There $\tilde f_2^*\tilde\omega=\tilde f_1^*h^*\tilde\omega=\tilde f_1^*\tilde\omega$, hence everywhere by continuity. So the forms $\tilde f^*\tilde\omega$ are $G$-invariant, independent of the lift, and define a form $f^*\omega$ on $Q$. Moreover $d f^*\omega=f^*d\omega$, and $(g\circ\pi)^*\omega=\pi^*g^*\omega$ for liftable $g$ and $\pi$, since $\tilde g\circ\tilde\pi$ lifts $g\circ\pi$. Hence the three items above hold for $C^{g,\mathrm{orb}}_*(X)$ and $S^{\mathrm{orb}}_*(X)$: integration is defined on $GC^{\mathrm{orb}}_n(X)$, and Stokes' theorem $\int_{\partial c}\omega=\int_cd\omega$ holds for $c\in C^{g,\mathrm{orb}}_{n+1}(X)$. The injectivity argument then applies to $p_*:H_n(X;\mathbb{Q})\to GH^{\mathrm{orb}}_n(X)$, with $C^{\mathrm{orb}}_*(X)$ in place of $C^\infty_*(X)$; it computes $H_*(X;\mathbb{Q})$ by part (A) of the proof of Proposition \ref{prop_orb_comparison}, which uses only singular chains. In place of de Rham's theorem one needs that integration over cycles of $C^{\mathrm{orb}}_*(X)\otimes\mathbb{R}$ identifies $H^n_{dR}(X)$ with $\mathrm{Hom}(H_n(X;\mathbb{R}),\mathbb{R})$. For orbifolds without boundary the de Rham isomorphism is due to Satake \cite{s}. We only sketch the statement in the present form: the Mayer--Vietoris argument applies, using the Poincar\'e lemma for invariant forms on the sets $W/H$ of Proposition \ref{prop_good_cover_orbifold} (the homotopy operator of the equivariant contraction preserves invariant forms) and the small chains lemma for $C^{\mathrm{orb}}_*$, and, if $X$ has corners, the inward flow of Corollary \ref{cor_corners_general}. Granting this, we obtain injectivity of $H_n(X;\mathbb{Q})\to GH^{\mathrm{orb}}_n(X)$ without the Mayer--Vietoris argument for geometric chains. Since $GH^{\mathrm{orb}}_*(X)\to GH^\infty_*(X)$ is an isomorphism (Proposition \ref{prop_orb_comparison}), integration also gives a pairing of $GH^\infty_n(X)$ with $H^n_{dR}(X)$, although $f^*\omega$ is not defined for every chain in $C^{g,\infty}_n(X)$.
\end{remark}

\begin{proof}[Proof of Corollary \ref{main_cor}]
All constructions above preserve the class of manifolds with corners: bisection pieces, bisection prisms, products with $[0,1]$, and standard simplices. Hence the same argument applies verbatim to $GC^{\infty,m}_*$.
\end{proof}

\begin{remark}[stratified sources]\label{rem_stratified}
One may ask whether orbifolds with corners can be replaced by a more general class of compact stratified spaces with corners, such as Kreck's stratifolds \cite{kr}. We record what we know; nothing here is used elsewhere in the paper.
\begin{enumerate}
\item Apart from the description of $1$-dimensional sources in Proposition \ref{prop_point}, the proof of Theorem \ref{main_thm} uses the following properties of the sources: a boundary operation with the properties stated in Convention \ref{conv_faces}, in particular the orientation reversing involution $\sigma$ of $\partial^2Q$; products with $[0,1]$, and the notion of smooth homotopy of Definition \ref{def_smooth_space}; Sard's theorem on all corner strata, and the bisection pieces and prisms of Lemmas \ref{lem_bisection_pieces} and \ref{lem_prism}; the standard simplices with the face signs of Section \ref{sec_comparison}; and, for thinness, the analogous properties of the targets $N$ (open subsets, products with $[0,1]$ and regular level sets, as in Lemmas \ref{lem_D_local} and \ref{lem_thin_interface}). As far as we can see, orbifold charts enter the proof only through these lemmas, so the argument should apply to any class of sources with these properties. We have not written out such a class.
\item For orbifolds, the smooth structure of the underlying space alone is not the right data. With its invariant functions, $\mathbb{R}^2/\{\pm1\}$ has a $3$-dimensional Zariski tangent space at $0$: the cotangent space $\mathfrak{m}_0/\mathfrak{m}_0^2$ is spanned by the classes of $x^2,xy,y^2$. So an orbifold is a stratifold in the sense of \cite{kr} only after its smooth structure, and with it the class of smooth maps, is changed. For $Q=\mathbb{R}^3_2/\mathbb{Z}_2$, with $\mathbb{Z}_2$ acting by $(y_1,y_2,v)\mapsto(y_2,y_1,-v)$, we have $\partial Q\cong\{y_1=0\}$, and invariant functions restrict to functions on $\partial Q$ that are even in $v$ along $y_2=0$. So $v$ is not locally the restriction of a smooth function on $Q$, and the smooth structure of $\partial Q$ is not induced from that of $Q$ by restriction along $i_Q$. It is therefore natural to make the corner data part of the structure, as in Joyce's s-manifolds with corners \cite{jo25}. Kreck's $c$-stratifolds have collared boundaries and no corners, and stratifold homology is defined by bordism of cycles rather than by a chain complex.
\item Suppose that, in addition to (1), the class contains the cone $C(Q)=Q\times[0,1]/Q\times\{1\}$ of each connected member $Q$ of positive dimension, with $\partial C(Q)=Q\times\{0\}\sqcup\bigsqcup_FC(F)$, $F$ running over the codimension one faces of $Q$, oriented as in $\partial(Q\times[0,1])$ (the apex lies in no codimension one face). Let $C'_*$ be the analogue for this class of the complex of Remark \ref{rem_role_of_D}. Then $\mathcal{C}[Q,c]:=(-1)^{n+1}[C(Q),c]$, $n=\dim Q$, satisfies $\partial\mathcal{C}+\mathcal{C}\partial=\mathrm{id}$ on $C'_n$ for $n\ge1$, so $H_n(C'_*)=0$ for $n\ge1$. If moreover $\partial C'_1=0$, as for orbifolds (Proposition \ref{prop_point}), then $H_*(C'_*)\cong\mathbb{Q}$, concentrated in degree $0$. The argument of the first paragraph of Remark \ref{rem_role_of_D} would then give the comparison theorem for this class without the third relation. For instance, $\mathbb{P}(1,1,p)$ bounds its cone, which is not an orbifold. The condition $\partial C'_1=0$ is not automatic: an oriented $1$-dimensional source with three edges meeting at a vertex, if allowed, would give $3[pt]=0$. Cones on manifolds are in general not manifolds, so this does not answer the question at the end of Remark \ref{rem_role_of_D}.
\item In the proof of Theorem \ref{main_thm}, the coefficients $\mathbb{Q}$ enter, as far as we can see, only through Lemma \ref{lem_Cg}, where $[Q,f]=-[Q,f]$ is used to conclude $[Q,f]=0$. Over $\mathbb{Z}$ one can impose the vanishing of classes with an orientation reversing automorphism as a separate relation, as in \cite{lip}. Then $\partial$ is well defined. If $a$ is such an automorphism of $(Q,f)$, then $a$ maps each face $F$ to $a(F)$ reversing the induced orientations; in an orbit of $a$ of odd length $k$ on the faces, $a^k$ restricts to an orientation reversing automorphism of each face, and the faces in an orbit of even length cancel in pairs. The proof of Lemma \ref{lem_Cg} then gives $\partial^2=0$. We expect the proof of Theorem \ref{main_thm} to go through with integral coefficients, but we have not checked this in detail. Independently of this, rational coefficients are used for the weights $1/|K|$ of non-effective sources (after Definition \ref{orbifold}), in Proposition \ref{prop_MVa_cont} and in Remark \ref{rem_twisted}.
\item The cones in (3) fill cycles mapped to a point. We do not see how they bear on Conjecture \ref{conj_continuous}: the proof of Proposition \ref{prop_MVa_cont} uses $D$, and Proposition \ref{prop_reduction} concerns the locality of $D$.
\end{enumerate}
We do not know a class of stratified spaces with corners that contains orbifolds with corners and has the properties in (1) and (3).
\end{remark}

\subsection{Relative homology and excision}\label{sec_relative}
In this subsection $X$ is a smooth space and $A\subseteq X$ is an \textit{arbitrary} subset; $A$ is not required to be open or to carry a smooth structure. Let $C^{g,\infty}_{*,A}(X)\subseteq C^{g,\infty}_*(X)$ be the span of the classes $[Q,f]$ with $f(|Q|)\subseteq A$. It is spanned by a subset of the basis of Lemma \ref{lem_Cg}, and it is a subcomplex, since faces map into the image of $f$. Put
\begin{gather*}
GC^\infty_{*,A}(X):=C^{g,\infty}_{*,A}(X)\big/\big(C^{g,\infty}_{*,A}(X)\cap D^\infty_*(X)\big),\\
GC^\infty_*(X,A):=C^{g,\infty}_*(X)\big/\big(C^{g,\infty}_{*,A}(X)+D^\infty_*(X)\big),
\end{gather*}
and let $GH^\infty_*(X,A)$ be the homology of $GC^\infty_*(X,A)$. Thinness is always taken in $X$. $GC^\infty_{*,A}(X)$ is the image of $C^{g,\infty}_{*,A}(X)$ in $GC^\infty_*(X)$, so there is a short exact sequence
\begin{equation}\label{eq_relative_ses}
0\to GC^\infty_{*,A}(X)\to GC^\infty_*(X)\to GC^\infty_*(X,A)\to0
\end{equation}
and a long exact sequence in homology. For open $A$, Lemma \ref{lem_D_local} gives $GC^\infty_{*,A}(X)=GC^\infty_*(A)$, and so
\begin{equation*}
GC^\infty_*(X,A)=GC^\infty_*(X)/GC^\infty_*(A).
\end{equation*} The same definitions make sense for continuous chains; only the smooth version is used below, since excision is proved by cutting. Similarly, let $C^\infty_{*,A}(X)\subseteq C^\infty_*(X)$ be spanned by the smooth singular simplices with image in $A$.

A smooth map $F:X\to Y$ with $F(A)\subseteq B$ induces chain maps $GC^\infty_*(X,A)\to GC^\infty_*(Y,B)$, by Lemma \ref{lem_D_local}. If $H$ is a smooth homotopy with $H(A\times[0,1])\subseteq B$, the operator $\mathcal{H}$ of Proposition \ref{prop_homotopy} maps $C^{g,\infty}_{*,A}(X)$ into $C^{g,\infty}_{*,B}(Y)$, since the image of $H\circ(f\times\mathrm{id})$ lies in $H(A\times[0,1])$. So $H_0$ and $H_1$ induce chain homotopic maps of the relative complexes, and of the complexes $GC^\infty_{*,A}(X)\to GC^\infty_{*,B}(Y)$.

\begin{proposition}[excision]\label{prop_excision}
Let $U\subseteq A$ and $B\subseteq X$ be open subsets with $U\cup B=X$. Then the inclusion induces a quasi-isomorphism
\begin{equation*}
j:GC^\infty_*(B,A\cap B)\to GC^\infty_*(X,A).
\end{equation*}
No good cover is assumed. In particular, if $Z\subseteq X$ is closed and $Z\subseteq\mathrm{int}\,A$, then $GC^\infty_*(X\setminus Z,A\setminus Z)\to GC^\infty_*(X,A)$ is a quasi-isomorphism.
\end{proposition}
\begin{proof}
$j$ is well defined, since $C^{g,\infty}_{*,A\cap B}(B)\subseteq C^{g,\infty}_{*,A}(X)$ and $D^\infty_*(B)\subseteq D^\infty_*(X)$. The sets $X\setminus B$ and $X\setminus U$ are disjoint and closed. By Definition \ref{def_smooth_space}(c) choose $\phi\in C^\infty(X)$ with values in $[0,1]$, $\phi=0$ on $X\setminus B$ and $\phi=1$ on $X\setminus U$, and use it in the cutting operators with $t\in(0,1)$. Then:
\begin{itemize}
\item the pieces produced by $\mathrm{Cut}^-$ map into $\{\phi\le t\}\subseteq U\subseteq A$;
\item the pieces produced by $\mathrm{Cut}^+$ map into $\{\phi\ge t\}\subseteq B$;
\item by Lemma \ref{lem_bisection_pieces}, $\partial\,\mathrm{Cut}^+x-\mathrm{Cut}^+\partial x$ is a combination $J(x)$ of interfaces, which map into $\{\phi=t\}\subseteq U\cap B\subseteq A\cap B$;
\item the pieces and the prisms of a class map into its image.
\end{itemize}
In each step below, $t$ is chosen regular for the finitely many classes involved and their faces, and such that the condition of Lemma \ref{lem_thin_interface} holds for chosen factorizations of the thin classes involved. Lemma \ref{lem_thin_interface}, applied to $1-\phi$ and $1-t$ (which exchanges $\mathrm{Cut}^-$ and $\mathrm{Cut}^+$), shows that then $\mathrm{Cut}^+$ maps degenerate chains to degenerate chains. A degenerate chain which maps into $B$ lies in $D^\infty_*(B)$ by Lemma \ref{lem_D_local}.

\textit{Surjectivity on homology.} Let $z\in C^{g,\infty}_*(X)$ with $\partial z=a+d$, where $a\in C^{g,\infty}_{*,A}(X)$ and $d\in D^\infty_*(X)$. By Lemma \ref{lem_cut},
\begin{equation*}
z=\mathrm{Cut}^-z+\mathrm{Cut}^+z+\partial Pz+Pa+Pd .
\end{equation*}
Here $\mathrm{Cut}^-z$ and $Pa$ lie in $C^{g,\infty}_{*,A}(X)$, and $Pd\in D^\infty_*(X)$. So $[z]=[\mathrm{Cut}^+z]$ in $GH^\infty_*(X,A)$. Moreover $\mathrm{Cut}^+z\in C^{g,\infty}_*(B)$ and
\begin{equation*}
\partial\,\mathrm{Cut}^+z=\mathrm{Cut}^+a+\mathrm{Cut}^+d+J(z),
\end{equation*}
where $\mathrm{Cut}^+a$ and $J(z)$ lie in $C^{g,\infty}_{*,A\cap B}(B)$ and $\mathrm{Cut}^+d\in D^\infty_*(B)$. So $\mathrm{Cut}^+z$ is a cycle of $GC^\infty_*(B,A\cap B)$, and $j[\mathrm{Cut}^+z]=[z]$.

\textit{Injectivity on homology.} Let $c\in C^{g,\infty}_*(B)$ with $\partial c=a'+d'$, $a'\in C^{g,\infty}_{*,A\cap B}(B)$, $d'\in D^\infty_*(B)$, and suppose that $c=\partial w+a+e$ with $w\in C^{g,\infty}_*(X)$, $a\in C^{g,\infty}_{*,A}(X)$ and $e\in D^\infty_*(X)$. Applying $\mathrm{Cut}^+$,
\begin{equation*}
\mathrm{Cut}^+c=\partial\,\mathrm{Cut}^+w-J(w)+\mathrm{Cut}^+a+\mathrm{Cut}^+e ,
\end{equation*}
where $\mathrm{Cut}^+w\in C^{g,\infty}_*(B)$, $J(w)$ and $\mathrm{Cut}^+a$ lie in $C^{g,\infty}_{*,A\cap B}(B)$, and $\mathrm{Cut}^+e\in D^\infty_*(B)$. On the other hand, in $C^{g,\infty}_*(B)$,
\begin{equation*}
c=\mathrm{Cut}^-c+\mathrm{Cut}^+c+\partial Pc+Pa'+Pd',
\end{equation*}
where $\mathrm{Cut}^-c$ maps into $U\cap B$, so that $\mathrm{Cut}^-c$ and $Pa'$ lie in $C^{g,\infty}_{*,A\cap B}(B)$, and $Pd'\in D^\infty_*(B)$ by Lemma \ref{lem_cut}(3) in $B$. Hence $c\equiv\partial(\mathrm{Cut}^+w+Pc)$ modulo $C^{g,\infty}_{*,A\cap B}(B)+D^\infty_*(B)$, i.e.\ $[c]=0$ in $GH^\infty_*(B,A\cap B)$.

For the last statement take $U=\mathrm{int}\,A$ and $B=X\setminus Z$.
\end{proof}

For open $A$, Proposition \ref{prop_excision} with $U=A$ can also be derived from Propositions \ref{prop_MVa} and \ref{prop_MVb}, applied to $X=A\cup B$, since
\begin{equation*}
\frac{GC^\infty_*(A)+GC^\infty_*(B)}{GC^\infty_*(A)}\cong\frac{GC^\infty_*(B)}{K},\qquad K=GC^\infty_*(A)\cap GC^\infty_*(B);
\end{equation*}
the proof above does not use Proposition \ref{prop_MVb}.

To compare with singular homology we need a condition on the pair. Call a smooth space $Y$ \textit{admissible} if $p:C^\infty_*(Y)\to GC^\infty_*(Y)$ and the inclusion $C^\infty_*(Y)\hookrightarrow C_*(Y)$ are quasi-isomorphisms. By Theorem \ref{main_thm} and Lemma \ref{lem_smooth_singular}, smooth spaces with good covers are admissible; so are orbifolds with corners: in the proof of Corollary \ref{cor_corners_general}, the inclusion $i:X^\circ\to X$ induces isomorphisms on $H_*(C^\infty_*)$ and on $H_*$, and $C^\infty_*(X^\circ)\hookrightarrow C_*(X^\circ)$ is a quasi-isomorphism by Lemma \ref{lem_smooth_singular} and Proposition \ref{prop_good_cover_orbifold}; so $C^\infty_*(X)\hookrightarrow C_*(X)$ is a quasi-isomorphism.

\begin{definition}\label{def_admissible_pair}
A pair $(X,A)$, $A\subseteq X$ an arbitrary subset, is \textit{admissible} if $X$ is admissible and $A$ has an open neighbourhood $U$ which is admissible and admits a smooth homotopy $H:U\times[0,1]\to U$ with $H_0=\mathrm{id}_U$, $H_1(U)\subseteq A$ and $H_s(A)\subseteq A$ for all $s$.
\end{definition}

\begin{lemma}\label{lem_relative_retraction}
Let $U$ and $H$ be as in Definition \ref{def_admissible_pair}. Then the inclusions $GC^\infty_{*,A}(X)\hookrightarrow GC^\infty_*(U)$, $C^\infty_{*,A}(X)\hookrightarrow C^\infty_*(U)$ and $C_*(A)\hookrightarrow C_*(U)$ are chain homotopy equivalences.
\end{lemma}
\begin{proof}
Since $C^{g,\infty}_{*,A}(X)\subseteq C^{g,\infty}_*(U)$, Lemma \ref{lem_D_local} gives $C^{g,\infty}_{*,A}(X)\cap D^\infty_*(X)=C^{g,\infty}_{*,A}(X)\cap D^\infty_*(U)$; so $GC^\infty_{*,A}(X)$ is a subcomplex $i:GC^\infty_{*,A}(X)\hookrightarrow GC^\infty_*(U)$. The map $H_1:U\to U$ induces $H_{1*}$ on $GC^\infty_*(U)$ (Lemma \ref{lem_D_local}), with image in $GC^\infty_{*,A}(X)$ since $H_1(U)\subseteq A$; call the corestriction $r$. The operator $\mathcal{H}$ of Proposition \ref{prop_homotopy} for $H$ satisfies $\partial\mathcal{H}+\mathcal{H}\partial=H_{1*}-\mathrm{id}$ on $GC^\infty_*(U)$, and it preserves $GC^\infty_{*,A}(X)$ because $H(A\times[0,1])\subseteq A$. Hence $i\circ r\simeq\mathrm{id}$ and $r\circ i\simeq\mathrm{id}$. The same argument applies to smooth singular chains, with the prism operator of the proof of Lemma \ref{lem_smooth_singular}, and to singular chains, with the restriction $H_1|:U\to A$.
\end{proof}

\begin{theorem}\label{thm_relative}
Let $(X,A)$ be admissible. Then $p$ and the inclusion induce quasi-isomorphisms
\begin{equation*}
GC^\infty_*(X,A)\xleftarrow{\ p\ }C^\infty_*(X)/C^\infty_{*,A}(X)\longrightarrow C_*(X)/C_*(A).
\end{equation*} Hence there is an isomorphism
\begin{equation*}
p_*:H_*(X,A;\mathbb{Q})\xrightarrow{\ \cong\ }GH^\infty_*(X,A),
\end{equation*}
compatible with the long exact sequences of the pair (for $GH^\infty$, the one induced by \eqref{eq_relative_ses}, whose third term is $H_*(GC^\infty_{*,A}(X))$; see Remark \ref{rem_relative_examples}(5)), including the connecting maps, and natural with respect to smooth maps of admissible pairs.
\end{theorem}
\begin{proof}
$p$ maps $C^\infty_{*,A}(X)$ into $C^{g,\infty}_{*,A}(X)$, so it maps the short exact sequence $0\to C^\infty_{*,A}(X)\to C^\infty_*(X)\to C^\infty_*(X)/C^\infty_{*,A}(X)\to0$ to \eqref{eq_relative_ses}. In the commutative square
\begin{equation*}
\begin{array}{ccc}
C^\infty_{*,A}(X)&\hookrightarrow&C^\infty_*(U)\\
\downarrow p&&\downarrow p\\
GC^\infty_{*,A}(X)&\hookrightarrow&GC^\infty_*(U)
\end{array}
\end{equation*}
the horizontal maps are quasi-isomorphisms by Lemma \ref{lem_relative_retraction}, and the right map is one since $U$ is admissible. So the left map is a quasi-isomorphism. The middle map $p:C^\infty_*(X)\to GC^\infty_*(X)$ is one since $X$ is admissible. The five lemma gives the first claim. For the second, compare $0\to C^\infty_{*,A}(X)\to C^\infty_*(X)\to\cdot\to0$ with $0\to C_*(A)\to C_*(X)\to\cdot\to0$: the left map is a quasi-isomorphism by Lemma \ref{lem_relative_retraction} and the admissibility of $U$ (both sides are quasi-isomorphic to $C_*(U)$, compatibly), and the middle one by the admissibility of $X$. Morphisms of short exact sequences induce morphisms of the long exact sequences, so the connecting maps commute. The isomorphism is induced by $p$ and the inclusions, which are natural, and does not depend on $U$.
\end{proof}

\begin{remark}[examples and limitations]\label{rem_relative_examples}
\begin{enumerate}
\item If $A$ is open and admissible, for instance an open subset of a manifold or of an orbifold, take $U=A$ and $H(x,s)=x$. So $(X,A)$ is admissible whenever $X$ is admissible.
\item Let $X$ be a smooth manifold, possibly with boundary, and $A\subseteq X\setminus\partial X$ a closed submanifold without boundary. A tubular neighbourhood $U\cong\nu A$ and the fibrewise contraction $H_s(v)=(1-s)v$ satisfy Definition \ref{def_admissible_pair}; $U$ is a manifold, hence admissible.
\item Let $X$ be a smooth manifold with boundary and $A=\partial X$. A collar, i.e.\ an open embedding $c$ of $\partial X\times[0,1)$ onto an open neighbourhood $U$ of $\partial X$ with $c(x,0)=x$, together with $H_s(c(x,u))=c(x,(1-s)u)$ (a smooth homotopy by Lemma \ref{lem_orbifold_homotopy}), satisfies Definition \ref{def_admissible_pair}; $U$ is a manifold with boundary, hence admissible by Corollary \ref{cor_corners_general}. So $GH^\infty_*(X,\partial X)\cong H_*(X,\partial X;\mathbb{Q})$.
\item For a manifold or orbifold with corners and $A=\partial X$, as a subset, we have not constructed a global smooth homotopy as in Definition \ref{def_admissible_pair}. Locally there is one: in a linear chart as in Lemma \ref{lem_doubling}(1), with $\lambda:[0,\infty)\to[0,\infty)$ smooth, $\lambda(r)\le r$, $\lambda=0$ on $[0,\delta]$ and $\lambda(r)=r$ for $r\ge2\delta$, the map $H_s(y,l)=\big(((1-s)y_i+s\lambda(y_i))_i,\,l\big)$ is equivariant, preserves the boundary and maps $\{\min_iy_i<\delta\}$ into it for $s=1$; but we do not know how to glue such local homotopies. (A smooth retraction of a neighbourhood onto $\partial X$ does not exist near a point of depth $\ge2$, but Definition \ref{def_admissible_pair} does not require one.) For a subcomplex $A$ of a polyhedron $X$ as in Section \ref{sec_examples}(2), the retraction of a regular neighbourhood is piecewise linear, and we have not checked whether it can be chosen smooth for the smooth structure of Section \ref{sec_examples}(2). In both cases $GC^\infty_*(X,A)$ is defined, the long exact sequence and excision hold, but we do not claim Theorem \ref{thm_relative}.
\item If $A$ is itself an admissible smooth space whose smooth maps from orbifolds with corners are exactly the smooth maps to $X$ with image in $A$ (for instance, $A$ as in (2) or (3), or $A$ open), then $C^{g,\infty}_*(A)=C^{g,\infty}_{*,A}(X)$, and $D^\infty_*(A)$ is contained in $D^\infty_*(X)\cap C^{g,\infty}_{*,A}(X)$. So there is a natural map $GC^\infty_*(A)\to GC^\infty_{*,A}(X)$. If $(X,A)$ is admissible, it is a quasi-isomorphism: both complexes receive quasi-isomorphisms from $C^\infty_*(A)=C^\infty_{*,A}(X)$, by the admissibility of $A$ and by the proof of Theorem \ref{thm_relative}. For open $A$ it is an isomorphism (Lemma \ref{lem_D_local}). If moreover $A$ is a smooth retract of an open neighbourhood $V$, i.e.\ there is a smooth $q:V\to A$ with $q|_A=\mathrm{id}_A$, as in (2) (the bundle projection) and (3) ($c(x,u)\mapsto x$), then $D^\infty_*(A)=D^\infty_*(X)\cap C^{g,\infty}_{*,A}(X)$, so the map is an isomorphism: if $f=g\circ\pi$ and $f(|Q|)\subseteq A$, then $\pi(Q)$ lies in the open suborbifold $N'=g^{-1}(V)$ of $N$, and $f=(q\circ g|_{N'})\circ\pi$ is a factorization in $A$. For a general closed $A$ we do not know whether equality holds: a thin factorization $f=g\circ\pi$ in $X$ of a chain in $A$ need not have $g$ with image in $A$.
\end{enumerate}
\end{remark}

\subsection{Examples of smooth spaces with good covers}\label{sec_examples}
For an orbifold with corners $X$, the smooth functions on an open subset are those which lift to smooth invariant functions in orbifold charts, smooth up to the corners. A map $f:Q\to X$ which is smooth in the sense of Definition \ref{def_smooth_space} need not lift to charts (for instance $[0,1]\to\mathbb{R}/\mathbb{Z}_2$, $y\mapsto[\sqrt y]$: smooth functions on $\mathbb{R}/\mathbb{Z}_2$ are smooth functions of $x^2$). The following lemma shows that this causes no trouble for homotopies.

\begin{lemma}\label{lem_orbifold_homotopy}
Let $X$ be an orbifold with corners, $Y$ a smooth space and $H:X\times[0,1]\to Y$ continuous. Assume that for every $\varphi\in C^\infty(V)$, $V\subseteq Y$ open, the function $\varphi\circ H$ lifts, on the preimage of $H^{-1}(V)$ in $(\text{chart})\times[0,1]$, to smooth functions, smooth up to the corners. Then $H$ is a smooth homotopy in the sense of Definition \ref{def_smooth_space}. This applies in particular if $Y$ is an orbifold with corners and $H$ lifts locally to smooth maps of charts.
\end{lemma}
\begin{proof}
Let $f:Q\to X$ be smooth in the sense of Definition \ref{def_smooth_space}, and let $(q_0,s_0)\in Q\times[0,1]$. Choose a chart $\tilde U\cap C$ around $f(q_0)$, with $C=\{y_1,\dots,y_j\ge0\}$, on which the local group $G$ acts linearly, preserving $C$ and fixing the lift $0$ of $f(q_0)$. Let $\psi$ be the lift of $\varphi\circ H$.
\begin{itemize}
\item By Seeley's theorem \cite{se64}, applied successively to the coordinates $y_1,\dots,y_j$ and, if $s_0\in\{0,1\}$, to $s$, $\psi$ extends to a smooth function on a neighbourhood of $(0,s_0)$ in $\mathbb{R}^n\times\mathbb{R}$. Averaging over $G$ and multiplying by an invariant cut-off function, we may assume that $\psi$ is a $G$-invariant smooth function on $\mathbb{R}^n\times\mathbb{R}$.
\item By Schwarz's theorem \cite[Theorem~1]{sch75}, applied to the linear action on $\mathbb{R}^n\oplus\mathbb{R}$ which is trivial on the second summand,
\[
\psi(x,s)=h(p_1(x),\dots,p_m(x),s)
\]
with $h$ smooth and $p_i$ invariant polynomials.
\item The $p_i$ descend to smooth functions $\bar p_i$ near $f(q_0)$, so the $\bar p_i\circ f$ are smooth.
\end{itemize}
Hence, near $(q_0,s_0)$, the function $\varphi\circ H\circ(f\times\mathrm{id})=h(\bar p_1\circ f,\dots,\bar p_m\circ f,s)$ is smooth.
\end{proof}

\begin{enumerate}
\item \textit{Smooth manifolds} with their usual smooth functions. Smooth Urysohn functions exist, and geodesically convex balls for a Riemannian metric form a good cover. Here a smooth map from an orbifold with corners in the sense of Definition \ref{def_smooth_space} is the same as a smooth map in the usual sense.
\item \textit{Polyhedra.} Let $X=|K|\subseteq\mathbb{R}^N$ be the underlying space of a simplicial complex with linear simplices which is locally finite in $\mathbb{R}^N$, i.e. every point of $|K|$ has a neighbourhood in $\mathbb{R}^N$ meeting only finitely many simplices. Since $|K|$ is locally compact, it is closed in some open subset $W_0\subseteq\mathbb{R}^N$. Let $C^\infty(U)$ consist of the functions that are locally restrictions of smooth functions on open subsets of $\mathbb{R}^N$.
\begin{itemize}
\item Smooth Urysohn functions: write $U=X\cap W$ with $W\subseteq W_0$ open in $\mathbb{R}^N$. Relatively closed subsets of $U$ are closed in $W$, so we may restrict a smooth Urysohn function on $W$.
\item Good cover: the open stars of the vertices. A non-empty finite intersection of open stars is the open star $\mathrm{st}(\sigma)$ of a simplex $\sigma$. It is star-shaped with respect to the barycentre $b_\sigma$, and the linear contraction $(x,s)\mapsto(1-s)x+sb_\sigma$ preserves it. This contraction is the restriction of a smooth map of $\mathbb{R}^N$, hence a smooth homotopy.
\end{itemize}
\item \textit{Orbifolds} (without boundary). Smooth Urysohn functions are obtained from invariant partitions of unity. Good covers exist by Proposition \ref{prop_good_cover_orbifold}. (Moerdijk and Pronk \cite[Corollary~1.2.5]{mp99} construct an atlas whose charts $\tilde V$ and quotients $V$ are contractible and closed under finite intersections. Their charts are open stars of a triangulation, so the contractions are only continuous; smoothness in the sense of Definition \ref{def_smooth_space} is not addressed there.)
\item \textit{Orbifolds with corners}, which include the moduli spaces of Section \ref{sec_applications}. Axioms (a)--(c) of Definition \ref{def_smooth_space} are verified as in (3), using invariant partitions of unity which are smooth up to the corners. Instead of constructing good covers, we reduce to the interior (Corollary \ref{cor_corners_general}).
\end{enumerate}

\begin{proposition}\label{prop_good_cover_orbifold}
Every orbifold $X$ without boundary admits a good cover.
\end{proposition}
\begin{proof}
Choose a Riemannian metric on $X$. For $x\in X$ choose a chart $(\tilde U_x,G_x)$ with a lift $\tilde x$ fixed by $G_x$, and $\rho_x>0$ with the following properties:
\begin{itemize}
\item $B(\tilde x,8\rho_x)$ is relatively compact in $\tilde U_x$, and $8\rho_x$ is smaller than the convexity radius and the injectivity radius at every point of $B(\tilde x,8\rho_x)$;
\item for every ball $B$ of radius at most $3\rho_x$ centred in $B(\tilde x,2\rho_x)$ and every finite set $S\subseteq B$, the function $P_S(z)=\sum_{a\in S}d(z,a)^2$, which is smooth on $B$ by the first condition, has exactly one critical point in $B$; this holds for small $\rho_x$ by Karcher \cite[Theorem~1.2]{ka77}.
\end{itemize}
Put $U_x:=B(\tilde x,\rho_x)/G_x$. Lifts of $\mathrm{id}_X$ between charts, defined on open subsets, are local isometries. Near a point of a chart with trivial stabiliser the projection to $X$ is a local homeomorphism, and such points are dense since $X$ is effective.

\textit{Continuation.} Let $\tilde U,\tilde U'$ be charts, $\gamma:[0,T]\to\tilde U$ a geodesic and $\mu$ a lift of $\mathrm{id}_X$ near $\gamma(0)$ into $\tilde U'$. Assume that the geodesic $\gamma'$ of $\tilde U'$ with initial vector $d\mu(\dot\gamma(0))$ is defined on $[0,T]$. Then $\gamma(t)$ and $\gamma'(t)$ have the same image in $X$ for all $t$. Indeed, the set of $t$ such that this holds on $[0,t]$ is a closed interval $[0,t_1]$ with $t_1>0$. If $t_1<T$, choose a lift $\mu_1$ of $\mathrm{id}_X$ near $\gamma(t_1)$ with $\mu_1\gamma(t_1)=\gamma'(t_1)$. For $t\le t_1$ close to $t_1$ the points $\mu_1\gamma(t)$ and $\gamma'(t)$ lie in one orbit of the finite stabiliser of $\gamma'(t_1)$, so for some $g$ in it the closed set of $t$ with $g\mu_1\gamma(t)=\gamma'(t)$ contains an open interval. The geodesics $g\mu_1\gamma$ and $\gamma'$ then agree near $t_1$, contradicting the maximality of $t_1$.

Let $U_{x_0}\cap\dots\cap U_{x_m}\neq\emptyset$, with $\rho_{x_0}$ maximal, and let $\pi:\tilde U_{x_0}\to X$ be the projection. Its fibres are $G_{x_0}$-orbits. Pick $z$ in the intersection, with lifts $\tilde z\in B(\tilde x_0,\rho_{x_0})$ and $w\in B(\tilde x_j,\rho_{x_j})$. Continuing a lift $\mu$ of $\mathrm{id}_X$ with $\mu(w)=\tilde z$ along the geodesic from $w$ to $\tilde x_j$, whose length is less than $\rho_{x_j}\le\rho_{x_0}$, gives a lift of $x_j$ in $B(\tilde x_0,2\rho_{x_0})$; the continued geodesic is defined since $8\rho_{x_0}$ is smaller than the injectivity radius at $\tilde z$. The lifts of $x_j$ form one $G_{x_0}$-orbit and $G_{x_0}$ fixes $\tilde x_0$, so every lift $\tilde y$ of $x_j$ lies in $B(\tilde x_0,2\rho_{x_0})$, and $B(\tilde y,\rho_{x_j})\subseteq B(\tilde x_0,3\rho_{x_0})$.

For such $\tilde y$ let $\lambda$ be a lift of $\mathrm{id}_X$ near $\tilde y$ with $\lambda(\tilde y)=\tilde x_j$, and extend it by $\lambda:=\exp_{\tilde x_j}\circ d\lambda_{\tilde y}\circ\exp_{\tilde y}^{-1}$ to $B(\tilde y,\rho_{x_j})$. Both exponential maps are diffeomorphisms on the $\rho_{x_j}$-balls, so $\lambda$ is a diffeomorphism $B(\tilde y,\rho_{x_j})\to B(\tilde x_j,\rho_{x_j})$, and by continuation along radial geodesics it lies over $\mathrm{id}_X$. Hence $\lambda$ is a lift of $\mathrm{id}_X$ near every point with trivial stabiliser (its image then has trivial stabiliser too, the order of the local group being independent of the chart), and so it is an isometry. If $g\in G_{\tilde y}$, the stabiliser of $\tilde y$ in $G_{x_0}$, then $\lambda g\lambda^{-1}$ is an isometry of $B(\tilde x_j,\rho_{x_j})$ over $\mathrm{id}_X$. Near a point with trivial stabiliser it agrees with an element of $G_{x_j}$, hence it agrees with that element on the whole connected ball. Arguing symmetrically, $\lambda$ intertwines $G_{\tilde y}$ and $G_{x_j}$. In particular $\pi(B(\tilde y,\rho_{x_j}))=U_{x_j}$, so $\pi^{-1}(U_{x_j})$ is the union of the balls $B(\tilde y,\rho_{x_j})$, $\tilde y$ running over the lifts of $x_j$.

\textit{Two such balls are equal or disjoint.} Via $\lambda$, the map $B(\tilde y,\rho_{x_j})/G_{\tilde y}\to X$ is injective, and the stabiliser in $G_{x_0}$ of every point of $B(\tilde y,\rho_{x_j})$ is contained in $G_{\tilde y}$; the local groups have the same order. Suppose $z\in B(\tilde y,\rho_{x_j})\cap gB(\tilde y,\rho_{x_j})$, so $z=gz'$ with $z'\in B(\tilde y,\rho_{x_j})$. Injectivity gives $z'=hz$ with $h\in G_{\tilde y}$. Then $gh$ fixes $z$, so $g\in G_{\tilde y}$ and $gB(\tilde y,\rho_{x_j})=B(\tilde y,\rho_{x_j})$. In particular distinct lifts of $x_j$ give disjoint balls.

\textit{The intersection is a disjoint union of quotients.} The preimage of $U_{x_0}\cap\dots\cap U_{x_m}$ in $B(\tilde x_0,\rho_{x_0})$ is therefore the disjoint union of the non-empty sets
\begin{equation*}
W=B(\tilde x_0,\rho_{x_0})\cap B(\tilde y_1,\rho_{x_1})\cap\dots\cap B(\tilde y_m,\rho_{x_m}),
\end{equation*}
which $G_{x_0}$ permutes. Hence the intersection is the disjoint union of the open sets $W/H$, where $H\subseteq G_{x_0}$ is the stabiliser of $W$.

\textit{Each $W/H$ is smoothly contractible.} Each $W$ is strongly convex, as an intersection of balls of radius at most $\rho_{x_0}$ centred in $B(\tilde x_0,2\rho_{x_0})$. Let $S\subseteq W$ be an $H$-orbit. The balls $B_0:=B(\tilde x_0,3\rho_{x_0})$, $B(\tilde x_0,\rho_{x_0})$ and $B(\tilde y_j,\rho_{x_j})$ contain $S$ and satisfy the hypotheses of the second condition for $x=x_0$. Hence $P_S$ has a unique critical point $c$ in $B_0$. The critical point of $P_S$ in each of the other balls, which are contained in $B_0$, is a critical point in $B_0$, so it equals $c$, and $c\in W$. Since $H$ fixes $\tilde x_0$, it acts by isometries on $B_0$ preserving $S$, so $hc$ is a critical point of $P_S$ in $B_0$ and $hc=c$ for $h\in H$. As $W\subseteq B(c,2\rho_{x_0})$ and $2\rho_{x_0}$ is smaller than the injectivity radius at $c$, the geodesic contraction $(x,s)\mapsto\exp_c((1-s)\exp_c^{-1}x)$ is smooth on $W\times[0,1]$. It maps into $W$ by strong convexity and is $H$-equivariant. By Lemma \ref{lem_orbifold_homotopy}, applied to the chart $(W,H)$ of $W/H$, it descends to a smooth contraction of $W/H$.
\end{proof}

\begin{corollary}\label{cor_corners_general}
Let $X$ be an orbifold with corners. Then $p:C^\infty_*(X)\to GC^\infty_*(X)$ is a quasi-isomorphism, and $H_*(C^\infty_*(X))\cong H_*(X;\mathbb{Q})$. If a finite group $\varrho$ acts on $X$ by diffeomorphisms, these isomorphisms are $\varrho$-equivariant.
\end{corollary}
\begin{proof}
\textit{An inward vector field.} Inward pointing vectors form a convex cone in every chart, invariant under the local groups and under $\varrho$. Averaging and a $\varrho$-invariant partition of unity give a smooth $\varrho$-invariant vector field $v$ on $X$ which points strictly inwards along the boundary.

\textit{Retraction to the interior.} The forward flow $\Phi_t$ of $v$ is defined in charts by extending the lift of $v$ beyond the corners (Seeley \cite{se64}). Its existence time is locally bounded below by a positive constant, so there is a smooth $\varepsilon:X\to(0,1]$ such that $\Phi_t(x)$ is defined for $0\le t\le\varepsilon(x)$. In a chart with corner coordinates $y_1,\dots,y_k\ge0$ we have $\frac{d}{dt}y_i(\Phi_t)>0$ wherever $y_i=0$. Hence, for $t>0$, $\Phi_t(x)$ lies in the interior $X^\circ$, and forward trajectories starting in $X^\circ$ stay in $X^\circ$. Since forward trajectories starting in $X$ stay in $X$, where every extension agrees with the invariant lift of $v$, uniqueness of solutions shows that $\Phi_t$ does not depend on the extension and commutes with the local groups and with embeddings of charts, so $\Phi_t$ is well defined on $X$. Hence $r(x)=\Phi_{\varepsilon(x)}(x)$ maps $X$ to $X^\circ$. The map $H(x,s)=\Phi_{s\varepsilon(x)}(x)$ gives homotopies $\mathrm{id}_X\simeq i\circ r$ and $\mathrm{id}_{X^\circ}\simeq r\circ i$, where $i:X^\circ\to X$ is the inclusion. They lift to smooth maps of charts, hence are smooth homotopies by Lemma \ref{lem_orbifold_homotopy}.

\textit{Conclusion.} By Proposition \ref{prop_homotopy} and the corresponding statements for smooth singular and singular chains, $i$ induces isomorphisms on $GH^\infty_*$, on $H_*(C^\infty_*)$ and on $H_*$. Theorem \ref{main_thm} applies to $X^\circ$ by Proposition \ref{prop_good_cover_orbifold}, and the claim follows from naturality of $p$. If $\varrho$ acts, the isomorphisms are induced by the $\varrho$-equivariant maps $i$ and $p$, hence are $\varrho$-equivariant.
\end{proof}

\subsection{Liftable maps and orbifold morphisms}\label{sec_orb_chains}
For an orbifold target, integration of forms (Remark \ref{rem_stokes}) needs chains whose maps lift to charts, and pullback along orbifold bundles (Section \ref{sec_operations}) moreover needs an orbifold morphism as part of the data. We introduce the corresponding complexes and show that they compute the same homology.

\begin{definition}\label{def_orb_chains}
Let $X$ be an orbifold with corners.
\begin{enumerate}[(a)]
\item A continuous map $f:|Q|\to|X|$, $Q$ an orbifold with corners, is \textit{liftable} if every $q\in Q$ has charts $(\tilde U,G,\varphi)$ of $Q$ around $q$ and $(\tilde V,H,\psi)$ of $X$ around $f(q)$, and a smooth map $\tilde f:\tilde U\to\tilde V$, smooth up to the corners, with $\psi\circ\tilde f=f\circ\varphi$. No equivariance is required of $\tilde f$.
\item An \textit{orbifold with corners in the wide sense} is a proper \'etale Lie groupoid with corners, up to Morita equivalence, which need not be effective. Here the spaces of arrows and objects are manifolds with corners, the source and target maps are local diffeomorphisms of manifolds with corners (so they preserve depth), and Morita equivalences are given by principal bibundles which are manifolds with corners. As without corners \cite{mm,mp}, such a groupoid is isomorphic near each object $x$ to the action groupoid of the finite group $\mathcal{G}_x$ on a neighbourhood of $x$, which can be linearized by Lemma \ref{lem_doubling}(1); conversely, an orbifold with corners in the sense of Definition \ref{orbi_corners} is represented by the groupoid of germs of embeddings between the charts of an atlas \cite{mp}; for example $BH=(H\rightrightarrows pt)$ for a finite group $H$. Dividing the local groups by the kernels of their actions gives an effective orbifold with corners $N_{\mathrm{red}}$ with $|N_{\mathrm{red}}|=|N|$. An \textit{orbifold morphism} $\hat f:Q\to X$ between such objects is a generalized (Hilsum--Skandalis) morphism of representing groupoids, smooth up to the corners, taken up to $2$-isomorphism \cite{mp,mm,alr}. Its underlying map is $|\hat f|:|Q|\to|X|$.
\item $C^{g,\mathrm{orb}}_*(X)\subseteq C^{g,\infty}_*(X)$ is the span of the classes $[Q,f]$ with $f$ liftable. Such a class is \textit{thin} if $f=g\circ\pi$, where $\pi:Q\to N$ and $g:|N|\to|X|$ are liftable and $N$ is an orbifold with corners with $\dim N<\dim Q$. This defines $S^{\mathrm{orb}}_*(X)$, $D^{\mathrm{orb}}_*(X):=\{c\in S^{\mathrm{orb}}_*(X):\partial c\in S^{\mathrm{orb}}_*(X)\}$ and $GC^{\mathrm{orb}}_*(X):=C^{g,\mathrm{orb}}_*(X)/D^{\mathrm{orb}}_*(X)$. $C^{\mathrm{orb}}_*(X)$ is the complex of liftable singular simplices $\Delta^n\to|X|$.
\item $C^{g,\mathrm{mor}}_*(X)$ is defined as $C^g_*(X)$ in Section \ref{sec_chains}, from pairs $(Q,\hat f)$ of a compact oriented (effective) orbifold with corners $Q$ and an orbifold morphism $\hat f:Q\to X$. Here $(Q',\hat f')$ and $(Q,\hat f)$ are isomorphic if there is an orientation preserving diffeomorphism $h:Q'\to Q$ with $\hat f\circ\hat h\cong\hat f'$, where $\hat h$ is the morphism of Lemma \ref{lem_orb_maps}(d), and $\partial[Q,\hat f]:=\sum_F[F,\hat f\circ i_Q|_F]$. A class is \textit{thin} if $\hat f\cong\hat g\circ\hat\pi$ for orbifold morphisms $\hat\pi:Q\to N$ and $\hat g:N\to X$, where $N$ is an orbifold with corners \textit{in the wide sense} with $\dim N<\dim Q$. This defines $S^{\mathrm{mor}}_*(X)$, $D^{\mathrm{mor}}_*(X)$ and $GC^{\mathrm{mor}}_*(X)$. $C^{\mathrm{mor}}_*(X)$ is spanned by the ($2$-isomorphism classes of) orbifold morphisms $\hat\sigma:\Delta^n\to X$, with faces $\hat\sigma\circ d_i$.
\end{enumerate}
In both cases $p$ is defined as in Section \ref{sec_comparison}.
\end{definition}

\begin{lemma}\label{lem_orb_maps}
\begin{enumerate}[(a)]
\item Let $\hat f:Q\to X$ be an orbifold morphism, $q\in Q$, $(\tilde U,G)$ a chart of $Q$ centred at $q$, with $\tilde U$ a $G$-invariant convex neighbourhood of the fixed point $\tilde q$ in a linear corner representation, and $(\tilde V,H)$ a chart of $X$ centred at $|\hat f|(q)$. After shrinking $\tilde U$, $\hat f$ is represented by a pair $(\tilde f,\lambda)$ of a homomorphism $\lambda:G\to H$ and a $\lambda$-equivariant smooth map $\tilde f:\tilde U\to\tilde V$. Two pairs represent $2$-isomorphic morphisms if and only if they differ by conjugation, $(h\tilde f,h\lambda h^{-1})$ with $h\in H$. In particular $|\hat f|$ is liftable.
\item Liftable maps are smooth in the sense of Definition \ref{def_smooth_space}, and compositions and products of liftable maps are liftable (products of lifts are lifts).
\item The maps $i_Q:\partial Q\to Q$ and $\sigma$ of Convention \ref{conv_faces}, the inclusions of $Q_\pm$ and $I$ into $Q$ (Definition \ref{def_bisection}), the projection $\mathrm{pr}:W(Q,\rho)\to Q$ (Definition \ref{def_prism}), products of orbifold morphisms and their restrictions to open suborbifolds are orbifold morphisms, and $i_Q\circ i_{\partial Q}\circ\sigma=i_Q\circ i_{\partial Q}$ as morphisms.
\item Let $h:Q'\to Q$ be a local diffeomorphism of orbifolds with corners, i.e.\ a continuous map which lifts locally to local diffeomorphisms of charts, with $Q$ effective. Then there is an orbifold morphism $\hat h$ with $|\hat h|=h$, unique up to a unique $2$-isomorphism, and $\widehat{h\circ k}\cong\hat h\circ\hat k$. In particular a finite group acting on $Q$ by diffeomorphisms acts on $C^{g,\mathrm{mor}}_*(Q)$ and $C^{\mathrm{mor}}_*(Q)$.
\end{enumerate}
\end{lemma}
\begin{proof}
(a) Over $\tilde U$ a generalized morphism to $H\ltimes\tilde V$ is a principal $H$-bundle over $\tilde U$ with a commuting $G$-action and an equivariant map to $\tilde V$ \cite{mp,alr}; after shrinking $\tilde U$ we may assume that this map lands in $\tilde V$. Since $\tilde U$ is contractible the bundle is trivial, and a trivialization, unique up to $H$ as $\tilde U$ is connected, turns the $G$-action into a homomorphism $\lambda$ and the map into $\tilde f$. A $2$-isomorphism is a smooth map $\tilde U\to H$, which is constant.

(b) Smooth functions on open subsets of $X$ lift to smooth invariant functions on charts, so $\varphi\circ f$ lifts locally to $\tilde\varphi\circ\tilde f$ and is smooth. Let $f:Q\to X$ and $g:X\to Y$ be liftable, $q\in Q$, $x=f(q)$, and let $\tilde f:\tilde U\to\tilde V_1$ and $\tilde g:\tilde V_2\to\tilde W$ be lifts near $q$ and $x$. Choose a chart $(\tilde V,K)$ centred at $x$, with lift $\tilde x$, and embeddings $\lambda_i:\tilde V\to\tilde V_i$. The lifts of $x$ in $\tilde V_1$ form one orbit of the local group, so after composing $\tilde f$ with an element of it we may assume $\tilde f(\tilde q)=\lambda_1(\tilde x)$. After shrinking $\tilde U$ to a chart centred at $q$, $\tilde f(\tilde U)\subseteq\lambda_1(\tilde V)$, and $\tilde g\circ\lambda_2\circ\lambda_1^{-1}\circ\tilde f$ lifts $g\circ f$.

(c) Over a chart $(\tilde U,G)$ of $Q$, the orbifolds $\partial Q$, $\partial^2Q$, $Q_\pm$, $I$ and $W(Q,\rho)$ are represented by the action groupoids of $G$ on $\partial\tilde U$, $\partial^2\tilde U$, $\tilde U\cap\{\pm\tilde\rho\ge0\}$, $\{\tilde\rho=0\}$ and $W(\tilde U,\tilde\rho)$ (Convention \ref{conv_faces}, Lemmas \ref{lem_bisection_pieces} and \ref{lem_prism}); these are Morita equivalent to the charts given by a component and its stabiliser. $G$ acts effectively on $\{\tilde\rho=0\}$: an element fixing it pointwise near a fixed point preserves the covector $d\tilde\rho\neq0$, so its differential is unipotent of finite order, hence the identity. The maps in question are induced by the evident $G$-equivariant maps, with $\lambda=\mathrm{id}_G$, which are compatible with embeddings of charts. On $\partial^2\tilde U$ the identity $i\circ i\circ\sigma=i\circ i$ holds pointwise. Products and restrictions are clear.

(d) Let $\tilde h_1,\tilde h_2:\tilde U\to\tilde V$ be two smooth lifts of $h$, with $\tilde h_1$ a local diffeomorphism, $\tilde U$ connected, and let $E_\eta$, $\eta\in H$, be the open set of points near which $\tilde h_2=\eta\tilde h_1$. The closed sets $\{\tilde h_2=\eta\tilde h_1\}$ cover $\tilde U$, so $\bigcup_\eta E_\eta$ is dense. If $w\in\bar E_\eta\cap\bar E_{\eta'}$, then $\eta^{-1}\eta'$ fixes $\tilde h_1(w)$ and, $d\tilde h_1(w)$ being invertible, has trivial differential there; an element of finite order with this property acts trivially near its fixed point, so $\eta=\eta'$ by effectiveness. Hence the sets $\bar E_\eta$ are disjoint, closed and cover $\tilde U$, and $\tilde h_2=\eta\tilde h_1$ for a single $\eta$. The same argument shows that $\tilde h\circ\gamma=\lambda(\gamma)\circ\tilde h$ for a unique $\lambda(\gamma)\in H$, that $\lambda$ is a homomorphism, and that the only $2$-automorphism of $(\tilde h,\lambda)$ is the identity. So the local pairs glue to $\hat h$, uniquely up to a unique $2$-isomorphism, and $\widehat{h\circ k}$ and $\hat h\circ\hat k$ are both such morphisms for $h\circ k$. By Example \ref{ex_lift_not_unique}, the assumption that $h$ is a local diffeomorphism cannot be dropped.
\end{proof}

Forgetting $\hat f$ gives chain maps $C^{\mathrm{mor}}_*(X)\to C^{\mathrm{orb}}_*(X)\subseteq C^\infty_*(X)$ and $C^{g,\mathrm{mor}}_*(X)\to C^{g,\mathrm{orb}}_*(X)\subseteq C^{g,\infty}_*(X)$, by Lemma \ref{lem_orb_maps}(a),(b). They map $D^{\mathrm{mor}}_*$ into $D^{\mathrm{orb}}_*$, and $D^{\mathrm{orb}}_*$ into $D^\infty_*$: if $\hat f\cong\hat g\circ\hat\pi$ through $N$, then $|\hat f|=|\hat g|\circ|\hat\pi|$ with $|\hat\pi|:Q\to N_{\mathrm{red}}$ and $|\hat g|:|N_{\mathrm{red}}|\to|X|$ liftable. So we obtain chain maps $GC^{\mathrm{mor}}_*(X)\to GC^{\mathrm{orb}}_*(X)\to GC^\infty_*(X)$ compatible with $p$. The first map is not injective: if $\hat u,\hat u'$ are the two morphisms $[-1,1]\to\mathbb{R}/\{\pm1\}$ with lifts $u$ and $\mathrm{sgn}\cdot u$ of Example \ref{ex_lift_not_unique}, then $[[-1,1],\hat u]$ and $[[-1,1],\hat u']$ are distinct non-zero basis classes, neither is thin since $|\hat u|$ is not constant, so their difference is non-zero in $GC^{\mathrm{mor}}_1$, and it maps to $0$. We do not know whether the second map is injective. For instance $D^{\mathrm{orb}}_*\subseteq D^\infty_*\cap C^{g,\mathrm{orb}}_*$, and we do not know whether equality holds.

\begin{example}\label{ex_lift_not_unique}
An orbifold morphism is not determined by its underlying map. Let $X=\mathbb{R}/\{\pm1\}$, $u(t)=(2+t)e^{-1/t^2}$ for $t\neq0$ and $u(0)=0$. (We take $u$ not even: for even $u$, the map $t\mapsto-t$ would make all chains below vanish by Lemma \ref{lem_Cg}.) Then $u$ and $\mathrm{sgn}\cdot u$ are smooth lifts of the same liftable map $[-1,1]\to|X|$, $t\mapsto[u(t)]$. Since $[-1,1]$ is simply connected, orbifold morphisms $[-1,1]\to X$ correspond to smooth maps $[-1,1]\to\mathbb{R}$ up to sign, so these lifts define two different morphisms. Let $Y=(\mathbb{R}\times S^1)/\{\pm1\}\to X$, with $-1$ acting on $S^1$ by the rotation by $\pi$. The fibre products of Section \ref{sec_operations} along the two morphisms are both $[-1,1]\times S^1$, with the maps $(t,\theta)\mapsto[u(t),\theta]$ and $(t,\theta)\mapsto[\mathrm{sgn}(t)u(t),\theta]$ to $|Y|$. These maps differ by the rotation by $\pi$ over $\{t<0\}$, and the two chains are not isomorphic: an isomorphism $\phi$ preserves $u\circ\mathrm{pr}_1$; since $u>0$ for $t\neq0$, $u$ is strictly increasing on $(0,1]$ with image $(0,3e^{-1}]$ and $u([-1,0))=(0,e^{-1}]$, so $\phi$ is the identity over $\{t>0,\ u(t)>e^{-1}\}$, hence by connectedness on $(0,1]\times S^1$, hence the rotation by $\pi$ on $[-1,0)\times S^1$, which is not continuous at $t=0$.
\end{example}

\begin{proposition}\label{prop_orb_comparison}
Let $X$ be an orbifold with corners. In the commutative diagram
\[
\xymatrix{
C^{\mathrm{mor}}_*(X)\ar[r]\ar[d]_p & C^{\mathrm{orb}}_*(X)\ar[r]\ar[d]_p & C^\infty_*(X)\ar[r]\ar[d]_p & C_*(X;\mathbb{Q})\\
GC^{\mathrm{mor}}_*(X)\ar[r] & GC^{\mathrm{orb}}_*(X)\ar[r] & GC^\infty_*(X) &
}
\]
all maps are quasi-isomorphisms. In particular $p$ induces isomorphisms $H_*(X;\mathbb{Q})\cong GH^{\mathrm{orb}}_*(X)\cong GH^{\mathrm{mor}}_*(X)$. They are natural with respect to liftable maps, resp. orbifold morphisms, and equivariant with respect to every finite group acting on $X$ by diffeomorphisms.
\end{proposition}
\begin{proof}
The diagram commutes, and $C^\infty_*(X)\to C_*(X;\mathbb{Q})$ and $p:C^\infty_*(X)\to GC^\infty_*(X)$ are quasi-isomorphisms by Corollary \ref{cor_corners_general}. Hence it suffices to prove the following, for $\star\in\{\mathrm{orb},\mathrm{mor}\}$:
\begin{enumerate}[(A)]
\item the composite $C^\star_*(X)\to C_*(X;\mathbb{Q})$ is a quasi-isomorphism;
\item $p:C^\star_*(X)\to GC^\star_*(X)$ is a quasi-isomorphism.
\end{enumerate}
We go through the proofs of Theorem \ref{main_thm} and Corollary \ref{cor_corners_general} and check that every construction preserves $C^\star_*$, $C^{g,\star}_*$ and $S^\star_*$. For $\star=\mathrm{mor}$ we write $f$ also for $|\hat f|$, and diffeomorphisms are regarded as morphisms by Lemma \ref{lem_orb_maps}(d).
\begin{enumerate}[(1)]
\item \textit{Boundary.} By Lemma \ref{lem_orb_maps}(c), $\partial$ is defined on $C^{g,\star}_*$. It is well defined on isomorphism classes, since a $2$-isomorphism $\hat f\circ\hat h\cong\hat f'$ restricts to the faces. Lemma \ref{lem_Cg} holds verbatim, since $\sigma$ is a morphism with $i_Q\circ i_{\partial Q}\circ\sigma=i_Q\circ i_{\partial Q}$.
\item \textit{Locality and functoriality.} If $\hat g\circ\hat\pi$ maps into an open subset $U\subseteq X$, then $\hat\pi$ maps into the open suborbifold $|\hat g|^{-1}(U)$ of $N$, and restrictions to open suborbifolds are again morphisms, resp. liftable. Hence Lemma \ref{lem_D_local} holds verbatim for $S^\star_*$ and $D^\star_*$, functoriality being with respect to liftable maps, resp. morphisms (Lemma \ref{lem_orb_maps}(b)).
\item \textit{Dimension axiom and homotopy invariance.} Proposition \ref{prop_point} holds verbatim; a morphism to a point is unique. Let $H:X\times[0,1]\to Y$ be liftable, resp. a morphism. By Lemma \ref{lem_orb_maps}, $H\circ(f\times\mathrm{id})$ is liftable, resp. a morphism. If $f=g\circ\pi$, then $H\circ(f\times\mathrm{id})=(H\circ(g\times\mathrm{id}))\circ(\pi\times\mathrm{id})$, and $\dim(N\times[0,1])<\dim(Q\times[0,1])$. So the proof of Proposition \ref{prop_homotopy} applies, without Lemma \ref{lem_orbifold_homotopy}. The prism operator of singular theory composes $\sigma\times\mathrm{id}$ with affine maps, so $C^\star_*$ is homotopy invariant as well.
\item \textit{Cutting.} For $\phi\in C^\infty(X)$, $\phi\circ f$ lifts locally to $\tilde\phi\circ\tilde f$, so bisecting functions, regular values and Sard's theorem are as in Section \ref{sec_cutting}. The pieces, interfaces and prisms are obtained by composing with the morphisms $j:Q_\pm\to Q$, $I\to Q$ and $\mathrm{pr}:W(Q,\rho)\to Q$ of Lemma \ref{lem_orb_maps}(c), and $f\cong g\circ\pi$ implies $f\circ j\cong g\circ(\pi\circ j)$. So $\mathrm{Cut}^\pm$ and $P$ preserve $C^{g,\star}_*$ and $S^\star_*$, and Lemma \ref{lem_cut} holds verbatim. In Lemma \ref{lem_thin_interface}, $L=(\phi\circ g)^{-1}(t)$ is given in a chart $\tilde N$ of $N$ by the invariant equation $\tilde\phi\circ\tilde g=t$, and $\tilde\pi$ maps $\tilde I=\{\tilde\phi\circ\tilde g\circ\tilde\pi=t\}$ to $\tilde L$. So $\pi|_I$ factors through $L$ by a liftable map, resp. a morphism. If $N$ is effective, so is $L$, by the covector argument in the proof of Lemma \ref{lem_orb_maps}(c) applied to $\tilde\phi\circ\tilde g$. Thus Lemma \ref{lem_thin_interface} holds for $S^\star_*$.
\item \textit{Mayer--Vietoris.} Open subsets of $X$ admit smooth Urysohn functions (Section \ref{sec_examples}). By (2) and (4), the proofs of Propositions \ref{prop_MVa} and \ref{prop_MVb} apply verbatim to $GC^\star_*$.
\item \textit{Contractible pieces.} Let $X$ have no boundary, and let $W/H$ be one of the open sets in the proof of Proposition \ref{prop_good_cover_orbifold}, an open suborbifold with chart $(W,H)$. The geodesic contraction of $W$ to $c$ is smooth and $H$-equivariant. So it defines a strict morphism $\hat K:[W/H]\times[0,1]\to[W/H]$ with $\hat K_0=\mathrm{id}$ and $\hat K_1=\hat c\circ\hat\beta$, where $\hat\beta:[W/H]\to BH$ is induced by $W\to pt$ and $\hat c:BH\to[W/H]$ by $c$.
\begin{itemize}
\item For $\star=\mathrm{orb}$, $|\hat K_1|$ is constant and factors through a point by liftable maps. So the proof of Corollary \ref{cor_contractible} applies, using (3).
\item For $\star=\mathrm{mor}$, let $z\in C^{g,\mathrm{mor}}_n(W/H)$, $n\ge1$, with $\partial z\in D^{\mathrm{mor}}_*$. By (3), $(\hat K_1)_*z-z=\partial\mathcal{H}z+\mathcal{H}\partial z$ with $\mathcal{H}\partial z\in D^{\mathrm{mor}}_*$. The terms $[Q,\hat c\circ(\hat\beta\circ\hat f)]$ of $(\hat K_1)_*z$ are thin, since they factor through $BH$, which has dimension $0$. For $n\ge2$ their faces are thin as well. For $n=1$, $Q$ is an interval or a circle (proof of Proposition \ref{prop_point}). A morphism from a point or an interval to $BH$ is unique up to $2$-isomorphism, so the two endpoints of an interval give the class $[pt,\hat c]$ with opposite signs. In both cases $(\hat K_1)_*z\in D^{\mathrm{mor}}_*$, so $z$ is a boundary in $GC^{\mathrm{mor}}_*(W/H)$. In degree $0$, $(\hat K_1)_*$ maps every chain to a multiple of $[pt,\hat c]$. This class is non-zero, because the signed count of points vanishes on boundaries and no $0$-chain is thin. Likewise, every morphism $\Delta^n\to BH$ is $2$-isomorphic to the trivial one, principal $H$-bundles over $\Delta^n$ being trivial. So $(\hat K_1)_*$ maps $C^{\mathrm{mor}}_*(W/H)$ onto the singular complex of a point.
\end{itemize}
Hence, for $W/H$, $GH^\star_*$ and $H_*(C^\star_*)$ are $\mathbb{Q}$, concentrated in degree $0$, and $p_*$ is an isomorphism; by the argument of Step 1 of the proof of Theorem \ref{main_thm}, $p_*$ is then an isomorphism for disjoint unions of such sets as well.
\item \textit{Small chains.} Iterated barycentric subdivision, and the chain homotopy between it and the identity, compose $\sigma$ with affine maps of simplices. So they preserve $C^\star_*$, and the small chains lemma holds for $C^\star_*$.
\item \textit{Orbifolds without boundary.} By Proposition \ref{prop_good_cover_orbifold}, $X$ has a good cover whose non-empty finite intersections are disjoint unions of sets as in (6). Steps 1--3 of the proof of Theorem \ref{main_thm} use only (2), (5), (6), (7) and the compact support of chains, so they prove (B). The same argument, as in Lemma \ref{lem_smooth_singular}, proves (A).
\item \textit{Corners.} The inward vector field $v$ in the proof of Corollary \ref{cor_corners_general} lifts to invariant vector fields on charts, compatible with embeddings of charts. Its local flows commute with the local groups and with the embeddings. So $r$ and the homotopies $H(x,s)=\Phi_{s\varepsilon(x)}(x)$ are induced by strict morphisms of groupoids, and in particular are liftable. By (3), the proof of Corollary \ref{cor_corners_general} reduces (A) and (B) to $X^\circ$, where they hold by (8).
\end{enumerate}
Equivariance follows from naturality and Lemma \ref{lem_orb_maps}(d).
\end{proof}

\begin{remark}[what changes]\label{rem_orb_changes}
We compare with the proofs of Theorem \ref{main_thm} and Corollary \ref{cor_corners_general}.
\begin{itemize}
\item Lemma \ref{lem_orbifold_homotopy} (Seeley, Schwarz) is not needed, since homotopies are liftable maps, resp. morphisms, on $X\times[0,1]$, and Lemma \ref{lem_orb_maps} suffices. Lemma \ref{lem_smooth_singular} is replaced by (A), which is proved in the same way.
\item In $GC^{\mathrm{mor}}_*$ the orbifold $N$ of a thin factorization may be non-effective. Of the proof above, only (6) uses this. Without it, $(\hat K_1)_*$ produces chains $[Q,\hat c\circ\hat\beta\circ\hat f]$ determined by a principal $H$-bundle over $Q$, which need not be thin: for $Q=S^1$, $X=\mathbb{R}^2/\mathbb{Z}_3$ and a bundle with monodromy $h\neq1$, the only effective orbifolds of dimension $0$ are points. This chain is nevertheless a boundary: it is $\pm\partial[D^2/\mathbb{Z}_3,\hat c\circ\hat\beta']$, where $\hat\beta'$ is given by $1\mapsto h$. We do not know whether all such chains are null-homologous when only effective $N$ are allowed. For $GC^{\mathrm{orb}}_*$ and $GC^\infty_*$, allowing non-effective $N$ changes nothing, since $N$ may be replaced by $N_{\mathrm{red}}$.
\item $GC^{\mathrm{orb}}_*$ is the quotient of the subcomplex $C^{g,\mathrm{orb}}_*\subseteq C^{g,\infty}_*$ by $D^{\mathrm{orb}}_*\subseteq D^\infty_*\cap C^{g,\mathrm{orb}}_*$. Using $D^\infty_*\cap C^{g,\mathrm{orb}}_*$ instead would give a genuine subcomplex of $GC^\infty_*$, and the proof above would apply verbatim, homotopies being smooth by Lemma \ref{lem_orbifold_homotopy}. We require liftable $g$, resp. morphisms $\hat g$, because thin factorizations must pull back forms (Remark \ref{rem_stokes}), resp. be pulled back along bundles (Section \ref{sec_operations}).
\item Everything else is unchanged: bisecting functions and Sard's theorem, prisms, Mayer--Vietoris and the colimit.
\end{itemize}
\end{remark}

\subsection{Continuous maps}\label{sec_continuous}
For continuous maps into arbitrary spaces, all results of Sections \ref{sec_chains}--\ref{sec_bisection} hold. In this subsection we show that the first half of the Mayer--Vietoris principle holds as well, and, for Euclidean neighbourhood retracts, we reduce Conjecture \ref{conj_continuous} to a locality property of degenerate chains. We first explain heuristically why the argument of Section \ref{sec_MV} does not apply directly.

\begin{remark}\label{rem_holonomy}
Let $C\subset S^1$ be a Cantor set, $\lambda\ge0$ a smooth function on $S^1$ whose zero set is exactly $C$, and $a$ the time-one map of the flow of $\lambda\,\partial_\theta$. Then $a$ fixes $C$ pointwise, and every orbit in a gap of $C$ accumulates at both endpoints of the gap. So an $a$-invariant continuous function is constant on each closed gap; if it is smooth, its derivative vanishes on the union of the gaps, which is dense, hence it is constant. On the other hand, a Cantor-type map $c:S^1\to S^1$ of degree one, collapsing each gap to a point, satisfies $c\circ a=c$.

Suppose that a geometric cycle has a codimension two corner $E\cong S^1$ on which the map is $c$. Suppose also that the isomorphisms through which the codimension one faces around $E$ cancel induce the automorphism $a$ of $(E,c)$ (the \emph{holonomy}). Then a cutting function which agrees on cancelling faces restricts to an $a$-invariant function on $E$. If it is smooth, it is constant on $E$. If moreover $\phi\circ c$ takes both values $0$ and $1$, then no constant is close to $\phi\circ c-t$. This is why Proposition \ref{prop_MVa_cont} below cuts each geometric simplex separately.
\end{remark}

Let $O$ be a metrizable space, $O=U\cup V$ with $U,V$ open, and let $\phi:O\to[0,1]$ be continuous with $\phi=0$ on $O\setminus V$ and $\phi=1$ on $O\setminus U$. Fix $t\in(0,1)$ and $0<\delta<\min(t,1-t)$. A bisecting function $\rho$ on $Q$ is \textit{adapted} to a geometric simplex $(Q,f)$ if $|\rho-(\phi\circ f-t)|<\delta$. Adapted functions exist: approximate $\phi\circ f-t$ by an invariant smooth function in charts, glue with a partition of unity, and shift by a small constant to make $0$ a regular value on all strata. For an adapted $\rho$, the pieces $Q_-$ and $Q_+$ map into $U$ and $V$ respectively.

For a geometric simplex $(Q,f)$ and a bisecting function $\rho$ on $Q$, we put $\mathrm{Cut}_\rho[Q,f]:=[Q_-,f]+[Q_+,f]$ and $P_\rho[Q,f]:=(-1)^{n+1}[W(Q,\rho),f\circ\mathrm{pr}]$. We let these operators act on the formal sum of the faces of $Q$, each face $F$ with the restriction $\rho|_F$; on prisms, $\rho$ acts through $\rho\circ\mathrm{pr}$. In this sense, Lemma \ref{lem_bisection_pieces} and Proposition \ref{prop_bisection} give
\begin{equation}\label{eq_cut_formal}
\partial\,\mathrm{Cut}_\rho=\mathrm{Cut}_\rho\,\partial,\qquad \partial P_\rho+P_\rho\partial=\mathrm{id}-\mathrm{Cut}_\rho.
\end{equation}
Only after applying them do we pass to classes in $C^g_*$.

\begin{lemma}\label{lem_two_cuts}
Let $\alpha,\beta$ be adapted bisecting functions on $(F,f)$ which are jointly transverse, i.e. on every stratum, at the points where $\alpha=\beta=0$, the differentials $d\alpha$ and $d\beta$ are linearly independent. Then, on $(F,f)$ and its faces:
\begin{enumerate}
\item $\mathrm{Cut}_\alpha\mathrm{Cut}_\beta=\mathrm{Cut}_\beta\mathrm{Cut}_\alpha$ and $P_\alpha\mathrm{Cut}_\beta=\mathrm{Cut}_\beta P_\alpha$;
\item $\Theta(\alpha,\beta):=P_\beta\mathrm{Cut}_\alpha-P_\alpha\mathrm{Cut}_\beta$ satisfies $\partial\Theta(\alpha,\beta)+\Theta(\alpha,\beta)\partial=\mathrm{Cut}_\alpha-\mathrm{Cut}_\beta$;
\item every geometric simplex occurring in $\mathrm{Cut}_\alpha(F)$ or in $\Theta(\alpha,\beta)(F)$ maps into $U$ or into $V$;
\item $P_\alpha(F)$ and $\Theta(\alpha,\beta)(F)$ are thin: their maps factor through the projection to $F$, which has smaller dimension.
\end{enumerate}
Moreover, given adapted $\alpha_1,\dots,\alpha_r$, a further adapted function $\beta$ jointly transverse to all of them exists.
\end{lemma}
\begin{proof}
(1) Both sides of the first identity are the common refinement $\{\pm\alpha\le0,\ \pm\beta\le0\}$; joint transversality makes it a valid bisection. For the second, $W(F,\alpha)\cap\{\beta\circ\mathrm{pr}\le0\}=W(\{\beta\le0\},\alpha)$. (2) follows from (1) and (\ref{eq_cut_formal}): $[\partial,P_\beta\mathrm{Cut}_\alpha]-[\partial,P_\alpha\mathrm{Cut}_\beta]=(1-\mathrm{Cut}_\beta)\mathrm{Cut}_\alpha-(1-\mathrm{Cut}_\alpha)\mathrm{Cut}_\beta$. (3) Each term contains a $\mathrm{Cut}$, and a prism maps into the image of its base. (4) is clear. For the last statement, $\{\alpha_i=0\}$ is a hypersurface in every stratum. Choose an adapted $\beta$ with $|\beta-(\phi\circ f-t)|<\delta/2$. Replace it by $\beta-s$, where $|s|<\delta/2$ is a regular value of $\beta$ on all strata of $F$ and of all hypersurfaces $\{\alpha_i=0\}$ (Sard). Then $\beta-s$ is adapted, bisecting and jointly transverse to all $\alpha_i$.
\end{proof}

\begin{proposition}[Mayer--Vietoris for continuous chains, first half]\label{prop_MVa_cont}
For $O$, $U$, $V$ as above, the inclusion $GC_*(U)+GC_*(V)\hookrightarrow GC_*(O)$ is a quasi-isomorphism.
\end{proposition}
\begin{proof}
The key observation is that a thin chain whose boundary is thin is degenerate. Therefore we do not need the cuttings of different geometric simplices to agree on common faces: the error terms are thin, and their boundaries turn out to be thin as well.

\textit{Surjectivity on homology.} Let $z=\sum_j\lambda_j[Q_j,f_j]$ with $\partial z\in D(O)$. Choose an adapted bisecting function $\rho_j$ for each $j$, with no compatibility required, and put $\mathrm{Cut}(z)=\sum_j\lambda_j\mathrm{Cut}_{\rho_j}[Q_j,f_j]$ and $P(z)=\sum_j\lambda_jP_{\rho_j}[Q_j,f_j]$. By (\ref{eq_cut_formal}), applied to each geometric simplex separately,
\begin{equation*}
z-\mathrm{Cut}(z)=\partial P(z)+\sum_y\sum_{o}\mu_o\,P_{\alpha_o}(y).
\end{equation*}
Here $y$ runs over the isomorphism classes, up to orientation, of codimension one faces of the $Q_j$. The index $o$ runs over the occurrences of $y$, with coefficient $\mu_o=\pm\lambda_j$, and $\alpha_o$ is the restriction of $\rho_j$ transported to a fixed representative $F_y$. If $(F_y,f|_{F_y})$ admits an orientation reversing automorphism $a$, the sign of $\mu_o$ depends on the chosen identification. In this case we write each occurrence $(\mu_o,\alpha_o)$ as $\frac12(\mu_o,\alpha_o)+\frac12(-\mu_o,\alpha_o\circ a^{-1})$. Since $a\times\mathrm{id}:W(F_y,\alpha_o)\to W(F_y,\alpha_o\circ a^{-1})$ is an orientation reversing isomorphism compatible with the maps, and similarly for the cut pieces, this changes none of the classes involved, and afterwards $\sum_o\mu_o=0$. Call $y$ \textit{cancelling} if $\sum_o\mu_o=0$. Otherwise $y$ occurs in $\partial z\in D(O)\subseteq S(O)$, hence is thin. For each cancelling $y$ choose an adapted $\beta_y$ jointly transverse to all $\alpha_o$ (and all $\alpha_o\circ a^{-1}$) (Lemma \ref{lem_two_cuts}), and put
\begin{equation*}
M:=\sum_{y\ \mathrm{cancelling}}\sum_o\mu_o\,\Theta(\alpha_o,\beta_y)(y),\qquad T:=\mathrm{Cut}(z)-M.
\end{equation*}
By Lemma \ref{lem_two_cuts}(3), $T\in C^g(U)+C^g(V)$. Using $\sum_o\mu_o=0$ for cancelling $y$ and Lemma \ref{lem_two_cuts}(2), we get
\begin{equation*}
\partial T=\sum_{y\ \mathrm{thin}}\sum_o\mu_o\,\mathrm{Cut}_{\alpha_o}(y)+\sum_{y\ \mathrm{cancelling}}\sum_o\mu_o\,\Theta(\alpha_o,\beta_y)(\partial y).
\end{equation*}
The first sum is thin, because cut pieces of thin simplices are thin. The second sum is thin by Lemma \ref{lem_two_cuts}(4). Since $\partial(\partial T)=0$, we get $\partial T\in D(O)$, so $T$ is a cycle. Moreover $z-T=\partial P(z)+\Lambda$ with
\begin{equation*}
\Lambda:=\sum_y\sum_o\mu_oP_{\alpha_o}(y)+M.
\end{equation*}
$\Lambda$ is thin by Lemma \ref{lem_two_cuts}(4), and $\partial\Lambda=\partial z-\partial T\in S(O)$. Hence $\Lambda\in D(O)$, and $z$ is homologous to $T$.

\textit{Injectivity on homology.} Let $c\in C^g(U)+C^g(V)$ with $\partial c\in D(O)$ and $c-\partial w\in D(O)$ for some $w\in C^g(O)$. Apply the construction above to the geometric simplices of $w$. This gives $T_w\in C^g(U)+C^g(V)$, with $\partial T_w$ equal to a thin chain plus $\sum_{y}\sum_o\mu_o\,\mathrm{Cut}_{\alpha_o}(y)$, the sum running over the non-cancelling face classes $y$ of $w$ which are not thin. Each such $y$ occurs in $c$, since $\partial w-c\in D(O)$, so $y$ maps into $U$ or into $V$. Put
\begin{equation*}
w':=T_w+\sum_{y}\sum_o\mu_o\,P_{\alpha_o}(y),
\end{equation*}
with $y$ running over the same classes; then $w'\in C^g(U)+C^g(V)$. By (\ref{eq_cut_formal}), $\partial w'=c+\tau$ with
\begin{multline*}
\tau=-c_{\mathrm{thin}}+\sum_{y\ \mathrm{thin}}\sum_o\mu_o\mathrm{Cut}_{\alpha_o}(y)\\+\sum_{y\ \mathrm{cancelling}}\sum_o\mu_o\Theta(\alpha_o,\beta_y)(\partial y)-\sum_{y}\sum_o\mu_oP_{\alpha_o}(\partial y),
\end{multline*}
where the last sum runs over the non-thin non-cancelling classes, and $c_{\mathrm{thin}}$ is the sum of the thin terms of $c$. Every term is thin. Since $\partial\tau=-\partial c\in S(O)$, we get $\tau\in D(O)$, so $c=\partial w'$ in $GC(U)+GC(V)$.
\end{proof}

The second half of the Mayer--Vietoris principle concerns the complex
\begin{equation*}
\mathcal{K}_*:=\frac{D_*(O)\cap\big(C^g_*(U)+C^g_*(V)\big)}{D_*(U)+D_*(V)},
\end{equation*}
which measures to what extent degenerate chains are \textit{local}.

\begin{proposition}\label{prop_reduction}
Let $O$, $U$, $V$ be as above, and suppose that $p_*$ is an isomorphism for $U$, $V$ and $U\cap V$. Then $p_*:H_*(O;\mathbb{Q})\to GH_*(O)$ is an isomorphism if and only if $\mathcal{K}_*$ is acyclic. Moreover $\mathcal{K}_0=\mathcal{K}_1=0$.
\end{proposition}
\begin{proof}
Put $GC^{\mathcal{U}}(O):=(C^g(U)+C^g(V))/(D(U)+D(V))$. The sequence
\begin{equation*}
0\to GC(U\cap V)\to GC(U)\oplus GC(V)\to GC^{\mathcal{U}}(O)\to0
\end{equation*}
is exact: if $a-b=d_U+d_V$ with $d_U\in D(U)$, $d_V\in D(V)$, then $a-d_U=b+d_V\in C^g(U)\cap C^g(V)=C^g(U\cap V)$, and injectivity on the left is Lemma \ref{lem_D_local}. Comparing with the Mayer--Vietoris sequence of $\mathcal{U}$-small singular chains, the five lemma shows that $p$ induces an isomorphism $H_*(O)\cong H_*(GC^{\mathcal{U}}(O))$. The map $GC^{\mathcal{U}}(O)\to GC(O)$ is the surjection onto $GC(U)+GC(V)$, whose kernel is $\mathcal{K}$, followed by the inclusion, which is a quasi-isomorphism by Proposition \ref{prop_MVa_cont}. This proves the equivalence. Finally, no $0$-dimensional geometric simplex is thin, so $D_0=0$. A thin $1$-dimensional geometric simplex has constant map, so it lies in $U$ or in $V$, and it is degenerate by itself. Hence $\mathcal{K}_1=0$.
\end{proof}
By the induction in the proof of Theorem \ref{main_thm}, the passage to colimits and the retraction argument, Conjecture \ref{conj_continuous} for Euclidean neighbourhood retracts therefore follows from the acyclicity of $\mathcal{K}_*$ for all open covers $\{U,V\}$ of open subsets of Euclidean spaces. The argument works degree by degree.

\begin{theorem}\label{thm_low_degrees}
Let $k\ge0$, and suppose that $H_j(\mathcal{K}_*)=0$ for all $j\le k$ and all open covers $\{U,V\}$ of open subsets of Euclidean spaces. Then for every Euclidean neighbourhood retract $X$, the map $p_*:H_j(X;\mathbb{Q})\to GH_j(X;\mathbb{Q})$ is an isomorphism for $j\le k$ and surjective for $j=k+1$. In particular, since $\mathcal{K}_0=\mathcal{K}_1=0$, $p_*$ is an isomorphism in degrees $\le1$ and surjective in degree $2$.
\end{theorem}
\begin{proof}
By the long exact sequence of $0\to\mathcal{K}\to GC^{\mathcal{U}}(O)\to GC(U)+GC(V)\to0$ and Proposition \ref{prop_MVa_cont}, the map $H_j(GC^{\mathcal{U}}(O))\to GH_j(O)$ is an isomorphism for $j\le k$ and surjective for $j=k+1$. We argue by induction on the number of convex sets in a finite union, as in the proof of Theorem \ref{main_thm}. The five lemma gives $H_j(O)\cong H_j(GC^{\mathcal{U}}(O))$ for $j\le k$, and the four lemma gives surjectivity for $j=k+1$, provided the same statements hold for $U$, $V$ and $U\cap V$. For convex sets they hold by Corollary \ref{cor_contractible}. Isomorphisms and surjections are preserved by colimits over increasing unions and by retracts.
\end{proof}

The next degree is governed by the following question.

\begin{question}\label{q_dimension}
Let $K$ be a compact subset of a finite disjoint union $\mathcal{N}$ of orbifolds, possibly with corners and not necessarily compact, of dimension at most $m$. For $e=1,\dots,r$, let $F_e$ be a compact orbifold with corners of dimension at most $m$ and $\gamma_e:F_e\to\mathcal{N}\times\mathcal{N}$ a smooth map with image in $K\times K$. Let $R$ be the smallest closed equivalence relation on $K$ containing $\gamma_1(F_1),\dots,\gamma_r(F_r)$. Is the covering dimension of $\hat Y:=K/R$ at most $m$?
\end{question}

\begin{proposition}\label{prop_K2}
Suppose that Question \ref{q_dimension} has a positive answer for $m=1$. Let $O$ be a metrizable absolute neighbourhood retract (for instance an open subset of a Euclidean space, or a manifold) and $O=U\cup V$ with $U,V$ open. Then $\mathcal{K}_2=0$.
\end{proposition}
\begin{proof}
\textit{Reduction.} Let $c\in D_2(O)\cap(C^g(U)+C^g(V))$ and put $W:=U\cap V$. Since $S_*(O)$ is spanned by a subset of the basis of $C^g_*(O)$, every basis element occurring in $c$ is thin. Let $a$ be the sum of the terms of $c$ that lie in $C^g(U)$, and put $b:=c-a\in C^g(V)$. By Lemma \ref{lem_D_local}, $a\in S_2(U)$ and $b\in S_2(V)$. Let $w$ be the sum of the non-thin terms of $\partial a$, and $s:=\partial a-w\in S_1(U)$. Since $\partial a+\partial b=\partial c\in S_1(O)$, every non-thin basis element occurring in $\partial a$ also occurs in $\partial b$, hence lies in $C^g(U)\cap C^g(V)=C^g(W)$. Thus $w\in C^g_1(W)$. It suffices to find $t\in S_2(W)$ with $\partial t-w\in S_1(W)$. Indeed, then $a-t\in D_2(U)$, $b+t\in D_2(V)$, and $c=(a-t)+(b+t)$.

\textit{The configuration.} Replacing $c$ by $Nc$ for a common denominator $N$ and reversing orientations where necessary, write $a=\sum_{i=1}^r[Q_i,f_i]$, with repetitions allowed. Choose thin factorizations $f_i=g_i\circ\pi_i$ with $\pi_i:Q_i\to N_i$ and $\dim N_i\le1$. Choose a maximal set of disjoint pairs and self-paired faces of the following two kinds:
\begin{itemize}
\item pairs $\{F,F'\}$ together with an orientation reversing diffeomorphism $h:F\to F'$ such that $f_j\circ h=f_i|_F$, where $F$ is a face of $Q_i$ and $F'$ a face of $Q_j$;
\item single faces $F$ together with an orientation reversing automorphism $h$ of $(F,f_i|_F)$.
\end{itemize}
By maximality, the classes of the remaining faces add up to $\partial a$. So a remaining face is either thin, and then its map is constant, or it is one of the terms of $w$.

\textit{The quotient.} Let $K$ be the disjoint union of the compact sets $\pi_i(Q_i)$ and of one point $p_F$ for each remaining thin face $F$. Let $R$ be the smallest closed equivalence relation on $K$ containing
\begin{itemize}
\item the sets $\{(\pi_i(x),\pi_j(h(x))):x\in F\}$, for the faces of the first two kinds;
\item the sets $\{(\pi_i(x),p_F):x\in F\}$, for the remaining thin faces.
\end{itemize}
These are images of smooth maps from compact $1$-dimensional orbifolds. Let $q:K\to\hat Y:=K/R$ be the quotient map. Since $K$ is compact metrizable and $R$ is closed, $\hat Y$ is compact metrizable. The $f_i$ agree along paired faces and are constant on thin faces. Hence the maps $g_i$, together with $p_F\mapsto f_i(F)$, descend to a continuous map $G:\hat Y\to U$ with $G\circ q\circ\pi_i=f_i$. By assumption, $\dim\hat Y\le1$.

\textit{Extension.} Let $A\subseteq\hat Y$ be the union of the sets $q(\pi_i(F))$ over the faces $F$ contributing to $w$. Then $A$ is compact and $G(A)\subseteq W$. Hence $G(A)$ meets only finitely many path components $W_1,\dots,W_l$ of $W$, and $A$ is the disjoint union of the sets $A_\nu:=A\cap G^{-1}(W_\nu)$, which are open and closed in $A$. Fix $\nu$ and a point $x_\nu\in W_\nu$, and define $e_\nu:A\to W_\nu$ by $e_\nu=G$ on $A_\nu$ and $e_\nu=x_\nu$ on $A\setminus A_\nu$. The set $W_\nu$ is an open subset of an absolute neighbourhood retract, hence a path-connected and locally path-connected absolute neighbourhood retract, i.e. it is $C^0$ and $LC^0$. Since $\dim\hat Y\le1$, the extension theorem of Kuratowski \cite{ku35} and Dugundji \cite[Theorem~9.1]{du58} shows that $e_\nu$ extends to a continuous map $E_\nu:\hat Y\to W_\nu$.

\textit{The filling.} Put $t_\nu:=\sum_i[Q_i,E_\nu\circ q\circ\pi_i]$. Each term is thin, because the restriction of $E_\nu\circ q$ to $\pi_i(Q_i)$ extends to an open neighbourhood in $N_i$, $W_\nu$ being an absolute neighbourhood retract. The faces of $t_\nu$ behave as follows.
\begin{itemize}
\item Faces of the first two kinds cancel or vanish, because $q\circ\pi_i=q\circ\pi_j\circ h$ on them.
\item Remaining thin faces still have constant maps.
\item A face contributing to $w$ keeps its map $G\circ q\circ\pi_i=f_i$ if its image lies in $A_\nu$, and gets a constant map otherwise.
\end{itemize}
Hence $\partial t_\nu-w_\nu$ is thin, where $w_\nu$ is the sum of the terms of $w$ that map into $W_\nu$. Then $t:=\frac1N\sum_\nu t_\nu$ has the required property for the original $c$.
\end{proof}

\begin{corollary}\label{cor_degree_two}
Suppose that Question \ref{q_dimension} has a positive answer for $m=1$. Then for every Euclidean neighbourhood retract $X$, the map $p_*$ is an isomorphism in degrees $\le2$ and surjective in degree $3$.
\end{corollary}
\begin{proof}
Combine Propositions \ref{prop_reduction} and \ref{prop_K2} with Theorem \ref{thm_low_degrees} for $k=2$.
\end{proof}

We now settle Question \ref{q_dimension} for $m=1$ in a large class of cases. A continuous map $\gamma=(\alpha,\beta):[0,1]\to[0,1]^2$ is \textit{piecewise monotone} if $[0,1]$ is a finite union of closed intervals (\textit{laps}) on each of which both $\alpha$ and $\beta$ are strictly monotone or constant. For instance, real-analytic maps are piecewise monotone. For general $\gamma$, let $\Sigma^*(\gamma)$ be the closed set of points of $[0,1]$ that have no neighbourhood on which $\gamma$ is piecewise monotone, i.e. the points where laps accumulate. If $\gamma$ is $C^1$, then $\Sigma^*(\gamma)$ is contained in the set of critical points of $\alpha$ and $\beta$.

\begin{proposition}\label{prop_dim_laps}
Let $\gamma_1,\dots,\gamma_r:[0,1]\to[0,1]^2$ be continuous, let $R$ be the smallest closed equivalence relation on $[0,1]$ containing their images, $Y=[0,1]/R$ and $q:[0,1]\to Y$ the quotient map.
\begin{enumerate}[(a)]
\item If every $\gamma_e$ is piecewise monotone, then $Y$ is the union of a finite set and countably many open sets, each homeomorphic to a subset of $\mathbb{R}$. In particular $\dim Y\le1$.
\item In general, $\dim Y\le1$ if and only if $\dim q(A^*)\le1$, where $A^*:=\bigcup_e\alpha_e(\Sigma^*(\gamma_e))$. In particular $\dim Y\le1$ if every $\Sigma^*(\gamma_e)$ is countable.
\item If $r=1$ and $\gamma_1=(\mathrm{id},T)$ for an arbitrary continuous map $T:[0,1]\to[0,1]$, then $Y$ is a point or an arc.
\end{enumerate}
\end{proposition}
The proof is given in Appendix \ref{app_intervals}.

Parts (a) and (b) of Proposition \ref{prop_dim_laps} remain valid in the setting of Question \ref{q_dimension} for $m=1$. Only finitely many components of $\mathcal{N}$ meet $K$, and the underlying space of each is an interval or a circle. Cut each such circle at a point $p$, outside $K$ unless the circle is contained in $K$, that is a regular value of all components of the generators, embed the resulting intervals disjointly in $[0,1]$, and extend $R$ by the identity on the complement of $K$, which changes the quotient only by adding open intervals. Subdivide the domains $F_e$ at the finitely many preimages of the cut points, cut the remaining circle components of the $F_e$ at an arbitrary point, and, for each cut circle contained in $K$, add the constant generator identifying its two ends. In the proof of Proposition \ref{prop_K2}, the generators are the pairs $(\pi_i|_F,\pi_j\circ h|_F)$ and $(\pi_i|_F,p_F)$ on $1$-dimensional faces. Hence the construction of Proposition \ref{prop_K2} works without further assumptions for those $c$ whose terms admit thin factorizations for which these generators are piecewise monotone (for example, real-analytic), or more generally have countable sets $\Sigma^*$. This does not yet give $\mathcal{K}_2=0$.

\begin{remark}[the kernel of $p_*$ and thin cycles]\label{rem_thin_cycles}
Call a singular simplex $\sigma:\Delta^n\to X$ \textit{thin} if $\sigma=g\circ\pi$, where $\pi:\Delta^n\to N$ is smooth in the orbifold sense, $N$ is an orbifold (possibly with corners, not necessarily compact) of dimension $<n$, and $g$ is continuous. If $z$ is a singular cycle made of thin simplices, then $p(z)$ is thin with zero boundary, so $p(z)\in D_n(X)$ and $p_*[z]=0$. Conversely, we sketch why every class in $\ker p_*$ is of this form. We use the fundamental chains of Definition \ref{def_fundamental_chains} and Lemma \ref{lem_fundamental_chains} in Appendix \ref{app_fundamental}; the proof of that lemma does not use triangulations of $Q$. For every compact oriented orbifold with corners $Q$ we fix $\mu_Q$ and $\nu_Q$ as in Lemma \ref{lem_fundamental_chains}(c). Thus $\mu_Q$ is a fundamental chain of $Q$, $\partial\mu_Q=(i_Q)_*\nu_Q$, and the restriction of $\nu_Q$ to every component $F$ of $\partial Q$ is a fundamental chain of $F$.

\textit{Sketch.}
\begin{itemize}
\item Suppose $p(s)=\partial w+d$ with $d\in D_n(X)$. Choose a representative $(Q,f)$ of each basis class of Lemma \ref{lem_Cg} occurring in $w$, in $d$, or as an iterated face of these, and put $\varphi(Q,f):=f_*\mu_Q$.
\item Put $\varphi:=0$ on classes admitting an orientation reversing automorphism $a$. This is consistent, since $-a_*\mu_Q$ is again a fundamental chain, so $f_*\mu_Q=\frac12f_*(\mu_Q+a_*\mu_Q)$ is a thin chain plus the boundary of a thin chain, by the next item.
\item By Lemma \ref{lem_fundamental_chains}(d), two fundamental chains of the same $F$ differ by $\partial\beta+(i_F)_*\gamma$, with $\beta\in C^{\mathrm{orb}}_{\dim F+1}(F)$ and $\gamma\in C^{\mathrm{orb}}_{\dim F}(\partial F)$. (By Lemma \ref{lem_fundamental_chains}(a) the relevant homology is that of the pair $(|F|,B_F)$, but the correction term is $(i_F)_*\gamma$ with $\gamma$ a chain on $\partial F$, not merely a chain on the subspace $B_F$. When faces are not embedded, chains on $B_F$ need not lift to $\partial F$; see Example \ref{ex_teardrop_chains}.) After push-forward, $\beta$ and $(i_F)_*\gamma$ consist of thin simplices, since they factor through $F$, resp. $\partial F$, whose dimensions are smaller than their degrees; so the difference is a thin chain plus the boundary of a thin chain.
\item Hence $\partial\varphi(w)-\varphi(\partial w)$ is a thin chain plus the boundary of a thin chain: each occurrence of a face is compared with the chosen representative by the previous item, whether or not it cancels in $\partial w$. No compatibility of fundamental chains or triangulations along cancelling faces is needed, and holonomy plays no role.
\item Likewise $s-\varphi(p(s))$ is a thin chain plus the boundary of a thin chain: if $s'$ is a simplex of $s$, we may take the chosen representative of its class to be $(\Delta^n,s_0)$, so that $s'=s_0\circ h$ for an orientation preserving diffeomorphism $h$ of $\Delta^n$; then the identity of $\Delta^n$ and $h$ are both fundamental chains of $\Delta^n$. Moreover $\varphi(d)$ is thin.
\item Altogether $s-\partial\big(\varphi(w)+\theta\big)=z$, for a chain $\theta$ of thin simplices and a cycle $z$ of thin simplices.
\end{itemize}
Thus $\ker p_*$ consists exactly of the classes of thin singular cycles. In the smooth theory on manifolds these classes vanish by Remark \ref{rem_stokes}, and on orbifolds by Theorem \ref{main_thm}. In the continuous theory, thin simplices can have space-filling images.

Let $X$ be a metrizable absolute neighbourhood retract, and suppose that Question \ref{q_dimension} has a positive answer for $m=n-1$. Then thin singular $n$-cycles are null-homologous.
\begin{itemize}
\item Replacing $z$ by an integer multiple, write $z=\sum_a\varepsilon_a\,g_a\circ\pi_a$ with $\varepsilon_a=\pm1$, repetitions allowed, $\pi_a:\Delta^n\to N_a$, $\dim N_a\le n-1$. Since $\partial z=0$, the face occurrences $(a,j)$ can be matched in pairs $\{(a,j),(b,l)\}$ with opposite signs and $g_a\circ\pi_a\circ d_j=g_b\circ\pi_b\circ d_l$.
\item Let $K:=\bigsqcup_a\pi_a(\Delta^n)\subseteq\bigsqcup_aN_a$. Let $R$ be the smallest closed equivalence relation on $K$ containing the images of the smooth maps $(\pi_a\circ d_j,\pi_b\circ d_l):\Delta^{n-1}\to K\times K$ for the matched pairs. This is the setting of Question \ref{q_dimension} with $m=n-1$, so $\hat Y:=K/R$ is a compact metrizable space with $\dim\hat Y\le n-1$.
\item Put $g:=\bigsqcup_ag_a|_{\pi_a(\Delta^n)}$. The relation $\{(y,y'):g(y)=g(y')\}$ is a closed equivalence relation on $K$ ($X$ being Hausdorff) containing the generators, hence containing $R$. So $g$ descends to a continuous map $G:\hat Y\to X$ with $G\circ q\circ\pi_a=g_a\circ\pi_a$, where $q:K\to\hat Y$ is the quotient map.
\item The chain $\tilde z:=\sum_a\varepsilon_a\,q\circ\pi_a$ is a singular cycle in $\hat Y$, because matched faces become equal after composing with $q$, and $G_*\tilde z=z$.
\item Embed $X$ as a closed subset of a convex subset $C$ of a normed space (Kuratowski--Wojdys\l{}awski), and let $r:O'\to X$ be a retraction of an open neighbourhood $O'\subseteq C$. Choose $\varepsilon>0$ with $B_C(G(y),\varepsilon)\subseteq O'$ for all $y\in\hat Y$, and a finite open cover $\mathcal U$ of $\hat Y$ of order $\le n-1$ with $\operatorname{diam}G(U)<\varepsilon$ for $U\in\mathcal U$ (by uniform continuity of $G$ and \cite[Theorem~1.6.12]{en78}).
\item Let $\kappa:\hat Y\to P:=|\mathcal N(\mathcal U)|$ be the canonical map, and $\psi(\sum_Ut_Uv_U):=\sum_Ut_U\,G(y_U)$ with $y_U\in U$. The homotopy $r\big((1-s)G+s\,\psi\kappa\big)$ shows $G\simeq r\psi\kappa$. Since $\dim P\le n-1$, $[z]=(r\psi)_*\kappa_*[\tilde z]=0$.
\end{itemize}
One cannot argue with $H_n(\hat Y)$ directly: the singular homology of a compactum can be non-zero above its dimension \cite{bm62}.

Together with the sketch above, this would show that $p_*:H_n(X;\mathbb{Q})\to GH_n(X)$ is injective for every metrizable absolute neighbourhood retract $X$ as soon as Question \ref{q_dimension} has a positive answer for $m=n-1$, without any information on $\mathcal K_*$. In that case, consider the exact sequence $H_{n+1}(GC^{\mathcal U}(O))\to GH_{n+1}(O)\to H_n(\mathcal K_*)\to H_n(GC^{\mathcal U}(O))\to GH_n(O)$ from the proof of Proposition \ref{prop_reduction}, where $H_*(GC^{\mathcal U}(O))\cong H_*(O)$ if $p_*$ is an isomorphism for $U$, $V$, $U\cap V$.
\begin{itemize}
\item The map $H_n(\mathcal K_*)\to H_n(O)$ vanishes, so $H_n(\mathcal K_*)$ is the cokernel of $p_*$ in degree $n+1$.
\item Hence the extension obstruction of Remark \ref{rem_locality_problem}(2), which obstructs $\mathcal K_n=0$, can only affect surjectivity of $p_*$ in degree $n+1$.
\item In item (1) of that remark, it is injectivity in degree $2$ that would fail. This is excluded by the present argument as well.
\end{itemize}
\end{remark}

\begin{remark}[the locality problem]\label{rem_locality_problem}
Question \ref{q_dimension} is open. For $m=1$ it is answered positively in the cases of Proposition \ref{prop_dim_laps}; the case $m\ge2$ is not treated here. Consequently $\mathcal{K}_2=0$ is known only conditionally (Proposition \ref{prop_K2}), and the acyclicity of $\mathcal{K}_*$ in degrees $\ge3$ is open.
\begin{enumerate}
\item \textit{Why degree $2$ matters.} Cover $S^2$ by open neighbourhoods $U,V$ of the two closed hemispheres. Suppose there were a degenerate $2$-chain $e_U+e_V$ with $e_U\in S_2(U)$, $e_V\in S_2(V)$ and $\partial e_U\equiv\omega\equiv-\partial e_V$ modulo thin chains, for a non-thin $1$-cycle $\omega$ in $U\cap V$ representing the generator of $H_1(U\cap V)$. By Theorem \ref{thm_low_degrees} in degree $1$, such an element is non-zero in $\mathcal{K}_2$, and the fundamental class of $S^2$ would vanish in $GH_2(S^2)$. By Proposition \ref{prop_K2}, this cannot happen if Question \ref{q_dimension} has a positive answer for $m=1$.
\item \textit{Higher degrees.} The proof of Proposition \ref{prop_K2} extends as follows. The only change is that a remaining thin face $F$ with factorization $f_i|_F=g_F\circ\pi_F$ contributes $\pi_F(F)$ to $K$ and the generator $(\pi_i|_F,\pi_F)$, instead of a point $p_F$. Suppose that Question \ref{q_dimension} has a positive answer for $m=n-1$, and that the path components of $U\cap V$ are $(n-2)$-connected. Then $\mathcal{K}_n=0$, because maps from closed subsets of metrizable spaces of dimension $\le n-1$ into $(n-2)$-connected absolute neighbourhood retracts extend \cite[Theorem~9.1]{du58}. (An ANR is $LC^{n-2}$, and $(n-2)$-connected means $C^{n-2}$.) In general the extension of $e_\nu$ is obstructed by classes in $\check H^{j+1}(\hat Y,A;\pi_j(W_\nu))$, $1\le j\le n-2$, with local coefficients if $W_\nu$ is not simple. We do not know whether these obstructions are realized by degenerate chains.
\item \textit{On Question \ref{q_dimension}.} By Proposition \ref{prop_dim_laps}, for $m=1$ the only open case is that of curves whose laps accumulate on an uncountable set. A single graph $(\mathrm{id},T)$ is settled by part (c).
\begin{itemize}
\item For curves meeting the diagonal, two self-maps suffice. Splice all such curves along the diagonal into one curve $(\alpha,\beta)$ with values in $[0,\frac12]$ and starting at $(\frac12,\frac12)$. Let $r_1,r_2$ be the identity on $[0,\frac12]$ and equal to $\alpha$, resp. $\beta$, on $[\frac12,1]$. Then $x\sim r_1(x)$, $x\sim r_2(x)$ give the same quotient. Curves not meeting the diagonal are not covered by this reduction.
\item The proof of Proposition \ref{prop_dim_laps}(c) extends to finitely many pairwise commuting self-maps. We only sketch this: the minimal ideal of the compact abelian enveloping semigroup is a group, and it acts on $\Lambda=u[0,1]$ through $\{\mathrm{id},\sigma\}$. For general pairs, for example the two retractions above, the enveloping semigroup is right-zero, and the argument gives no information.
\item In the setting of Proposition \ref{prop_dim_laps}, if every $\gamma_e$ meets the diagonal, every map $g:Y\to S^1$ is null-homotopic. Indeed, lift $g\circ q$ to $\Phi:[0,1]\to\mathbb{R}$; then $\Phi\circ\beta_e-\Phi\circ\alpha_e$ is an integer constant, which vanishes at a point on the diagonal. Hence $\{\Phi(x)=\Phi(y)\}$ is a closed equivalence relation containing $R$, and $\Phi$ descends to a lift of $g$. So $Y$, a Peano continuum, satisfies $\dim Y\le1$ if and only if it contains no simple closed curve. If $\dim Y\le1$, the identity of a simple closed curve would extend to a map $Y\to S^1$ \cite[Theorem~3.2.10]{en78}, which is null-homotopic. Conversely, a Peano continuum without simple closed curves is a dendrite \cite[Ch.~X]{nadler}. Dendrites are regular, i.e.\ every point has arbitrarily small open neighbourhoods with finite boundary \cite[Ch.~X]{nadler}, so $\operatorname{ind}Y\le1$, hence $\dim Y\le1$ \cite[Theorem~1.7.7]{en78}.
\end{itemize}
Two further remarks show that the question is delicate.
\begin{itemize}
\item Finiteness of the number of generators is essential. Every closed equivalence relation is generated by a countable dense subset, so countably many constant generators can produce $\hat Y\cong[0,1]^2$, for example from the fibres of a space-filling curve.
\item Metric arguments cannot work. In Remark \ref{rem_holonomy}, take the Cantor set $C$ of Hausdorff dimension $0$, which is possible since every closed subset of a manifold is the zero set of a smooth function. Take $K=S^1$, with $R$ generated by the graph of $a$. Then $K/R\cong S^1$, but for no compatible metric is the quotient map H\"older continuous.
\end{itemize}
\end{enumerate}
\end{remark}

\subsection{Operations on geometric chains}\label{sec_operations}
As stated in the introduction, geometric chains are well behaved under natural operations. We discuss two of them.
\begin{enumerate}
\item (Cartesian product).
Let $X,Y$ be topological spaces. There is a product $GC_*(X) \otimes GC_*(Y) \rightarrow GC_*(X\times Y)$, defined by
\begin{equation*}
[P,f_X]\otimes[Q,f_Y]\mapsto[P\times Q,\ f_X\times f_Y],
\end{equation*}
where the Cartesian product $P\times Q$ of orbifolds with corners carries the induced orientation and is again an orbifold with corners. It satisfies the Leibniz rule $\partial(x\times y)=\partial x\times y+(-1)^{|x|}x\times\partial y$, since $\partial(P\times Q)=\partial P\times Q\sqcup(-1)^{\dim P}P\times\partial Q$ as oriented orbifolds. It preserves degenerate chains: if $x$ is thin, then $x\times y$ is thin, and if $x\in D$, then both terms of $\partial(x\times y)$ are thin. The same formula defines $GC^\infty_*(X)\otimes GC^\infty_*(Y)\to GC^\infty_*(X\times Y)$ whenever $X\times Y$ carries a smooth structure for which products of smooth maps are smooth, for instance for orbifolds with corners.
\item (Pullback).
Let $f: Y \rightarrow X$ be an oriented smooth fibre bundle over a smooth manifold $X$ without boundary, whose fibre $F$ is a compact oriented smooth manifold, possibly with corners. For a smooth geometric simplex $x: P \rightarrow X$, the fibre product $R:=P\times_X Y$ is a compact orbifold with corners, locally $P\times F$, oriented base first, and we put $f^*[P,x]:=[R,\mathrm{pr}_Y]$. This defines $f^*: GC^\infty_*(X) \rightarrow GC^\infty_{*+\dim F}(Y)$.
\begin{itemize}
\item If $x=g\circ\pi$ is thin, then $R\to N\times_X Y$ exhibits $f^*[P,x]$ as thin.
\item We have $\partial f^*x=f^*\partial x+(-1)^{|x|}\,(i_Y)_*(f|_{\partial})^*x$, where $f|_\partial=f\circ i_Y:\partial Y\to X$ is the bundle with fibre $\partial F$ (Convention \ref{conv_faces}). In particular $f^*$ is a chain map if $\partial F=\emptyset$.
\item If $x\in D^\infty_*(X)$, then $x$ and $\partial x$ are thin, so $f^*x$, $f^*\partial x$ and $(f|_\partial)^*x$ are thin. By the boundary formula $f^*x\in D^\infty_*(Y)$, so $f^*$ preserves degenerate chains.
\end{itemize}
\textit{Pullback over orbifolds.} Let $X$ be an orbifold with corners and $f:Y\to X$ an oriented orbifold fibre bundle with fibre $F$ as above. That is, over every chart $(\tilde V,H)$ of $X$, $f$ is the projection $H\ltimes(\tilde V\times F)\to H\ltimes\tilde V$, with $H$ acting on $F$ by orientation preserving diffeomorphisms, compatibly with embeddings of charts. The pullback is defined on the complex $GC^{\mathrm{mor}}_*$ of Definition \ref{def_orb_chains}. For $[P,\hat x]\in C^{g,\mathrm{mor}}_*(X)$ let $R:=P\times_XY$ be the fibre product of groupoids \cite{mm}. Where $\hat x$ is given by $(\tilde x,\lambda)$ on a chart $(\tilde U,G)$ (Lemma \ref{lem_orb_maps}(a)), $R$ is represented by the action groupoid of $G$ on $\tilde U\times F$, with $G$ acting on $F$ through $\lambda$. Near a point $(\tilde u,y)$ with stabiliser $G'$, the product of a $G'$-invariant neighbourhood of $\tilde u$ and a linear chart of $F$ at $y$ for $G'$ (Lemma \ref{lem_doubling}(1)) is a chart in the sense of Definition \ref{orbi_corners}; its local group may permute the faces of $F$, as that definition allows. So $R$ is a compact orbifold with corners, effective because $G$ acts effectively on $\tilde U$, and oriented base first. Its boundary is $\partial P\times_XY\sqcup P\times_X\partial_vY$, where $\partial_vY\to X$ is the bundle with fibre $\partial F$. Put $f^*[P,\hat x]:=[R,\widehat{\mathrm{pr}}_Y]$.
\begin{itemize}
\item An isomorphism of sources and a $2$-isomorphism of morphisms induce an isomorphism of fibre products over $Y$, so $f^*$ is well defined on $C^{g,\mathrm{mor}}_*(X)$.
\item If $\hat x\cong\hat g\circ\hat\pi$ through $N$, then $R\cong P\times_N(N\times_XY)$, and $\widehat{\mathrm{pr}}_Y$ factors through $N\times_XY$, an orbifold with corners in the wide sense of dimension $\dim N+\dim F<\dim R$. So $f^*$ preserves $S^{\mathrm{mor}}_*$. The boundary formula above holds with the same proof, with $\partial_vY$ and its inclusion into $Y$ in place of $\partial Y$ and $i_Y$ (if $\partial X\neq\emptyset$, then $\partial Y\neq\partial_vY$), so $f^*$ preserves $D^{\mathrm{mor}}_*$.
\end{itemize}
Hence $f^*:GC^{\mathrm{mor}}_*(X)\to GC^{\mathrm{mor}}_{*+\dim F}(Y)$ is defined, and it is a chain map if $\partial F=\emptyset$. By Proposition \ref{prop_orb_comparison}, these complexes compute $H_*(X;\mathbb{Q})$ and $H_*(Y;\mathbb{Q})$. The definition uses two features of $GC^{\mathrm{mor}}_*$. First, $f^*[P,\hat x]$ depends on $\hat x$ and not only on $|\hat x|$ (Example \ref{ex_lift_not_unique}). Second, thin factorizations consist of morphisms; for a factorization through liftable maps only, $N\times_XY$ is not defined. The fibre product of underlying spaces is in general not an orbifold. For example, take $Y=(\mathbb{R}^2\times S^1)/\mathbb{Z}_2\to\mathbb{R}^2/\mathbb{Z}_2$, with $\mathbb{Z}_2$ acting by $-1$ on $\mathbb{R}^2$ and by rotation by $\pi$ on $S^1$, and $x(t)=[(t,0)]$, $t\in[-1,1]$. The fibre of $|Y|\to|X|$ over $x(t)$ is $S^1$ for $t\neq0$ and $S^1/\mathbb{Z}_2$ for $t=0$, so $|P|\times_{|X|}|Y|$ is not a manifold near $t=0$. For a manifold $X$ and merely continuous $x$, the fibre product is only a topological fibre bundle over $|P|$, and the definition requires additional care, which we do not address here.
\end{enumerate}

\begin{remark}[twisted coefficients]\label{rem_twisted}
Let a finite group $\varrho$ act on an orbifold with corners $M$ by diffeomorphisms, and let $L$ be a one-dimensional representation of $\varrho$ over $\mathbb{Q}$. Then $\varrho$ acts on $GC^\infty_*(M)$ by $g[Q,f]=[Q,g\circ f]$, $g\in\varrho$, preserving $D^\infty_*$, and diagonally on $GC^\infty_*(M)\otimes_{\mathbb{Q}}L$. We put
\begin{equation*}
GC^\infty_*(M/\varrho;\mathbb{Q}^\epsilon):=\big(GC^\infty_*(M)\otimes_{\mathbb{Q}}L\big)_\varrho .
\end{equation*}
Taking coinvariants of a finite group is exact over $\mathbb{Q}$, and $p$ is $\varrho$-equivariant. Hence Corollary \ref{cor_corners_general} gives $GH^\infty_*(M/\varrho;\mathbb{Q}^\epsilon)\cong(H_*(M;\mathbb{Q})\otimes L)_\varrho$, the rational homology of the orbifold $M/\varrho$ with coefficients in the orbifold local system $\mathbb{Q}^\epsilon$ defined by $L$ (if $\varrho$ acts freely, the homology of the quotient space with local coefficients). Since diffeomorphisms underlie unique orbifold morphisms (Lemma \ref{lem_orb_maps}(d)), the same holds for $GC^{\mathrm{orb}}_*$ and $GC^{\mathrm{mor}}_*$ by Proposition \ref{prop_orb_comparison}.
\end{remark}

As indicated in the introduction, we expect GHT to be useful in situations where singular chains lead to difficulties. In the next section we illustrate this with the author's work on Batalin--Vilkovisky (BV) algebra structures in open-closed topological conformal field theory.

\section{Applications}\label{sec_applications}
As was briefly mentioned in the introduction, geometric chains are well suited for the construction of a BV structure on chains on the compactified moduli spaces of Riemann surfaces with boundary and marked points, extending the work of Costello \cite{cos04,cos05} on algebraic definitions of Gromov--Witten potentials. We explain this in more detail. The BV operator $\Delta$ involves pullbacks of chains along circle bundles, which are not defined on singular chains. Costello works with homotopy coinvariants of the circle actions on singular chains (see the footnote on page 5 of \cite{cos05}, and Section 5 there). This is not merely a technical device: the homotopy coinvariant model is also what allows one to include $\psi$-class insertions in the definition of categorical Gromov--Witten invariants, and to act on Hochschild chains, which makes the theory explicit in terms of ribbon graphs.

Geometric chains give another chain-level model for this BV structure, since pullbacks along circle bundles are defined on them (Section \ref{sec_operations}); this model does not use equivariant homotopy theory. We do not claim that it replaces the homotopy coinvariant model of \cite{cos05}; in particular, we do not discuss $\psi$-class insertions or the action on Hochschild chains here. Theorem \ref{main_thm}, in the form of Corollary \ref{cor_corners_general}, Proposition \ref{prop_orb_comparison} and Remark \ref{rem_twisted}, ensures that geometric chains compute the rational homology of the moduli spaces, with the twisted coefficients, which is what is needed to carry the homological arguments of \cite{cos05} over to the open-closed case. 
The following statements concern the BV structure on the space of geometric chains mentioned above.

\begin{definition}
Let $V$ be a graded vector space over a field $k$.
A \textit{dg-BV algebraic structure} on $V$ is a
quadruple $(V, \bullet, d, \Delta)$ satisfying the following three conditions:
\begin{enumerate}
\item $(V,\bullet,d)$ is a differential graded, (graded) commutative, (graded) associative algebra
over $k$. The differential $d$ is of degree 1 and $d(1) = 0$.
\item $\Delta$ is a second-order differential operator with respect to $\bullet$, i.e. the degree of $\Delta$ is 1,
$\Delta^2 = 0$, $\Delta(1) = 0$, and for all $a,b,c \in V$,
\begin{align*}
\Delta (a \bullet b \bullet c) & = & \Delta (a \bullet b) \bullet c + (-1)^{|a|}a \bullet \Delta (b \bullet c) + (-1)^{(|a|+1)|b|}b \bullet \Delta (a \bullet c) \\
& &-(\Delta a) \bullet b \bullet c - (-1)^{|a|} a \bullet (\Delta b) \bullet c - (-1)^{|a|+|b|} a \bullet b \bullet (\Delta c)
 \end{align*}
where $|\cdot|$ denotes the degree of an element.
\item $d$ and $\Delta$ graded-commute, i.e. $[d, \Delta] = d \Delta + \Delta d = 0$.
\end{enumerate}
\end{definition}
If $V$ is concentrated in degree 0 (with $d$ trivial), then the dg-BV algebraic structure is called a \textit{BV algebraic structure}.
\begin{remark}
Condition 2 is equivalent to requiring that the deviation of the operator $\Delta$ from being a derivation, defined by
\begin{equation*}
\{, \} := \Delta (ab) - \Delta (a)b - (-1)^{|a|}a\Delta (b)
\end{equation*}
is a (graded) Lie bracket and that $\{, \}$ is a (graded) derivation in each variable, i.e.
\begin{eqnarray*}
\{a,bc\} = b\{a,c\} + (-1)^{|b|}
\{a,b\}c \\
\{ab,c\} = a\{b,c\} + (-1)^{|a|}
\{a,c\}b
\end{eqnarray*}
Condition 3 is equivalent to $d$ being a (graded) derivation of the Lie bracket $\{, \}$, i.e.
\begin{equation*}
d\{a,b\} = \{da,b\} + (-1)^{|a|}
\{a,db\}
\end{equation*}
\end{remark}
The open-closed analogues of the moduli spaces of Riemann surfaces with marked points are the moduli spaces of bordered Riemann surfaces with marked points both in the interior and on the boundary. More precisely, consider the moduli spaces of isomorphism classes of stable
bordered Riemann surfaces of type $(g,b)$ with $(n, \vec{m})$ marked points and certain extra data,
namely, decorations by a real tangent direction, i.e., a ray, in the complex tensor product of the tangent spaces on each side of each interior node. We denote this moduli space by $\underline{M}^{b,\vec{m}}_{g,n}$. We work with the moduli spaces
\begin{equation*}
\underline{M}^{b,m}_{g,n} / \varrho = (\prod_{\vec{m}:\sum {m_i} = m}
\underline{M}^{b,\vec{m}}_{g,n} / Z_{m_1} \times \cdots \times Z_{m_b} ) / \varrho_b \times \varrho_n
\end{equation*}
of stable bordered Riemann surfaces as above with unlabelled boundary components
and marked points, that is, the quotient of the disjoint union
$ \underline{M}^{b,m}_{g,n} / \varrho = \prod_{\vec{m}:\sum {m_i} = m}
\underline{M}^{b,\vec{m}}_{g,n} $
of moduli spaces with labelled boundaries and marked points by an appropriate action
of the permutation group $\varrho = (\prod_{\vec{m}:\sum_{m_i} = m} Z_{m_1} \times \cdots \times Z_{m_b} ) \times \varrho_b \times \varrho_n $. Both $\underline{M}^{b,\vec{m}}_{g,n}$ and $\underline{M}^{b,m}_{g,n} / \varrho$ are oriented orbifolds with corners of dimension $6g-6+2n+3b+m$. The latter is equipped with a local system $\mathbb{Q}^\epsilon$ coming from a certain sign representation of $\varrho$ on a one-dimensional vector space $L$ over $\mathbb{Q}$, which is the orientation sheaf of $\underline{M}^{b,m}_{g,n} / \varrho$. The local system $\mathbb{Q}^\epsilon$ is obtained by the orbifold construction from the determinant bundle of $\underline{M}^{b,m}_{g,n}$, which is trivial since this space is oriented, together with the action of $\varrho$. See \cite{vhz} for details.

The space on which we introduce a BV algebraic structure is the space of geometric chains.
\begin{equation*}
V := \bigoplus_{b,m,n}C^{geom}_{*}(\underline{M}^{b,m}_n / \varrho; \mathbb{Q}^\epsilon)
\end{equation*}
where $\underline{M}^{b,m}_n /\varrho$ is the (real compactified) moduli space of stable bordered Riemann surfaces with $b$ boundary
components, $n$ interior marked points, and $m$ boundary marked points, just like the Riemann
surfaces in $\underline{M}^{b,m}_{g,n} / \varrho$, but possibly having several connected components of various
genera. Throughout this section, $C^{geom}_*$ denotes the complex $GC^{\mathrm{mor}}_*$ of Definition \ref{def_orb_chains}, whose chains are given by orbifold morphisms, with twisted coefficients as in Remark \ref{rem_twisted}. By Proposition \ref{prop_orb_comparison} it is quasi-isomorphic to $GC^\infty_*$, compatibly with $p$. We use $GC^{\mathrm{mor}}_*$ because pullbacks along circle bundles are defined on it (Section \ref{sec_operations}). It carries the grading of \cite{cos05,vhz}: a $k$-dimensional chain on $\underline{M}^{b,m}_{g,n}$ has degree $\dim\underline{M}^{b,m}_{g,n}-k$, so fundamental chains have degree $0$ and $d=\partial$ has degree $1$.
\begin{itemize}
\item The fundamental chain $[\underline{M}^{b,m}_{g,n}/\varrho]$ is the image of $\sum_{\vec m:\sum m_i=m}[\underline{M}^{b,\vec m}_{g,n},\hat\iota_{\vec m}]\otimes1$ in $\big(GC^{\mathrm{mor}}_*(\underline{M}^{b,m}_{g,n})\otimes L\big)_\varrho$, where $\hat\iota_{\vec m}$ is the inclusion of a union of components, a strict orbifold morphism.
\item Its boundary is computed with the boundary of Convention \ref{conv_faces}, so the moduli spaces need not have embedded faces.
\item Some of these moduli spaces are not effective: the generic point of a component may have a non-trivial automorphism group $\Gamma$ (e.g. for the one-holed torus). For these we use the convention after Definition \ref{orbifold}, i.e. the effective orbifold $\underline{M}_{\mathrm{red}}$ with the same underlying space, weighted by $1/|\Gamma|$. As a target, such a component is likewise replaced by $\underline{M}_{\mathrm{red}}$, without weight, so that Definition \ref{def_orb_chains} and Proposition \ref{prop_orb_comparison} apply; this does not change the homology, since $|\underline{M}_{\mathrm{red}}|=|\underline{M}|$. A circle bundle over such a component descends to $\underline{M}_{\mathrm{red}}$ after dividing each fibre by the image of $\Gamma$, which acts on it by orientation preserving diffeomorphisms of finite order, $\Gamma$ acting trivially on the base and the bundle being oriented. A finite subgroup of $\mathrm{Diff}^+(S^1)$ is conjugate to a group of rotations and acts freely, so the quotient fibre is again a circle; if $\Gamma$ acts effectively on the fibres, the total space is effective and is not changed by this operation. The resulting factors belong to the normalizations below. Pushforward along a gluing map $\underline{M}\to\underline{M}'$ in $GC^{\mathrm{mor}}_*$ requires a morphism $\underline{M}_{\mathrm{red}}\to\underline{M}'_{\mathrm{red}}$. Such a morphism exists if the generic automorphism group of $\underline{M}$ is mapped into that of $\underline{M}'$; otherwise only the underlying map is liftable. We do not verify this condition for the maps of \cite{vhz}.
\end{itemize}
The normalizations (orbit sizes, $1/|\varrho|$, $1/|\Gamma|$) are to be matched with those of \cite{vhz}.

One can canonically define an operator $\Delta = \Delta_c + \Delta_o + \Delta_{co}$ on $V$. It is proved that this is a BV operator.
\begin{theorem}
The operator $\Delta = \Delta_c+\Delta_o+\Delta_{co}$ is a graded second-order
differential on the dg graded commutative algebra $V$ and thereby defines the structure of
a dg BV algebra on $V$.
\end{theorem}
We do not present the construction of $\Delta$ in detail; instead, we refer the interested reader to \cite{vhz} for a very clear account. The operations of \cite{vhz} are built from Cartesian products, pushforwards along the gluing maps, pullbacks along circle bundles and $\partial$. On $GC^{\mathrm{mor}}_*$, pushforwards along orbifold morphisms and $\partial$ preserve degenerate chains by Lemma \ref{lem_D_local} (item (2) in the proof of Proposition \ref{prop_orb_comparison}), and so do products, as in Section \ref{sec_operations}, using products of morphisms. Pullbacks along oriented circle bundles over the moduli orbifolds are defined on $GC^{\mathrm{mor}}_*$ and preserve degenerate chains (Section \ref{sec_operations}). By Proposition \ref{prop_orb_comparison} and Remark \ref{rem_twisted}, $GC^{\mathrm{mor}}_*$ computes the rational homology of the moduli spaces, with the twisted coefficients. Here the gluing maps and the circle bundles are understood as orbifold morphisms and orbifold bundles, as given by the construction of the moduli spaces as quotients, and as compatible with the actions of $\varrho$ and with $L$, so that they induce maps of the coinvariants of Remark \ref{rem_twisted}; we do not verify this here. We emphasize that the BV operator is constructed more naturally and easily on geometric chains than on singular chains. This is one reason why we prefer geometric chains here.

This result is not entirely new: it is a mathematically rigorous version of ideas developed originally in physics by Sen and Zwiebach; see \cite{sz} for the closed case and \cite{zwi} for the open-closed case.

To develop a purely algebraic treatment of Gromov--Witten invariants, one first needs to replace the fundamental classes of the compactified moduli spaces of Riemann surfaces with boundary and marked points by a certain solution of the quantum master equation associated with the BV algebra constructed above. The fundamental result is the following theorem on the existence and uniqueness of solutions of this equation. 

Let $V^{b,m}_{g,n}$ be the image of the inclusion $C^{geom}_*(\underline{M}^{b,m}_{g,n}/\varrho; \mathbb{Q}^\epsilon) \rightarrow V$.
\begin{theorem}
For each $g,n,m,b$ with $2 - 2g - n - b - m/2 < -\frac{1}{2}$
, there exists
an element $S^{b,m}_{g,n} \in V^{b,m}_{g,n}$ of degree 0, with the following properties.
\begin{enumerate}
\item $S^{0,0}_{0,3}$
is a 0-chain in the moduli space of Riemann spheres with 3 unparameterised,
unordered closed boundaries and with no open or free boundaries.
\item Form the generating function
\begin{equation*}
S = \sum_{\substack{g,n,b,m \ge 0 \\ 2g+n+b+m/2-2>0}}\hbar^{p-\chi}\lambda^{-2\chi}S^{b,m}_{g,n} \in \lambda F(M)[[\sqrt{\hbar},\lambda]]
\end{equation*}
here $p = 1 - \frac{m+n}{2},\chi = 2 - 2g - n - b - \frac{m}{2}$.
\end{enumerate}
S satisfies the Batalin-Vilkovisky quantum master equation:
\begin{equation*}
\hat{d}e^{S/\hbar} = 0
\end{equation*}
(Remember that in dg BV algebra $B[[\hbar,\lambda]]$, we let $\hat{d} = d + \hbar \Delta$).
Equivalently,
\begin{equation*}
\hat{d}S + \frac{1}{2}\{S, S\} = 0
\end{equation*}
Further, such an $S$ is unique up to homotopy through such elements.
\end{theorem}
This result has the same form as the one for the closed case. By Proposition \ref{prop_orb_comparison} and Remark \ref{rem_twisted}, the complexes $GC^{\mathrm{mor}}_*$ compute the same homology as singular chains for the moduli spaces involved, with the twisted coefficients. The proof of \cite{cos05} then applies essentially verbatim; for more details see \cite{hy}.

As in the closed case, this solution also encodes the fundamental chains of moduli spaces of Riemann surfaces with boundary and marked points, as shown in \cite{vhz}:
\begin{theorem}
For the notation of above theorem, if we take $S^{b,m}_{g,n}:= [\underline{M}^{b,m}_{g,n}/\varrho] \in C^{geom}_0(\underline{M}^{b,m}_{g,n}/ \varrho; \mathbb{Q}^\epsilon)$ be the fundamental (geometric) chains. Then $S$ satisfies the quantum master equation:
\begin{equation*}
(d + \hbar \Delta)e^{S/\hbar} = 0
\end{equation*}
\end{theorem}

Together with the results of \cite{cos05}, these theorems essentially state that the fundamental classes of moduli spaces of Riemann surfaces with marked points, or the (geometric) fundamental chains of moduli spaces of bordered Riemann surfaces with marked points, are encoded in the unique (up to homotopy) solution of the quantum master equation associated with a canonically defined BV algebraic structure on the singular (or geometric) chain complexes. From here one can proceed to define the Gromov--Witten potential of any (open-closed) topological conformal field theory; this subject, called categorical Gromov--Witten theory, was initiated by Costello, motivated by Kontsevich's homological mirror symmetry. The closed case was completed in \cite{cos05}. In the open-closed case, Ulmer develops open-closed categorical enumerative invariants \cite{ul25,ul26}, and Hirschi and Hugtenburg construct, for a relatively spin embedded Lagrangian, a chain-level open-closed field theory on moduli spaces of stable curves with boundary that extends the Fukaya $A_\infty$ algebra to all genera \cite{hh25,hh26}, as a first step toward Kontsevich's conjecture that the Fukaya category determines the Gromov--Witten invariants.


Finally, we remark that the theory in this paper uses \textit{smooth} orbifolds with corners as sources: bisecting functions, Sard's theorem and thin factorizations all require a smooth structure on $Q$. We do not know whether a version with topological sources exists; note that by \cite{man} there are topological manifolds of every dimension $\geq 5$ that are not triangulable. We leave this as a topic for future study.

\subsection{Further applications}\label{sec_further}
\textit{Morse theory.} Let $h$ be a Morse function on a closed manifold $M$, with a metric $g$ such that $-\nabla h$ is Morse--Smale \cite[Definition 2.7]{qin} and $g$ is locally trivial \cite[Definition 2.16]{qin}: near each critical point $a$ there are coordinates $(x,y)$ in which $g$ is Euclidean and $h=h(a)-\frac12|x|^2+\frac12|y|^2$. Let $U(a)$, $A(a)$ be the descending and ascending manifolds, $\mathcal M(a,b)=(U(a)\cap A(b))/\mathbb R$ for $b\ne a$, of dimension $\mathrm{ind}\,a-\mathrm{ind}\,b-1$, $\overline U(a)$, $\overline{\mathcal M}(a,b)$ the compactifications by broken flow lines, and $e_a:\overline U(a)\to M$ the evaluation map \cite[Definitions 2.4, 2.8, \S3.2]{qin}. Orient each $U(a)$; this orients $\overline U(a)$ and, by \cite[\S6.1]{qin}, $\mathcal M(a,b)$. The group $CM_k(h;\mathbb Q)$ has as basis the critical points of index $k$, and $da=\sum_{\mathrm{ind}\,b=\mathrm{ind}\,a-1}\#\mathcal M(a,b)\,b$, where $\gamma\in\mathcal M(a,b)$ counts $+1$ if and only if, at $z=\gamma\cap h^{-1}(t)$, $t\in(h(b),h(a))$ regular, the orientation of $T_z(U(a)\cap h^{-1}(t))$ for which $(-\nabla h,\dots)$ is positive in $T_zU(a)$ agrees, via $T_zh^{-1}(t)/T_z(A(b)\cap h^{-1}(t))\cong T_zM/T_zA(b)$, with the orientation of the normal bundle of $A(b)$ given at $b$ by $T_bU(b)$. This is the differential of \cite[Theorem 3.9]{qin}; $d^2=0$ by \cite[Theorem 3.6(1)]{qin}.

\begin{lemma}\label{lem_qin}
\cite[Theorems 3.3, 3.4, 3.6]{qin} (a) $\overline{\mathcal M}(a,b)$ is a compact oriented manifold with corners of dimension $\mathrm{ind}\,a-\mathrm{ind}\,b-1$ with interior $\mathcal M(a,b)$. (b) $\overline U(a)$ is a compact oriented manifold with faces \cite[Definition 2.14]{qin} of dimension $\mathrm{ind}\,a$ with interior $U(a)$; $e_a$ is smooth and is the inclusion on $U(a)$; concatenation gives smooth embeddings $i_{(a,b)}:\overline{\mathcal M}(a,b)\times\overline U(b)\to\overline U(a)$, and $e_a\circ i_{(a,b)}=e_b\circ\mathrm{pr}_2$. (c) The open $1$-stratum of $\overline U(a)$ is $\bigsqcup_bi_{(a,b)}(\mathcal M(a,b)\times U(b))$, and $i_{(a,b)}$ preserves orientations there, the $1$-stratum carrying the boundary orientation, outward normal first. These smooth structures depend only on $(h,g)$, since evaluations at regular level sets are embeddings \cite[Theorem 3.3(3), Lemma 5.4]{qin}.
\end{lemma}

\begin{proposition}\label{prop_morse}
$\Phi(a):=[\overline U(a),e_a]$ defines a chain map $\Phi:CM_*(h;\mathbb Q)\to GC^\infty_*(M)$, depending only on $h$, $g$ and the orientations of the $U(a)$. If $M$ is oriented, $\Phi$ is a quasi-isomorphism.
\end{proposition}
\begin{proof}
Let $k=\mathrm{ind}\,a$. (i) As $\overline U(a)$ has faces, the $l$ local boundary components at a point of depth $l$ lie in distinct components of $\partial\overline U(a)$, so $i=i_{\overline U(a)}$ is injective on each component $\partial_F$, with image $F$. As $\partial_F$ is compact, near $x\in F$ the set $F$ is the hyperplane of the unique $(x,\beta)\in\partial_F$. The interior of $\partial\overline U(a)$ is identified by $i$ with the $1$-stratum. Let $C$ be a component of $\overline{\mathcal M}(a,b)\times\overline U(b)$. By Lemma \ref{lem_qin}(c), $i_{(a,b)}$ maps the interior of $C$, connected and dense, onto a component of the $1$-stratum, hence $C$ onto some $F$; distinct $C$ give distinct $F$, and every $F$ arises. Since $i_{(a,b)}|_C$ takes values locally in one hyperplane, $(i|_{\partial_F})^{-1}\circ i_{(a,b)}|_C:C\to\partial_F$ is a smooth bijection with invertible differential, hence a diffeomorphism (extend locally, inverse function theorem), orientation preserving on the interior, hence everywhere. As $\mathcal M(a,b)=\emptyset$ unless $\mathrm{ind}\,b<k$ (Morse--Smale), in $C^{g,\infty}_{k-1}(M)$
$$\partial\Phi(a)=\sum_b[\overline{\mathcal M}(a,b)\times\overline U(b),e_b\circ\mathrm{pr}_2]=\sum_{\mathrm{ind}\,b=k-1}\#\mathcal M(a,b)\,[\overline U(b),e_b]+T_a,$$
$T_a$ being the sum over $\mathrm{ind}\,b\le k-2$; for $\mathrm{ind}\,b=k-1$, $\overline{\mathcal M}(a,b)$ is a finite set of signed points $\gamma$, and $\mathrm{pr}_2:\{\gamma\}\times\overline U(b)\to\pm\overline U(b)$ is an isomorphism. By \cite[Proposition 7.4]{jo} ($Z$ a point), as in Proposition \ref{prop_homotopy},
$$\partial\big(\overline{\mathcal M}(a,b)\times\overline U(b)\big)=\big(\partial\overline{\mathcal M}(a,b)\times\overline U(b)\big)\sqcup(-1)^{k-\mathrm{ind}\,b-1}\big(\overline{\mathcal M}(a,b)\times\partial\overline U(b)\big).$$
(ii) Let $\mathrm{ind}\,b\le k-2$. The components of $\overline{\mathcal M}(a,b)\times\overline U(b)$ have dimension $k-1$, and $e_b\circ\mathrm{pr}_2=e_b\circ\pi$, $\pi=\mathrm{pr}_2$ onto $N=\overline U(b)$, $\dim N=\mathrm{ind}\,b<k-1$, with $\pi$, $e_b$ smooth. So $T_a\in S^\infty_{k-1}(M)$.

(iii) The components of $\overline{\mathcal M}(a,b)\times\partial\overline U(b)$ are thin, via $\mathrm{pr}_2$ onto $\partial\overline U(b)$, of dimension $\mathrm{ind}\,b-1<k-2$, followed by $e_b\circ i_{\overline U(b)}$. If $\mathrm{ind}\,b\le k-3$, those of $\partial\overline{\mathcal M}(a,b)\times\overline U(b)$ are thin via $\mathrm{pr}_2$ onto $\overline U(b)$, $\mathrm{ind}\,b<k-2$. If $\mathrm{ind}\,b=k-2$, $\overline{\mathcal M}(a,b)$ is a compact oriented $1$-manifold, so its boundary consists of points $x$ with signs $\varepsilon(x)$, $\sum_x\varepsilon(x)=0$, the two ends of an arc having opposite signs. Since $\mathrm{pr}_2:\{x\}\times\overline U(b)\to\varepsilon(x)\overline U(b)$ is an isomorphism, in $C^{g,\infty}_{k-2}(M)$
$$\big[\partial\overline{\mathcal M}(a,b)\times\overline U(b),e_b\circ\mathrm{pr}_2\big]=\Big(\sum_x\varepsilon(x)\Big)[\overline U(b),e_b]=0.$$
Hence $\partial T_a\in S^\infty_{k-2}(M)$.

(iv) By (ii) and (iii), $T_a\in D^\infty_{k-1}(M)$, so $\partial\Phi=\Phi d$ in $GC^\infty_*(M)$ by (i). Reversing the orientation of $U(a)$ changes the signs of $a$ and $\Phi(a)$; nothing else was chosen.

(v) Let $\mathcal D'_j(M)$ be the currents of dimension $j$, $(\partial\tau)(\omega)=\tau(d\omega)$. By Remark \ref{rem_stokes}, $I[Q,f](\omega):=\int_Qf^*\omega$ defines a chain map $I:GC^\infty_*(M)\otimes\mathbb R\to\mathcal D'_*(M)$. $I\circ p$ is integration over smooth singular chains, a quasi-isomorphism by de Rham's theorem \cite{dr}, and $p$ is one by Theorem \ref{main_thm} and Lemma \ref{lem_smooth_singular}, $\mathbb R$ being flat over $\mathbb Q$; hence so is $I$. Next, $I\Phi(a)=[U(a)]$, since $\overline U(a)\setminus U(a)$ has measure zero. By (i) and induction, $e_b(\overline U(b))\subseteq U(b)\cup\bigcup_{\mathrm{ind}\,c<\mathrm{ind}\,b}U(c)$; this compact set misses $z\in U(a)$ if $b\ne a$, $\mathrm{ind}\,b=\mathrm{ind}\,a$, so a suitable bump form near $z$ pairs non-trivially only with $[U(a)]$. Thus the $[U(a)]$ are linearly independent, and $I\circ(\Phi\otimes\mathbb R)$ is an isomorphism of complexes onto $\mathrm{Stab}:=\mathrm{span}_{\mathbb R}\{[U(a)]\}$ followed by the inclusion $\mathrm{Stab}\subseteq\mathcal D'_*(M)$. In \cite{hl}, which uses the flow of $+\nabla f$, take $f=2h$; then $f=f(a)-|x|^2+|y|^2$, $g=|dx|^2+|dy|^2$ near $a$, so $g$ is $f$-tame \cite[\S3]{hl}, the stable manifolds $S_a$ are the $U(a)$, and the flow is Morse--Stokes by \cite[Remark 3.2]{hl}, proved in \cite{hl2}. By \cite[Theorems 4.1, 4.2]{hl} and the proof of the latter ($i_*\circ\widehat{\mathbb P}_*=\mathbb I_*$, both isomorphisms), the inclusion of $\mathrm{Stab}$ is a quasi-isomorphism (\cite{hl} grade currents by degree $\dim M-j$; this changes differentials only by signs). Hence $\Phi\otimes\mathbb R$ is a quasi-isomorphism, and so is $\Phi$, as $H_*(\Phi\otimes\mathbb R)=H_*(\Phi)\otimes\mathbb R$ and $\mathbb R$ is faithfully flat over $\mathbb Q$.
\end{proof}
On homology, Proposition \ref{prop_morse} is classical, and \cite{hl} already give a chain-level version without choices, with values in currents. Geometric chains add a target that computes $H_*(M;\mathbb{Q})$ and carries products and pullbacks along fibre bundles.

\textit{Circle actions.} Let $X$ be a smooth manifold without boundary with a smooth action $\mathrm{rot}:S^1\times X\to X$, $S^1=\mathbb R/\mathbb Z$, write $\mathrm{rot}_s:=\mathrm{rot}(s,\cdot)$, and put $\Delta_{\mathrm{rot}}[Q,f]:=[S^1\times Q,\ \mathrm{rot}\circ(\mathrm{id}\times f)]$, with the product orientation, $S^1$ oriented by $\mathbb R$. If $f=g\circ\rho$ is thin, with $\rho:Q\to N$, so is $\mathrm{rot}\circ(\mathrm{id}\times f)=\mathrm{rot}\circ(\mathrm{id}\times g)\circ(\mathrm{id}\times\rho)$, and $\partial(S^1\times Q)=-S^1\times\partial Q$ (Section \ref{sec_operations}) gives $\partial\Delta_{\mathrm{rot}}=-\Delta_{\mathrm{rot}}\partial$. So $\Delta_{\mathrm{rot}}$ is defined on $GC^\infty_*(X)$. Moreover $\Delta_{\mathrm{rot}}\circ\Delta_{\mathrm{rot}}=0$ already in $C^{g,\infty}_*(X)$: $\Delta_{\mathrm{rot}}\Delta_{\mathrm{rot}}[Q,f]$ is represented by $(s,t,q)\mapsto\mathrm{rot}_{s+t}f(q)$ on $S^1\times S^1\times Q$, and $(s,t,q)\mapsto(t,s,q)$ is an orientation reversing automorphism, so Lemma \ref{lem_Cg} applies. Thus $(GC^\infty_*(X),\partial,\Delta_{\mathrm{rot}})$ is a mixed complex. This is not new: singular chains form a mixed complex with $\Delta_{\mathrm{rot}}\sigma=\mathrm{rot}_*\mathrm{EZ}([S^1]\otimes\sigma)$, where $\mathrm{EZ}$ is the Eilenberg--Zilber (shuffle) map and $[S^1]$ the $1$-simplex $t\mapsto[t]$. Indeed, the shuffle map is strictly associative and strictly graded commutative and $S^1$ is abelian, so $m_*\mathrm{EZ}([S^1]\otimes[S^1])=0$ for the addition $m$ of $S^1$. What singular chains lack is a chain-level pullback without choices along the quotient map of a free action.

\begin{proposition}\label{prop_circle}
Let $\pi:E\to B$ be a principal $S^1$-bundle over a smooth manifold $B$ without boundary, with the fibres oriented by the action, and let $\Delta_{\mathrm{rot}}$ be the operator above for the action of $S^1$ on $E$. Let $\pi^*$ be the pullback of Section \ref{sec_operations} and $\pi_*$ the pushforward (Lemma \ref{lem_D_local}), both defined already on $C^{g,\infty}_*$. Then:
\begin{enumerate}
\item $\pi_*\pi^*=0$ on $GC^\infty_*(B)$;
\item $\pi^*\pi_*x=(-1)^{|x|}\Delta_{\mathrm{rot}}x$ for $x\in C^{g,\infty}_*(E)$, already in $C^{g,\infty}_*(E)$;
\item $\pi_*\Delta_{\mathrm{rot}}=0$ and $\Delta_{\mathrm{rot}}\pi^*=0$ on $GC^\infty_*$.
\end{enumerate}
\end{proposition}
\begin{proof}
$\pi_*$ is a chain map, and since $\partial S^1=\emptyset$ so is $\pi^*$ (Section \ref{sec_operations}); moreover $\partial\Delta_{\mathrm{rot}}=-\Delta_{\mathrm{rot}}\partial$. So in (1) and (3) it suffices to show that the operators map every class to a thin chain: the boundary of the image of $[P,f]$ is then, up to sign, the image of $\partial[P,f]$, again thin, and the image lies in $D^\infty_*$.

(1) Let $R:=P\times_BE$. Then $\pi_*\pi^*[P,f]=[R,\pi\circ\mathrm{pr}_E]=[R,f\circ\mathrm{pr}_P]$, which factors through $\mathrm{pr}_P:R\to P$, and $\dim P<\dim R$.

(2) Let $[Q,f]$ be a class on $E$. Then $\pi^*\pi_*[Q,f]=[R,\mathrm{pr}_E]$ with $R=Q\times_BE=\{(q,e):\pi f(q)=\pi e\}$. Let $\phi:Q\times S^1\to R$, $\phi(q,s)=(q,\mathrm{rot}_sf(q))$. In a trivialization $E|_U\cong U\times S^1$ in which the action is by translation, write $f(q)=(u(q),a(q))$; then $R$ is locally $Q\times S^1$ via $(q,(u,b))\mapsto(q,b)$, and $\phi$ becomes $(q,s)\mapsto(q,a(q)+s)$, a diffeomorphism with inverse $(q,b)\mapsto(q,b-a(q))$ which preserves the orientation of the fibre. As $R$ is oriented base first, $\phi$ preserves orientations. Since $\mathrm{pr}_E\circ\phi(q,s)=\mathrm{rot}_sf(q)$, composing with the exchange $S^1\times Q\cong Q\times S^1$, of sign $(-1)^{\dim Q}$, gives $\pi^*\pi_*[Q,f]=(-1)^{\dim Q}\Delta_{\mathrm{rot}}[Q,f]$.

(3) This follows from (1) and (2): $\pi_*\Delta_{\mathrm{rot}}=\pm\pi_*\pi^*\pi_*$ and $\Delta_{\mathrm{rot}}\pi^*=\pm\pi^*\pi_*\pi^*$. Directly: $\pi_*\Delta_{\mathrm{rot}}[Q,f]=[S^1\times Q,\pi\circ f\circ\mathrm{pr}_Q]$ factors through $\mathrm{pr}_Q$, which shows $\pi_*\Delta_{\mathrm{rot}}=0$ for every map $\pi$ invariant under the action, free or not. With $R=P\times_BE$ as in (1), $\Delta_{\mathrm{rot}}\pi^*[P,f]=[S^1\times R,(s,(p,e))\mapsto\mathrm{rot}_se]$ factors through $\mu:S^1\times R\to R$, $\mu(s,(p,e))=(p,\mathrm{rot}_se)$, which is smooth and lands in $R$ because $\pi\circ\mathrm{rot}_s=\pi$.
\end{proof}

By (1) and (2), $\Delta_{\mathrm{rot}}\circ\Delta_{\mathrm{rot}}=\pm\pi^*(\pi_*\pi^*)\pi_*=0$ in $GC^\infty_*(E)$, in accordance with the above; (2) says that $\Delta_{\mathrm{rot}}$ on $GC^\infty_*(E)$ factors through $GC^\infty_*(B)$. On homology, the corresponding identities are standard consequences of the Gysin sequence of $\pi$, with signs depending on conventions. Here they hold on chains and involve no choices; (1) and (3) use the thin relation, while (2) is an isomorphism of sources. We do not compare with $C_*(S^1)$-module structures on singular chains. The operator $\Delta_{\mathrm{rot}}$ is defined by the action map; only for free actions is it expressed, by (2), through a pullback along a circle bundle, as is the BV operator $\Delta$ of Section \ref{sec_applications}, where the pullback is composed with a gluing pushforward. The two operators are different. For orbifold principal $S^1$-bundles we expect the same statements on $GC^{\mathrm{mor}}_*$, with the pullback over orbifolds of Section \ref{sec_operations}; we have not checked the details.

\textit{String vertices.} We expect, but have not checked, the following. The hyperbolic string vertices $\mathcal{V}_{g,n}(L)$, $L\le2\,\mathrm{arcsinh}\,1$, of \cite{cz}, and their open-closed analogues \cite{cho}, should be compact orbifolds with corners whose codimension-one faces are images of twist-sewings of one or two vertices of lower complexity. Their fundamental chains would then solve the geometric master equation as an identity of geometric chains, whereas the proof of \cite[Theorem 1]{cz} compares sets of surfaces. This would take place in chains of the infinite-dimensional spaces of surfaces with local coordinates of \cite{cz}, not in the model of \cite{vhz} used above. Orientations and symmetry factors would have to be matched, and the uniqueness of $S$ in the existence theorem earlier in this section, which concerns chains on the moduli spaces of \cite{vhz}, does not apply directly.

\textit{Scope.} In this paper geometric chains are pulled back only along fibre bundles. Transverse fibre products of chains in manifolds, the intersection product, Poincar\'e duality, and Gysin and umkehr maps are constructed, via geometric cochains, for manifolds in \cite[\S7]{fmms}, and with orbifold sources and the relation of degenerate chains in \cite{yu_coh}. Geometric chains have no virtual chains, and we work with rational coefficients. For string topology and Floer theory, the M-homology of \cite{jo15} and the de Rham chains of \cite{ir} are better suited.

\appendix
\section{Quotients of intervals by finitely many curves}\label{app_intervals}
In this appendix we prove Proposition \ref{prop_dim_laps}. The notation is local to this appendix.

\begin{proof}[Proof of Proposition \ref{prop_dim_laps}]
(a) After subdividing at the endpoints of laps, $R$ is the smallest closed equivalence relation containing finitely many sets of three kinds: graphs of homeomorphisms $h:I\to J$ between closed intervals, sets $\{c\}\times J$ with $J$ a closed interval, and single points. Let $\Gamma$ be the set of finite compositions of the maps $h$ and their inverses, each with its maximal domain (a closed interval, possibly empty).

\textit{Step 1: monotone reduction.} Let $R_0$ be the set of pairs $(x,y)$ such that the closed interval between $x$ and $y$ lies in one class of $R$. Since $R$ is closed, $R_0$ is a closed equivalence relation; its classes are the components of the classes of $R$, which are closed intervals. Hence $Z:=[0,1]/R_0$ is an arc or a point, and the quotient map $m$ is monotone. Since $u\,R\,h(u)$, each $h$ maps components of classes into components of classes, so it descends to a homeomorphism between closed subarcs of $Z$ (possibly points). Each $\{c\}\times J$ descends to a single point, since $J\cup\{c\}$ lies in one class. The relation $(m\times m)(R)$ is the smallest closed equivalence relation on $Z$ containing the descended sets, since its preimage under $m\times m$ is a closed equivalence relation containing the original sets. Its quotient is $Y$. No class of $(m\times m)(R)$ contains a non-degenerate arc, since the preimage of such an arc would be an interval in one class of $R$. From now on we work on $Z$ and keep the notation.

\textit{Step 2: no shrinking.} Let $V$ be a non-degenerate closed interval, and suppose that $g_n\in\Gamma$ satisfy $V\subseteq\mathrm{dom}\,g_n$ and $|g_n(V)|\to0$. Passing to a subsequence, $g_n(V)$ converges to a point $c$. Since $R$ is closed, $V\times\{c\}\subseteq R$, a contradiction. Consequently, if $g\in\Gamma$ maps a non-degenerate closed interval $W$ into itself, then $g|_W$ is the identity or an involution. Indeed, if $g$ is increasing and $g(y)\ne y$, then the intervals between $g^k(y)$ and $g^{k+1}(y)$ are the images of a fixed interval under $g^k$, and their lengths tend to $0$.

\textit{Step 3: a finite bad set.} Let $P$ be the finite set of endpoints of the domains and ranges of the generators, together with the single points. The endpoints of domains of elements of $\Gamma$ lie in $\Gamma P$. Let $\hat F$ be the union of the classes of the points of $P$. It is closed and saturated, $q(\hat F)$ is finite, and $\hat F$ contains $F:=\overline{\Gamma P}$. By continuity, $F$ is $\Gamma$-invariant. Let $V$ be a component of $Z\setminus F$. For each $g\in\Gamma$, either $V\subseteq\mathrm{dom}\,g$ or $V\cap\mathrm{dom}\,g=\emptyset$, and in the first case $g(V)$ is again a component of $Z\setminus F$.
\begin{itemize}
\item By Step 2, the lengths of the components $g(V)$ are bounded below, so $V$ has only finitely many images $V_1,\dots,V_k$.
\item Two elements mapping $V$ onto the same $V_i$ differ by an element preserving $V$. By Step 2 there are at most two of them.
\end{itemize}
Hence the $\Gamma$-orbits in $U:=Z\setminus\hat F$ are finite.

\textit{Step 4: the relation on $U$.} Let $S$ be the union of the $\Gamma$-orbit relation on $U$ and of $R\cap(\hat F\times\hat F)$. Then $S\subseteq R$, and $S$ is an equivalence relation containing all generators. We show that $S$ is closed. Let $(x_n,y_n)\in S$ converge to $(x,y)$.
\begin{itemize}
\item If $x\in U$, then $x_n$ eventually lies in the component $V$ of $Z\setminus F$ containing $x$, and $y_n=\varphi_n(x_n)$ with $\varphi_n$ among finitely many maps. So $y=\varphi(x)$ for one of them.
\item If $x\in\hat F$, then $(x,y)\in R$, and $y\in\hat F$ because $\hat F$ is saturated.
\end{itemize}
By minimality, $R=S$.

\textit{Step 5: dimension.} For each orbit $\{V_1,\dots,V_k\}$ of components, $T:=U\cap(V_1\cup\dots\cup V_k)$ is open and saturated. Let $D$ be a fundamental domain in $V_1$: $D=U\cap V_1$ if no element of $\Gamma$ reverses $V_1$, and otherwise the part of $U\cap V_1$ on one side of the fixed point of the involution. In the second case the fixed point is included. Then $D$ is closed in $T$, because $V_1\cap T$ is open and closed in $T$. Since $T$ is saturated, for $C\subseteq D$ closed we have $q(C)=q(\overline C)\cap q(T)$, which is closed in $q(T)$. Hence $q|_D:D\to q(T)$ is a closed continuous bijection onto the open set $q(T)$, thus a homeomorphism. So $Y$ is the union of the finite set $q(\hat F)$ and countably many open subsets homeomorphic to subsets of $\mathbb{R}$. Each of them is a countable union of compact sets of dimension $\le1$, since open subsets of metric spaces are $F_\sigma$. By the countable sum theorem \cite[Theorem~3.1.8]{en78}, $\dim Y\le1$.

(b) Since $q\circ\alpha_e=q\circ\beta_e$, also $q(A^*)=q(A^\dagger)$ with $A^\dagger:=\bigcup_e(\alpha_e\cup\beta_e)(\Sigma^*(\gamma_e))$. Fix $\eta>0$, and let $J_\eta^e$ be the set of points at distance $\ge\eta$ from $\Sigma^*(\gamma_e)$. It is a finite union of closed intervals (possibly points), on which $\gamma_e$ is piecewise monotone by compactness. Let $S_\eta\subseteq R$ be the closed equivalence relation generated by the restrictions $\gamma_e|_{J_\eta^e}$. By (a), $\dim[0,1]/S_\eta\le1$. The remaining generators lie in $N\times N$, where $N$ is the closed $\varepsilon$-neighbourhood of $A^\dagger$ and $\varepsilon\to0$ as $\eta\to0$.

Let $Z$ be the $S_\eta$-saturation of $N$, which is closed. Let $R_Z$ be the smallest closed equivalence relation on $Z$ containing $S_\eta\cap Z^2$ and the remaining generators. Then $S_\eta\cup R_Z$ is a closed equivalence relation: for transitivity, if $(x,y)\in S_\eta$ and $(y,z)\in R_Z$, then $x\in Z$ because $Z$ is saturated. It contains all generators, so $R\subseteq S_\eta\cup R_Z$. Hence the classes of $R$ outside $Z$ are classes of $S_\eta$. Therefore every closed subset of $Y$ disjoint from $q(N)=q(Z)$ is homeomorphic to a closed subset of $[0,1]/S_\eta$, and has dimension $\le1$.

By compactness, $\bigcap_\eta q(N)=q(A^\dagger)$. So $Y\setminus q(A^*)$ is a countable union of closed sets of dimension $\le1$, and $q(A^*)$ is closed. The countable sum theorem \cite[Theorem~3.1.8]{en78} gives the first claim. If all $\Sigma^*(\gamma_e)$ are countable, then $q(A^*)$ is a countable compact metric space, hence $0$-dimensional.

(c) Here $R$ is the smallest closed equivalence relation containing the graph of $T$.

\textit{Reduction.} If $[x,y]$ lies in a class, then $T([x,y])$ is an interval in the same class, since $u\,R\,T(u)$. So $T$ maps components of classes into components of classes. Hence, as in Step 1 of (a), $T$ descends to the arc $Z$, and we may assume that all classes are totally disconnected.

\textit{Equicontinuity.} Suppose that $x_n\to x$, $T^{k_n}x_n\to y$ and $T^{k_n}x\to z$ with $y\ne z$. Let $t$ lie strictly between $y$ and $z$. For large $n$, the intermediate value theorem gives $y_n$ between $x_n$ and $x$ with $T^{k_n}y_n=t$. Since $(y_n,T^{k_n}y_n)\in R$ and $R$ is closed, $(x,t)\in R$. So the class of $x$ contains an interval, a contradiction. Hence $\{T^k\}$ is equicontinuous.

\textit{An idempotent.} By the Arzel\`a--Ascoli theorem, the set of uniform limits of sequences $T^{n_i}$ with $n_i\to\infty$ is a non-empty compact semigroup. By the Ellis--Numakura lemma \cite{numakura52}, \cite[Lemma~1]{ellis58} it contains an idempotent $u=\lim T^{n_i}$, and $u$ commutes with $T$. Let $\Lambda:=\bigcap_nT^n(Z)$; it is an arc or a point, and $T(\Lambda)=\Lambda$.
\begin{itemize}
\item $u(Z)\subseteq\Lambda$.
\item $u|_\Lambda=\mathrm{id}$. Given $\lambda\in\Lambda$, choose $x_i\in\Lambda$ with $T^{n_i}x_i=\lambda$. By equicontinuity, every limit point $x^*$ of the $x_i$ satisfies $u(x^*)=\lambda$, hence $u(\lambda)=u(u(x^*))=\lambda$.
\item $T|_\Lambda$ is a homeomorphism. Passing to a subsequence, $T^{n_i-1}\to v$, and $u=vT=Tv$. Since $u|_\Lambda=\mathrm{id}$, $T|_\Lambda$ is injective.
\end{itemize}

\textit{$T|_\Lambda$ is the identity or an involution.} If $T|_\Lambda$ is increasing and $T(\lambda)\ne\lambda$, then $T^n(\lambda)$ converges monotonically to a point different from $\lambda$. This contradicts $T^{n_i}(\lambda)\to u(\lambda)=\lambda$. If $T|_\Lambda$ is decreasing, only finitely many $n_i$ are odd, since a limit of decreasing maps is not the identity of a non-degenerate arc. Applying the same argument to $T^2$ shows $T^2|_\Lambda=\mathrm{id}$. So $T|_\Lambda$ is the identity or an involution $\sigma$.

\textit{Conclusion.} Let $\sim$ be equality on $\Lambda$, resp. the orbit relation of $\sigma$, and put $Q:=(u\times u)^{-1}(\sim)$. Then $Q$ is a closed equivalence relation containing the graph of $T$, because $u(Tx)=T(ux)\sim ux$. Hence $R\subseteq Q$. Conversely, $(x,ux)\in R$ as a limit of $(x,T^{n_i}x)$, and $(\lambda,T\lambda)\in R$, so $Q\subseteq R$. Thus $Y\cong\Lambda/{\sim}$, a point or an arc.
\end{proof}

\section{Fundamental chains of orbifolds with corners}\label{app_fundamental}
We construct the fundamental chains used in Remark \ref{rem_thin_cycles}. Throughout, $Q$ is a compact oriented orbifold with corners of dimension $k$, and coefficients are $\mathbb{Q}$.

\begin{definition}\label{def_fundamental_chains}
\begin{enumerate}[(1)]
\item Recall from Definition \ref{def_orb_chains}(c) the complex $C^{\mathrm{orb}}_*(Q)$ of liftable singular simplices. Thus $s:\Delta^n\to|Q|$ belongs to it if every point of $\Delta^n$ has a neighbourhood $W$ on which $s=\varphi\circ\tilde s$ for a chart $(\tilde U,G,\varphi)$ of $Q$ and a map $\tilde s:W\to\tilde U$ which is smooth up to the corners. Charts are related near corresponding points by embeddings, and the lifts of a point to a chart form one orbit of its group. Hence a simplex with image in $\varphi(\tilde U)$ is liftable if and only if it lifts locally to smooth maps into $\tilde U$. $C^{\mathrm{orb}}_*(Q)$ is preserved by precomposition with affine maps between simplices and by push-forward along diffeomorphisms of orbifolds with corners (which lift locally to charts) and along $i_Q$. By Proposition \ref{prop_orb_comparison} it computes $H_*(|Q|;\mathbb{Q})$; part (a) of the lemma below gives a direct proof of this, together with a relative version. For $U\subseteq|Q|$ open, $C^{\mathrm{orb}}_*(U)$ is spanned by the liftable simplices with image in $U$.
\item Put $B_Q:=i_Q(|\partial Q|)$. It is the closed set of points of positive depth, and $|Q|\setminus B_Q=|Q^\circ|$ is the interior. Let
\begin{equation*}
A_*(Q):=(i_Q)_*C^{\mathrm{orb}}_*(\partial Q),\qquad R_*(Q):=C^{\mathrm{orb}}_*(Q)/A_*(Q),
\end{equation*}
and, for $U\subseteq|Q|$ open, $A_*(U):=A_*(Q)\cap C^{\mathrm{orb}}_*(U)$ and $R_*(U):=C^{\mathrm{orb}}_*(U)/A_*(U)$. Thus $A_*(Q)$ is spanned by the simplices $i_Q\circ s$ with $s\in C^{\mathrm{orb}}_*(\partial Q)$, and $A_*(Q)\subseteq C_*(B_Q)$.
\item Let $Q^\circ_{\mathrm{reg}}\subseteq|Q^\circ|$ be the set of interior points with trivial isotropy group. Since $Q$ is effective, it is open and dense in $|Q|$, and it is an oriented $k$-manifold. For $z\in Q^\circ_{\mathrm{reg}}$ let $o_z\in H_k(|Q|,|Q|\setminus z)$ be the generator given by the orientation. The \textit{local degree} $\deg_z\alpha\in\mathbb{Q}$ of a class $\alpha\in H_k(|Q|,B_Q)$, or more generally of $\alpha\in H_k(|Q|,Y)$ with $z\notin Y$, is defined by $\alpha\mapsto(\deg_z\alpha)\,o_z$ under restriction to $H_k(|Q|,|Q|\setminus z)$. A chain $\mu\in C^{\mathrm{orb}}_k(Q)$ with $\partial\mu\in A_{k-1}(Q)$ has a class in $H_k(|Q|,B_Q)$, hence local degrees. It is a \textit{fundamental chain} of $Q$ if $\deg_z\mu=1$ for all $z\in Q^\circ_{\mathrm{reg}}$. If $Q=\emptyset$, then $\mu=0$.
\item By \cite[\S2]{jo}, applied in charts, $\partial^mQ$ is locally $(\partial^m\tilde U)/G$, where a point of $\partial^m\tilde U$ is a point of $\tilde U$ together with an ordered $m$-tuple of distinct local boundary components at it. Permuting the $m$-tuples defines an action of the symmetric group $S_m$ on $\partial^mQ$ by diffeomorphisms. The transposition of the entries $j+1,j+2$ acts by $\partial^{m-j-2}\sigma$, where $\sigma$ is the involution of $\partial^2(\partial^jQ)$ of Convention \ref{conv_faces}. It reverses the orientation, since the boundary of an orientation reversing diffeomorphism reverses the induced orientation. Hence $w\in S_m$ acts with orientation sign $\mathrm{sgn}(w)$. Unlike for manifolds, this action need not be free (Example \ref{ex_teardrop_chains}). The map $i_m:=i_{\partial^mQ}:\partial^{m+1}Q\to\partial^mQ$ forgets the last entry. So it is $S_m$-equivariant, where $S_m\subseteq S_{m+1}$ is the stabiliser of $m+1$, and $i_m\circ i_{m+1}\circ\sigma=i_m\circ i_{m+1}$ for the transposition $\sigma$ of $m+1,m+2$.
\end{enumerate}
\end{definition}

\begin{lemma}\label{lem_fundamental_chains}
Let $Q$ be a compact oriented orbifold with corners of dimension $k$.
\begin{enumerate}[(a)]
\item For every open $U\subseteq|Q|$ the inclusions
\begin{gather*}
C^{\mathrm{orb}}_*(U)\to C_*(U),\qquad A_*(U)\to C_*(U\cap B_Q),\\
R_*(U)\to C_*(U)/C_*(U\cap B_Q)
\end{gather*}
are quasi-isomorphisms. In particular $H_*(C^{\mathrm{orb}}_*(Q))\cong H_*(|Q|)$ and $H_*(R_*(Q))\cong H_*(|Q|,B_Q)$.
\item For every $\alpha\in H_k(|Q|,B_Q)$ and every component $Q_a$ of $Q$, $\deg_z\alpha$ is independent of $z\in Q^\circ_{\mathrm{reg}}\cap Q_a$; call it $\deg_a\alpha$. The map $\alpha\mapsto(\deg_a\alpha)_a$ is an isomorphism $H_k(|Q|,B_Q)\cong\mathbb{Q}^{\pi_0(Q)}$. Hence $H_k(R_*(Q))\cong\mathbb{Q}^{\pi_0(Q)}$ is detected by local degrees, and fundamental chains exist.
\item There are fundamental chains $\mu_m$ of $\partial^mQ$, $0\le m\le k$, such that
\begin{equation*}
w_*\mu_m=\mathrm{sgn}(w)\,\mu_m\quad(w\in S_m),\qquad\partial\mu_m=(i_m)_*\mu_{m+1},
\end{equation*}
where $\mu_{k+1}:=0$. In particular $\mu_Q:=\mu_0$ and $\nu_Q:=\mu_1$ satisfy $\partial\mu_Q=(i_Q)_*\nu_Q$ and $\partial\nu_Q\in(i_{\partial Q})_*C^{\mathrm{orb}}_{k-2}(\partial^2Q)$, and the restriction of $\nu_Q$ to every component $F$ of $\partial Q$ is a fundamental chain of $F$.
\item Two fundamental chains of $Q$ differ by $\partial\beta+(i_Q)_*\gamma$, with $\beta\in C^{\mathrm{orb}}_{k+1}(Q)$ and $\gamma\in C^{\mathrm{orb}}_k(\partial Q)$.
\end{enumerate}
\end{lemma}
\begin{proof}
If $k=0$, then $Q$ is a finite set of signed points and everything is obvious, so in (a) and (b) we assume $k\ge1$.

\textit{(a), linear charts and boxes.} By Lemma \ref{lem_doubling}(1), every point $x$ of depth $l$ has a chart $(\tilde U,G,\varphi)$ with $\tilde U\subseteq R^k_l=[0,\infty)^l\times\mathbb{R}^{k-l}$ and $\varphi(0)=x$ (a translation of the $v$-coordinates identifies $\tilde U$ with an open subset of $R^k_+$, as Definition \ref{orbi_corners} requires), on which $G$ acts by $g(u,v)=(P_gu,A_gv)$, with $P_g$ a permutation matrix and $A_g$ orthogonal for a $G$-invariant inner product. We call it a \textit{linear chart}. A \textit{box} of it is a non-empty set
\begin{equation*}
V=\{(u,v)\in R^k_l:\ a_j<u_j<b_j\ (j\le l),\ v\in E\}\subseteq\tilde U,
\end{equation*}
with $E\subseteq\mathbb{R}^{k-l}$ open and convex, such that each $gV$, $g\in G$, equals $V$ or is disjoint from $V$. If $H_V$ is the stabiliser of $V$, then $(V,H_V,\varphi|_V)$ is a chart, and we call $\varphi(V)$ a \textit{basic set}.
\begin{itemize}
\item If $V,V'$ are boxes, the non-empty sets $gV\cap g'V'$, $g,g'\in G$, have the same shape, are pairwise equal or disjoint, and are permuted by $G$. Hence the intersection of two basic sets of a chart is a finite disjoint union of basic sets of that chart.
\item Products of small intervals and small balls centred at a point of $\tilde U$ show that the basic sets form a basis of the topology of $\varphi(\tilde U)$.
\item Let $J_V$ be the set of $j\le l$ with $a_j<0$, i.e. those for which $V_j:=V\cap\{u_j=0\}$ is non-empty. Since $H_V$ permutes the intervals $(a_j,b_j)\cap[0,\infty)$, the point $c$ with $c_j=0$ for $j\in J_V$, $c_j=\frac12(a_j+b_j)$ for $j\notin J_V$, and $v$-component $|H_V|^{-1}\sum_{h\in H_V}A_hv_0$ for some $v_0\in E$, lies in $V$ and is fixed by $H_V$. The homotopy $F_t(y)=(1-t)y+tc$ preserves $V$ and each $V_j$, $j\in J_V$, and commutes with $H_V$.
\end{itemize}

\textit{(a), basic sets.} Let $U=\varphi(V)$ be a basic set. A simplex $s\in C^{\mathrm{orb}}_n(U)$ lifts locally to smooth maps $\tilde s$ into $V$, and $F$ descends to a homotopy of $U$ whose composite with $s\times\mathrm{id}_{[0,1]}$ lifts locally to $(1-t)\tilde s+tc$. So the prism operator built from the affine triangulation of $\Delta^n\times[0,1]$ maps $C^{\mathrm{orb}}_*(U)$ into $C^{\mathrm{orb}}_{*+1}(U)$. It is a chain homotopy between the identity and the chain map sending every $n$-simplex to the constant $n$-simplex at $\varphi(c)$. Hence $H_*(C^{\mathrm{orb}}_*(U))\cong\mathbb{Q}$, concentrated in degree $0$. Next, $i_Q^{-1}(U)=\big(\bigsqcup_{j\in J_V}V_j\big)/H_V$, and $F$ induces a homotopy $F'$ of it with $i_Q\circ F'_t=F_t\circ i_Q$. Hence the prism operator preserves $A_*(U)$. If $J_V\neq\emptyset$, the constant simplices at $\varphi(c)$ lie in $A_*(U)$, so $H_*(A_*(U))\cong\mathbb{Q}$ in degree $0$; if $J_V=\emptyset$, then $A_*(U)=0$. The same homotopies show that $U$ is contractible, and that $U\cap B_Q=\varphi\big(\bigcup_{j\in J_V}V_j\big)$ is contractible if $J_V\neq\emptyset$ and empty otherwise. So the first two maps in (a) are quasi-isomorphisms for basic sets, hence also for finite disjoint unions of basic sets.

\textit{(a), globalisation.} The complexes $C^{\mathrm{orb}}_*(U)$ and $A_*(U)$ are spanned by sets of simplices which are preserved by precomposition with affine maps. Hence the subdivision operator and the chain homotopy of the small chains lemma preserve them. As for singular chains, they satisfy Mayer--Vietoris sequences for open sets $U,V$, compatible with the inclusions into $C_*$; note that $A_*(U)\cap A_*(V)=A_*(U\cap V)$, since $A_*(Q)$ is spanned by simplices. By the five lemma, if the first two maps in (a) are quasi-isomorphisms for $U$, $V$ and $U\cap V$, then also for $U\cup V$. Moreover, every element of the complexes involved lies in the subcomplex of a finite union of members of a given basis of $U$, and homology commutes with directed unions. Fix a linear chart. By induction on $m$, using the first item above, the maps are quasi-isomorphisms for unions of $m$ finite disjoint unions of basic sets of this chart, hence for all open subsets of $\varphi(\tilde U)$. By induction on $m$ again, they are quasi-isomorphisms for unions $W_1\cup\dots\cup W_m$ of open sets each contained in the image of a linear chart, since $(W_1\cup\dots\cup W_{m-1})\cap W_m=\bigcup_{i<m}(W_i\cap W_m)$. Hence they are quasi-isomorphisms for all open $U$. The statement for $R_*$ follows by the five lemma.

\textit{(b), local homology.} Let $x$ have depth $l$, and choose a linear chart at $x$ with $\tilde U=\{(u,v)\in R^k_l:|u_j|<r,\ |v|<r\}$ for small $r$. If $l\ge1$, the point $e=(\varepsilon,\dots,\varepsilon,0)$, $0<\varepsilon<r$, is fixed by $G$, and $\tilde U$ and $\tilde U\setminus\{0\}$ are star-shaped with respect to $e$. Hence $\varphi(\tilde U)$ and $\varphi(\tilde U)\setminus\{x\}$ are contractible, and by excision $H_*(|Q|,|Q|\setminus x)=0$. If $l=0$, excision and radial deformation give $H_j(|Q|,|Q|\setminus x)\cong\tilde H_{j-1}(S^{k-1}/G)\cong\tilde H_{j-1}(S^{k-1})_G$, by the transfer for finite groups with rational coefficients \cite[Ch.~III]{bredon72}, applied for instance to a $G$-equivariant triangulation of $S^{k-1}$. Since $G$ preserves the orientation, this is $\mathbb{Q}$ for $j=k$ and $0$ otherwise.

\textit{(b), local constancy.} Let $x\in|Q^\circ|$, take a linear chart at $x$ with $\tilde U$ the open $r$-ball, and let $D$ be the open $\rho$-ball, $\rho<r$. By excision and radial deformation, the restriction $H_k(|Q|,|Q|\setminus\varphi(\bar D))\to H_k(|Q|,|Q|\setminus x)\cong\mathbb{Q}$ is an isomorphism. Let $[D]\in H_k(\tilde U,\tilde U\setminus\bar D)$ be the orientation class. A point $z\in\varphi(D)\cap Q^\circ_{\mathrm{reg}}$ has $|G|$ preimages in $D$, near each of which $\varphi$ is an orientation preserving homeomorphism, so $\varphi_*[D]$ restricts to $|G|\,o_z$. Hence $\varphi_*[D]$ generates $H_k(|Q|,|Q|\setminus\varphi(\bar D))$, and for every class $\zeta$ in this group, $\deg_z\zeta$ is independent of $z\in\varphi(D)\cap Q^\circ_{\mathrm{reg}}$. Applied to the image of $\alpha\in H_k(|Q|,B_Q)$, this shows that every point of $|Q^\circ|$ has a neighbourhood on whose regular points $\deg_z\alpha$ is constant. As $Q^\circ_{\mathrm{reg}}$ is dense, $\deg\alpha$ extends to a locally constant function on $|Q^\circ|$. Finally, $|Q^\circ|\cap Q_a$ is connected: $|Q^\circ|$ is dense, and every point has arbitrarily small basic neighbourhoods $\varphi(V)$ with $\varphi(V)\cap|Q^\circ|=\varphi(V\cap\{u_1,\dots,u_l>0\})$ connected. This proves the first claim of (b).

\textit{(b), a collar.} Put $M:=|Q^\circ|$. If $B_Q=\emptyset$, put $U:=\emptyset$ and $K_0:=|Q|$. Otherwise let $v$ and $\Phi$ be as in the proof of Corollary \ref{cor_corners_general}, for $X=Q$. By compactness, $B_Q\subseteq\bigcup_i\varphi_i(N_i)$ for finitely many linear charts $(\tilde U_i,G_i,\varphi_i)$, $\tilde U_i\subseteq R^k_{l_i}$, and compact sets $N_i\subseteq\tilde U_i$. Let $\tilde v_i$ be a smooth extension of the lift of $v$ to an open set $\tilde U_i'\subseteq\mathbb{R}^k$ with $\tilde U_i'\cap R^k_{l_i}=\tilde U_i$ \cite{se64}, with flow $\tilde\Phi^i$, and choose $0<\varepsilon_1<\min_{B_Q}\varepsilon$ such that $\tilde\Phi^i$ is defined on $[-2\varepsilon_1,2\varepsilon_1]\times N_i$. By the proof of Corollary \ref{cor_corners_general}, $\Phi_t(x)\in M$ for $t>0$, $\Phi$ commutes with the local groups, and $\varphi_i$ maps forward trajectories of $\tilde\Phi^i$ in $\tilde U_i$ to trajectories of $\Phi$. Let $\kappa:B_Q\times[0,1]\to|Q|$, $\kappa(x,s):=\Phi_{s\varepsilon_1}(x)$.
\begin{itemize}
\item \textit{$\kappa$ is injective.} Let $\kappa(x,s)=\kappa(x',s')$ with $a:=s\varepsilon_1\ge b:=s'\varepsilon_1$. If $b=0$, then $\kappa(x,s)=x'\in B_Q$, so $a=0$ and $x=x'$. If $b>0$, the curves $t\mapsto\Phi_{a-b+t}(x)$ and $t\mapsto\Phi_t(x')$, $0<t\le b$, are integral curves of $v$ in $M$ with the same endpoint. By uniqueness of integral curves of invariant vector fields in linear charts centred at points of $M$, the set of $t$ where they agree is open and closed in $(0,b]$, so they agree, and $t\to0$ gives $\Phi_{a-b}(x)=x'\in B_Q$. Hence $a=b$ and $x=x'$.
\item \textit{$U:=\kappa(B_Q\times[0,1))$ is open.} Let $y=\kappa(x,s)$ with $s<1$, choose $i$ and $\tilde x\in N_i$ with $\varphi_i(\tilde x)=x$, and put $t:=s\varepsilon_1$, $\tilde y:=\tilde\Phi^i_t(\tilde x)$. Since $u_j(\tilde x)=0$ for some $j\le l_i$ and $\tilde v_i$ points strictly inwards, $\tilde\Phi^i_{-\delta}(\tilde x)\notin R^k_{l_i}$ for some $\delta>0$ with $t+\delta<\varepsilon_1$. As the domain of a flow is open, every $\tilde y'\in\tilde U_i$ close to $\tilde y$ has a backward trajectory defined up to time $t+\delta$, with $\tilde\Phi^i_{-(t+\delta)}(\tilde y')\notin R^k_{l_i}$. Let $\tau$ be maximal with $\tilde\Phi^i_{-[0,\tau]}(\tilde y')\subseteq R^k_{l_i}$. Then $\tau<t+\delta$, and $\tilde x':=\tilde\Phi^i_{-\tau}(\tilde y')$ has some $u_j(\tilde x')=0$, so $x':=\varphi_i(\tilde x')\in B_Q$ and $\varphi_i(\tilde y')=\Phi_\tau(x')=\kappa(x',\tau/\varepsilon_1)\in U$. As $\varphi_i$ is open, $U$ is a neighbourhood of $y$.
\end{itemize}
Hence $\kappa$ is an embedding, $U$ deformation retracts onto $B_Q$ by $\kappa(x,s)\mapsto\kappa(x,(1-\lambda)s)$, $\lambda\in[0,1]$, $K_0:=|Q|\setminus U$ is a compact subset of $M$, and $S:=U\setminus B_Q=\kappa(B_Q\times(0,1))$ is open in $|Q|$. Writing $H_*(|Q|\,|\,K):=H_*(|Q|,|Q|\setminus K)$, restriction gives isomorphisms $H_*(|Q|,B_Q)\cong H_*(|Q|,U)=H_*(|Q|\,|\,K_0)$. Since $\tilde\Phi^i_t$ is injective and commutes with $G_i$ on $\tilde U_i$, the isotropy group of $\kappa(x,s)$ embeds into that of $x$; so $\kappa(x,s)\in Q^\circ_{\mathrm{reg}}$ if $s>0$ and $x$ has trivial isotropy.

\textit{(b), duality in the interior.} For $x\in M$ let $o_x\in H_k(|Q|\,|\,x)\cong\mathbb{Q}$ be the image of $|G|^{-1}\varphi_*[D]$, for a linear chart and a ball $D$ as in the proof of local constancy. This class has local degree $1$ at the regular points of $\varphi(D)$, and $H_k(|Q|\,|\,\varphi(\bar D))\cong\mathbb{Q}$ is detected by the local degree at any one of them; so $o_x$ does not depend on the choices, and it is the orientation generator if $x\in Q^\circ_{\mathrm{reg}}$. We claim that for every compact $K\subseteq M$:
\begin{enumerate}[(D1)]
\item $H_i(|Q|\,|\,K)=0$ for $i>k$;
\item a class in $H_k(|Q|\,|\,K)$ whose image in $H_k(|Q|\,|\,x)$ vanishes for every $x\in K$ is zero;
\item some $\alpha_K\in H_k(|Q|\,|\,K)$ has image $o_x$ in $H_k(|Q|\,|\,x)$ for every $x\in K$.
\end{enumerate}
This is \cite[Lemma~3.27]{hat02} for the rational homology manifold $M$, with the same proof. First let $K\subseteq\varphi(\tilde U)$ for a linear chart at an interior point, $\tilde U$ the open $r$-ball, and put $\tilde K:=\varphi^{-1}(K)$, a compact $G$-invariant set. By excision and the transfer \cite[Ch.~III]{bredon72}, $\varphi_*$ induces isomorphisms $H_*(\tilde U,\tilde U\setminus\tilde K)_G\cong H_*(|Q|\,|\,K)$, natural in $\tilde K$. So (D1) follows from \cite[Lemma~3.27(b)]{hat02}. In (D2) write the class as $\varphi_*\beta$. Then $\beta':=\sum_{g\in G}g_*\beta$ has $\varphi_*\beta'=|G|\varphi_*\beta$, and for $y\in\tilde K$ its image in $H_k(\tilde U,\tilde U\setminus Gy)$ is $G$-invariant and vanishes in the coinvariants, hence is zero, as $V^G\to V_G$ is bijective over $\mathbb{Q}$. So $\beta'$ restricts to $0$ at every $y\in\tilde K$, hence $\beta'=0$ by the uniqueness in \cite[Lemma~3.27(a)]{hat02}, and $\varphi_*\beta=0$. For (D3) choose $\rho<r$ such that the open $\rho$-ball $D$ contains $\tilde K$, and restrict $|G|^{-1}\varphi_*[D]$ to $K$; by the description of $o_x$ above, applied to small linear charts at $x\in K$ with image in $\varphi(D)$, it has image $o_x$. Next, if (D1)--(D3) hold for $K_1$ and $K_2$, and (D1), (D2) for $K_1\cap K_2$, then they hold for $K_1\cup K_2$, by the Mayer--Vietoris sequence of the open sets $|Q|\setminus K_1$, $|Q|\setminus K_2$: (D1) and (D2) hold as $H_{k+1}(|Q|\,|\,K_1\cap K_2)=0$, and $\alpha_{K_1}$, $\alpha_{K_2}$ have the same image in $H_k(|Q|\,|\,K_1\cap K_2)$ by (D2), so they come from a class $\alpha_{K_1\cup K_2}$. By induction on $m$, using $(K_1\cup\dots\cup K_{m-1})\cap K_m=\bigcup_{i<m}(K_i\cap K_m)$, the claim holds for unions of $m$ compact sets each contained in the image of a linear chart at an interior point. Every compact $K\subseteq M$ is such a union.

\textit{(b), conclusion.} For $x\in M$ and $\varphi(\bar D)$ as in the proof of local constancy, an element of $H_k(|Q|\,|\,\varphi(\bar D))\cong H_k(|Q|\,|\,x)$ is determined by its local degree at the regular points of $\varphi(D)$. If $\alpha\in H_k(|Q|,B_Q)$ has $\deg_a\alpha=0$ for all $a$, then $\deg_z\alpha=0$ for all $z\in Q^\circ_{\mathrm{reg}}$, so the image of $\alpha$ in $H_k(|Q|\,|\,x)$ vanishes for every $x\in M$, and $\alpha=0$ by (D2) for $K_0$ and the collar. Conversely, (D3) for $K_0$ gives $\alpha\in H_k(|Q|,B_Q)$ with image $o_x$ in $H_k(|Q|\,|\,x)$ for $x\in K_0$, hence with local degree $1$ at the regular points near $x$. Every component $Q_a$ meets $K_0$: either $\partial Q_a=\emptyset$ and $|Q_a|\subseteq K_0$, or $\kappa(x,1)\in K_0\cap Q_a$ for $x\in B_Q\cap Q_a$, by injectivity of $\kappa$. So $\deg_a\alpha=1$ for all $a$. Since each $|Q_a|$ is open and closed, $H_k(|Q|,B_Q)=\bigoplus_aH_k(|Q_a|,B_Q\cap|Q_a|)$, and the component of $\alpha$ in the summand of $Q_a$ has local degree $1$ on $Q_a$ and $0$ elsewhere. This proves surjectivity. By (a), $H_k(R_*(Q))\cong H_k(|Q|,B_Q)$, and the class of local degree $1$ is represented by a fundamental chain.

\textit{(c), the boundary.} Call $x\in B_Q$ \textit{good} if it has depth one and trivial isotropy group. Then $x=i_Q(p)$ for a unique $p\in(\partial Q)^\circ_{\mathrm{reg}}$, and $i_Q$ maps a neighbourhood of $p$ homeomorphically onto a neighbourhood of $x$ in $B_Q$. We orient $B_Q$ near $x$ accordingly, and define local degrees of $(k-1)$-cycles of $B_Q$ at good points as in Definition \ref{def_fundamental_chains}. Good points are dense in $B_Q$: in a linear chart, the points of depth $\ge2$ and the points with non-trivial isotropy form a nowhere dense subset of each $\tilde U\cap\{u_j=0\}$, because a non-trivial element of $G$ cannot fix an open subset of a boundary hyperplane (Convention \ref{conv_faces}).
\begin{enumerate}[(i)]
\item \textit{A class $\xi\in H_{k-1}(B_Q)$ whose local degrees at all good points vanish is zero.} Let $e\in H_1\big((0,1),(0,1)\setminus[\frac13,\frac23]\big)\cong\mathbb{Q}$ be a generator, and $e'$ its image in $H_1\big((0,1),(0,1)\setminus\{\frac12\}\big)$. Since $\partial(\xi\times e)=\pm\,\xi\times\partial e$, the long exact sequences of the pairs show that $\xi\mapsto\xi\times e$ and $\xi\mapsto\xi\times e'$ are isomorphisms from $H_{k-1}(B_Q)$. Via the homeomorphism $\kappa:B_Q\times(0,1)\to S$ of (b) and excision, $\zeta:=\kappa_*(\xi\times e)\in H_k(|Q|\,|\,T)$ and its image $\zeta'=\kappa_*(\xi\times e')\in H_k(|Q|\,|\,T')$ are defined, where $T:=\kappa(B_Q\times[\frac13,\frac23])$ and $T':=\kappa(B_Q\times\{\frac12\})$, and $\zeta'=0$ implies $\xi=0$. Let $x\in T'$, and take a linear chart at $x$ with $\varphi(\bar D)\subseteq\kappa(B_Q\times(\frac13,\frac23))$, an open set. As good points are dense in $B_Q$, $\varphi(D)$ contains a point $z=\kappa(b,t)$ with $b$ good, and $z\in Q^\circ_{\mathrm{reg}}$ by (b). By naturality of the cross product of pairs (the subspaces involved are open), the image of $\zeta$ in $H_k(|Q|\,|\,z)$ is $\kappa_*(\xi_b\times e_t)$, where $\xi_b\in H_{k-1}(B_Q,B_Q\setminus b)$ and $e_t\in H_1\big((0,1),(0,1)\setminus\{t\}\big)$ are the images of $\xi$ and $e$; and $\xi_b=0$ as $\deg_b\xi=0$. So $\deg_z\zeta=0$, hence the images of $\zeta$ in $H_k(|Q|\,|\,\varphi(\bar D))$ and of $\zeta'$ in $H_k(|Q|\,|\,x)$ vanish (proof of (b)). By (D2) for $T'$, $\zeta'=0$, so $\xi=0$.
\item \textit{If $\alpha\in H_k(|Q|,B_Q)$ has local degree $1$, then $\delta\alpha\in H_{k-1}(B_Q)$ has local degree $1$ at every good point $x$.} A chart at $x$ with trivial group, followed by an affine map, gives an orientation preserving homeomorphism $h$ of a neighbourhood $W$ of $x$ onto a neighbourhood of an interior point of a face of $\Delta^k$, with $h(W\cap B_Q)=h(W)\cap\partial\Delta^k$. Let $K\subseteq W$ be a closed neighbourhood of $x$ such that $h(K)$ is a small closed half ball. By excision and naturality of connecting homomorphisms, $\deg_x\delta\alpha$ depends only on the image of $\alpha$ in $H_k(|Q|,B_Q\cup(|Q|\setminus\mathring K))\cong\mathbb{Q}$, and this image is detected by the local degree at interior points of $K$. Transporting by $h$, the claim follows from the case of $\Delta^k$ and the class of $\iota_k=\mathrm{id}_{\Delta^k}$ in $H_k(\Delta^k,\partial\Delta^k)$: its boundary $\sum_i(-1)^id_i$ has local degree $1$ on the interior of every face, for the outward-normal-first orientation (Section \ref{sec_comparison}).
\item \textit{If $\nu$ is a fundamental chain of $\partial Q$ and $(i_Q)_*\nu$ is a cycle, then $(i_Q)_*\nu$ has local degree $1$ at every good point $x=i_Q(p)$.} Indeed, $\partial\nu\in(i_{\partial Q})_*C^{\mathrm{orb}}_{k-2}(\partial^2Q)$ is supported away from $p$, and $i_Q$ induces an isomorphism $H_{k-1}(|\partial Q|,|\partial Q|\setminus p)\to H_{k-1}(B_Q,B_Q\setminus x)$ which maps the class of $\nu$ to that of $(i_Q)_*\nu$.
\end{enumerate}

\textit{(c), the inductive step.} Let $P$ be a compact oriented orbifold with corners of dimension $n$, and $\nu$ a fundamental chain of $\partial P$ such that $(i_P)_*\nu$ is a cycle. Then there is a fundamental chain $\mu$ of $P$ with $\partial\mu=(i_P)_*\nu$. Indeed, choose a fundamental chain $\mu'$ of $P$ by (b). The cycles $\partial\mu'$ and $(i_P)_*\nu$ of $A_*(P)$ represent $\delta[\mu']$ and $[(i_P)_*\nu]$ in $H_{n-1}(B_P)$. By (ii) and (iii) both have local degree $1$ at all good points, so they are equal by (i). By (a), $A_*(P)\to C_*(B_P)$ is a quasi-isomorphism, so $(i_P)_*\nu-\partial\mu'=\partial a$ with $a\in A_n(P)$. Then $\mu:=\mu'+a$ has the same class as $\mu'$ in $H_n(|P|,B_P)$, so it is a fundamental chain, and $\partial\mu=(i_P)_*\nu$.

\textit{(c), the tower.} We construct $\mu_m$ by downward induction on $m$, starting with $\mu_{k+1}=\mu_{k+2}=0$. Suppose that $\mu_{m+1}$ and $\mu_{m+2}$ have been constructed. Since $\mu_{m+2}$ is antisymmetric under the transposition $\sigma$ of $m+1,m+2$, and $i_m\circ i_{m+1}\circ\sigma=i_m\circ i_{m+1}$,
\begin{equation*}
\partial(i_m)_*\mu_{m+1}=(i_m\circ i_{m+1})_*\mu_{m+2}=(i_m\circ i_{m+1})_*\sigma_*\mu_{m+2}=-(i_m\circ i_{m+1})_*\mu_{m+2},
\end{equation*}
so $(i_m)_*\mu_{m+1}$ is a cycle, as $2$ is invertible in $\mathbb{Q}$. The inductive step, applied to $P=\partial^mQ$ and $\nu=\mu_{m+1}$, gives a fundamental chain $\mu$ of $\partial^mQ$ with $\partial\mu=(i_m)_*\mu_{m+1}$. Put
\begin{equation*}
\mu_m:=\frac1{m!}\sum_{w\in S_m}\mathrm{sgn}(w)\,w_*\mu .
\end{equation*}
Each $w\in S_m$ is a diffeomorphism of orientation sign $\mathrm{sgn}(w)$, and $w\circ i_m=i_m\circ w$, so $w_*$ preserves $A_*(\partial^mQ)$, and $\mathrm{sgn}(w)\,w_*\mu$ is a fundamental chain. Moreover $\partial(\mathrm{sgn}(w)\,w_*\mu)=\mathrm{sgn}(w)\,(i_m)_*w_*\mu_{m+1}=(i_m)_*\mu_{m+1}$, as $\mu_{m+1}$ is antisymmetric under $S_{m+1}\supseteq S_m$. Hence $\mu_m$ is an antisymmetric fundamental chain with $\partial\mu_m=(i_m)_*\mu_{m+1}$. The last assertion of (c) holds because $\partial^2Q=\bigsqcup_F\partial F$.

\textit{(d).} If $\mu,\mu'$ are fundamental chains, then $\mu-\mu'$ is a cycle of $R_*(Q)$ all of whose local degrees vanish. By (a) and (b) it is a boundary in $R_*(Q)$, i.e. $\mu-\mu'=\partial\beta+a$ with $a\in A_k(Q)=(i_Q)_*C^{\mathrm{orb}}_k(\partial Q)$.
\end{proof}

Neither the proof of the lemma nor that of Theorem \ref{main_thm} uses triangulations of orbifolds with corners: in (b) and (c) the inward flow of Corollary \ref{cor_corners_general} provides a collar of $B_Q$, and the duality (D1)--(D3) on $|Q^\circ|$ replaces a triangulation. Rational coefficients are used in the transfer, in the division by $|G|$ in the proof of local constancy and of (D2), (D3), in the cancellation showing that $(i_m)_*\mu_{m+1}$ is a cycle, and in the antisymmetrisation.

\begin{example}\label{ex_teardrop_chains}
Let $Q$ be the teardrop of Convention \ref{conv_faces}: a closed disc with one corner point $c$, and $\partial Q=F\cong[0,1]$ with $i_Q(0)=i_Q(1)=c$. Then $\partial^2Q=\{0,1\}$ with signs $-1,+1$, and $\sigma$ exchanges the two points, so $\mu_2=[1]-[0]$ is antisymmetric. For $\mu_1$ one can take a liftable $1$-simplex in $F$ from $0$ to $1$. It is not a cycle in $\partial Q$, but $(i_Q)_*\mu_1$ is a loop at $c$, which bounds a fundamental chain $\mu_0$. Accordingly $H_1(A_*(Q))\cong H_1(B_Q)\cong H_1(S^1)$, whereas $H_1(C^{\mathrm{orb}}_*(\partial Q))\cong H_1([0,1])=0$; so $(i_Q)_*$ is not a quasi-isomorphism onto $A_*(Q)$. Moreover $A_*(Q)$ is strictly smaller than $C^{\mathrm{orb}}_*(Q)\cap C_*(B_Q)$: a smooth path in $B_Q$ which enters $c$ along one local boundary component and leaves it along the other (necessarily flat at $c$) does not lift to $\partial Q$. The involution $\sigma$ need not be free. For $G=\mathbb{Z}_2$ acting on $R^3_2$ by $(y_1,y_2,v)\mapsto(y_2,y_1,-v)$, which preserves the orientation, $\partial^2(R^3_2/G)\cong\mathbb{R}$ (the $v$-axis with the ordering $(1,2)$), and $\sigma$ acts by $v\mapsto-v$. There antisymmetry is a genuine condition, achieved by averaging with rational coefficients.
\end{example}

\section*{Acknowledgments}
The author thanks A. A. Voronov, who suggested this topic in discussions while the author was at the School of Mathematics, University of Minnesota. This work was carried out independently and is not related to the author's current employment.

\section*{Tool Use}
The author uses AI to check the validity of the original proofs and, under the author's direction, to sketch some new proofs. The overall ideas and methods are from the author.

\noindent  Hao Yu, {\it Independent researcher, Beijing, China.}\\
\indent {\it Previously: School of Mathematics, University of Minnesota, Minneapolis, MN 55455, USA.}
\\
{\tt dustless2014@163.com}

\end{document}